\documentclass[11pt]{article}
\usepackage{enumitem}
\usepackage[square,numbers]{natbib}
\usepackage{amsmath,amsthm,amssymb}
\usepackage{graphicx,psfrag,epsf}
\usepackage{epstopdf}
\usepackage{epsfig}
\usepackage{ulem}
\usepackage{subcaption}
\usepackage{bbm}	
\usepackage{color}
\usepackage{mathtools} 
\usepackage{authblk}
\usepackage{tikz}
\usepackage[all,cmtip,color]{xy}
\usepackage{stmaryrd}

\usepackage{algorithm}
\usepackage{algorithmic}
\usepackage{colortbl}
\usepackage{tabulary}

\newcommand{\bmT}{\backslash \mT}
\RequirePackage[colorlinks,citecolor=blue,urlcolor=black]{hyperref}

\theoremstyle{plain}
\newtheorem{lemma}{Lemma}
\newtheorem{theorem}{Theorem}
\newtheorem{ass}{Assumption}
\newtheorem{proposition}{Proposition}
\newtheorem{corollary}{Corollary}

\theoremstyle{definition} 
\newtheorem{definition}{Definition} 
\newtheorem{example}{Example}
\newtheorem{remark}{Remark}
\newcommand{\iid}{\stackrel{iid}{\sim}}

\newcommand{\wt}[1]{\widetilde{#1}}

\newcommand{\bm}[1]{\boldsymbol{#1}}

\newcommand\mC{\mathcal C}

\def\C {\,|\:}

\newcommand\mN{\mathcal N}
\newcommand\mT{\mathcal T}
\newcommand\mG{\mathcal G}

\newcommand\mP{\mathcal P}
\newcommand\mX{\mathcal X}

\newcommand\mE{\mathcal E}
\newcommand\E{\mathbb E}

\newcommand\bG{\Gamma}
\renewcommand\P{\mathbb P}

\newcommand\mU{\mathcal U}
\newcommand\mA{\mathcal A}
\newcommand\mB{\mathcal B}

\newcommand\R{\mathbb R}
\renewcommand\b{\bm{\beta}}
\newcommand\e{\mathrm e}
\newcommand\N{\mathbb N}
\newcommand\bT{\mathbb T}

\newcommand{\Lmax}{L_{max}}

\newcommand{\blind}{0}

\makeatletter
\newcommand{\mylabel}[2]{#2\def\@currentlabel{#2}\label{#1}}
\makeatother

\begin{document}

\def\spacingset#1{\renewcommand{\baselinestretch}%
{#1}\small\normalsize} \spacingset{1}


\if0\blind
{
      \title{\sf
      On Mixing Rates for Bayesian CART
  }

\author{Jungeum Kim\thanks{Jungeum Kim is a postdoctoral researcher at the Booth School of Business of the University of Chicago (\texttt{jungeum.kim@chicagobooth.edu})}\, }
\author{Veronika Ro{\v{c}}kov{\'a}\thanks{Veronika Ro{\v{c}}kov{\'a} is  Professor in Econometrics and Statistics and James S. Kemper Faculty Scholar at the Booth School of Business of the University of Chicago (\texttt{veronika.rockova@chicagobooth.edu}). The author gratefully acknowledges support from the James S. Kemper Foundation Faculty Research Fund at the Booth School of Business as well as the National Science Foundation (DMS:1944740). }}
\affil{Booth School of Business, University of Chicago}

\date{}
  \maketitle
} \fi

\if1\blind
{
  \bigskip
  \bigskip
  \bigskip
  \begin{center}
    {\LARGE\bf Title}
\end{center}
  \medskip
} \fi
\begin{abstract}
The success of Bayesian inference with MCMC depends critically on Markov chains rapidly reaching the posterior distribution. Despite the plentitude of inferential theory for posteriors in Bayesian non-parametrics, convergence properties of MCMC algorithms that simulate from such ideal inferential targets are not thoroughly understood. This work focuses on the Bayesian CART {\em algorithm} which forms a building block of Bayesian Additive Regression Trees (BART). We derive upper bounds on mixing times for typical posteriors under various proposal distributions. Exploiting the wavelet representation of trees, we provide sufficient conditions for Bayesian CART to mix well (polynomially) under certain hierarchical connectivity restrictions on the signal. We also derive a negative result showing that Bayesian CART (based on simple {\em grow} and {\em prune} steps) cannot reach deep isolated signals in faster than exponential mixing time. To remediate myopic tree exploration, we propose Twiggy Bayesian CART which attaches/detaches entire twigs (not just single nodes) in the proposal distribution. We show polynomial mixing of Twiggy Bayesian CART without assuming that the signal is connected on a tree. Going further, we show that informed variants achieve even faster mixing. A thorough simulation study highlights discrepancies between spike-and-slab priors and Bayesian CART under a variety of proposals.
\end{abstract}

\spacingset{1.2}
\section{Introduction}

The advent of Markov Chain Monte Carlo (MCMC) has accelerated the widespread adoption of Bayesian methods in practice.
Bayesian inference via  MCMC simulation, however, depends critically on Markov chains reaching their stationary distribution reasonably fast. 
The folk wisdom is that MCMC is far slower than optimization and is only warranted when uncertainty quantification is desperately needed \citep{ma2019sampling}.
Positive findings have nevertheless been reported where rapid (polynomial)   MCMC mixing times are, in fact, attainable in complex combinatorial problems (such as  Bayesian variable selection \cite{yang2016computational}). 
This paper aims to create similar reasons for optimism (as well as caution) in the context for Bayesian tree-based regression.

Bayesian tree-based regression (Bayesian CART of \citep{denison1998bayesian,chipman1998bayesian} and BART of \citep{chipman2010bart}) is one of the most popular machine learning tools in practice today.  A host of frequentist theory now exists to certify their inferential validity \citep{castillo2021uncertainty,rockova2021ideal, jeong2020art, linero2018bayesian}. While estimation and inferential theory already exists, properties of MCMC approximations to these ideal inferential targets are conspicuously missing.  This paper addresses {\em computational} properties of the Bayesian CART algorithm \citep{denison1998bayesian,chipman1998bayesian} as opposed to statistical properties of the Bayesian CART posterior. We attempt to quantify (with lower and upper bounds) the speed at which practically used MCMC algorithms converge to the ideal inferential targets. Characterizations of MCMC mixing times for Bayesian CART (besides a lower bound in a recent independent paper \cite{ronen2022mixing}) have been unavailable. 

 There is an apparent disconnect between theory for optimization and sampling \citep{ma2019sampling} and between theory for posterior distributions and their MCMC approximations. 
Bayesian CART  implementations \citep{denison1998bayesian,chipman1998bayesian} are instantiations of the Metropolis-Hastings (MH) algorithm \citep{metropolis1953equation}  with local grow and prune proposal steps for addition or deletion of a node. The BART algorithm is essentially a Bayesian back-fitting extension where Bayesian CART is applied to the residuals for each individual tree.
In spite of widespread popularity, difficulties in mixing have been reported \citep{chipman1998bayesian,wu2007bayesian, pratola2016efficient, linero2017review,hill2020bayesian}. Several enhancements have been proposed such as modifications of the proposal \citep{wu2007bayesian, pratola2016efficient},  ``warm start" initializations \citep{he2021stochastic}, or running multiple chains \citep{carnegie2019comment, dorie2019automated}. This work attempts to characterize the computational bottlenecks of Bayesian CART and performs a comparative study of various proposal distributions  in terms of mixing times. 

Our  computational complexity analysis builds on several fundamental papers studying Metropolis procedures \citep{lovasz1993random, frieze1994sampling, mengersen1996rates}. Notably,
 \citep{mengersen1996rates} derive necessary and sufficient conditions for  MH algorithms (with independent or symmetric proposals) to converge at a geometric rate to a prescribed continuous distribution. \citep{belloni2009computational} study computational complexity of MCMC  based on Metropolis random walks  as both the sample and parameter dimensions grow to infinity for non-concave and possibly non-smooth likelihoods. 
 We focus on a spectral bound approach suitable for Markov chains whose states are combinatorial structures.
For finite-state Markov chains, the spectral gap can be bounded in terms of quantities associated with its graph \citep{jerrum1988conductance, diaconis1991geometric, frigessi1993convergence}. Perhaps the first systematic approach to handling spectral bounds was developed in \citep{lawler1988bounds} using the conductance concept due to \citep{cheeger2015lower}.
Conductance  is a measure of edge expansion of the Markov chain, see \citep{lovasz1990mixing}  who proved the connection between conductance and convergence for the continuous state space. 
Lower bounds on the conductance, which give upper mixing bounds, are typically obtained by a technique of canonical paths where the idea is to find a set of paths such that no edge is very heavily congested. 
By using the canonical path argument, \citep{yang2016computational} show the rapid mixing of Bayesian variable selection. This bound is improved in \citep{zhou2021dimension} in the context of informed MCMC (that uses posterior information in the proposal) using the drift condition of \citep{jerison2016drift} rooted in the coupling inequality \citep{pitman1976coupling, lindvall2002lectures}.  We consider locally informed proposals as well and, using similar drift conditions, we conclude linear mixing in $n$.
Our work draws parallels between tree-based regression and structured wavelet shrinkage \citep{castillo2021uncertainty}.
The wavelet representation of dyadic trees  turns the tree selection problem into a variable selection problem with hierarchical constraints. 
The constraint creates certain reachability barriers and  requires more sophisticated movements across the state space and a more careful design of canonical 
paths. In this work, we navigate the complex relationship between the MH proposal distribution and the mixing rate. While \citep{yang2016computational} used the deterministic stepwise selection algorithm as  an inspiration to construct canonical paths, we use the CART algorithm \citep{breiman2017classification} as an inspiration. 

We primarily focus on a one-dimensional setting with dyadic splits (noting that non-dyadic CART can be analyzed in a similar manner as in \cite{castillo2021uncertainty}) where the MH proposal distribution consists of a simple attachment of a terminal node (GROW) or a detachment of two sister bottom nodes (PRUNE) \citep{denison1998bayesian,chipman1998bayesian}.  This algorithm is used in practice, for example in the context of uncertainty quantification for non-parametric regression with spatially varying smoothness \citep{rockova2021ideal}.
Rapid mixing rate bounds in \citep{yang2016computational} and \citep{zhou2021dimension} critically rely on an asymptotic unimodality of the posterior distribution which can be translated in our context as model selection consistency. We first characterize sufficient conditions for tree selection consistency.
Second, we show a negative result (an exponential mixing lower bound) where Bayesian CART fails at reaching deep isolated signals obscured by layers of noise.  This motivates our proposal of Twiggy Bayesian CART, a new MH proposal distribution which attaches and deletes twigs (as opposed to individual nodes) to extend reachability. We show that Twiggy Bayesian CART attains polynomial mixing in non-parametric regression when the truth is a step function. 
To dilute the negative message about Bayesian CART, we show that it, in fact, achieves rapid mixing when the truth consists of wavelet signals that are connected along a tree. This is expected since myopic additions and deletions can reach deep signal through intermediate steps. It is interesting that the upper bound for Bayesian CART is then faster by a factor of $n$ relative to spike-and-slab priors \citep{yang2016computational}. This  may indicate smoothing benefits of  tree-shaped regularization that avoids the addition of spurious high-resolution signals. Finally, using the two-drift condition argument \citep{zhou2021dimension}, we show linear mixing of Markov chains under locally informed proposals.

Recently, independently from our work, \citep{ronen2022mixing} studied  Bayesian CART with PRUNE and GROW movements in a multi-dimensional setting, where  a lower bound that scales exponentially with $n$ is shown exploiting the bottleneck that happens when one splits on a wrong variable early in the tree. Our work differs in several aspects: we exploit the wavelet formulation of trees to show consistency and upper bounds on mixing. The paper \citep{ronen2022mixing} only discusses a lower bound. Our lower  bound is for the univariate case and focuses on the bottleneck that happens when deep signal is surrounded by noise.

{The paper is structured as follows. In Section \ref{sec:BC}, we provide a brief review of the Bayesian CART and establish its tree selection consistency. The Twiggy Bayesian CART and the informed variations are introduced in Section \ref{sec:patulous}. The theoretical framework and analysis of the mixing rates are presented in Section \ref{sec:mixing} and Section \ref{sec:Mixing_Rate}. The numerical study in Section \ref{sec:neumerical} reinforces our theoretical findings on both simulation and real datasets. The paper concludes with Section \ref{sec:conclusion}. }

\section{Bayesian CART}\label{sec:BC}
Regression trees perform structured wavelet shrinkage \citep{donoho1995adapting, castillo2021uncertainty}, where the underlying tree provides a skeleton for signal coefficients.
This regression re-interpretation of Bayesian CART allows for straightforward implementations of the Bayesian CART algorithm
through closed-form tree posterior probabilities.

\subsection{Trees as Wavelets}\label{sec:one_dime_model}
We assume that observed continuous outcomes $Y=(Y_1,\dots,Y_n)'$ arise from 
\begin{equation}\label{model}
Y_i=f_0(x_i)+\varepsilon_i,\quad \varepsilon_i\iid\mathcal N(0,1),\quad i=1,\dots, n=2^{L_{max}+1}
\end{equation}
where $\mX=\{x_i=i/n:1\leq i\leq n\}$ are fixed observations on a regular\footnote{The fixed grid assumption $x_i=i/n$ could be avoided using either unbalanced Haar wavelets \citep{fryzlewicz2007unbalanced} or regularity relaxations \citep{rockova2021ideal}.} grid. We focus on wavelet  reconstructions of  $f_0$ using the standard Haar wavelet basis 
$\psi_{-10}(x)=I_{[0,1]}(x)$ and $\psi_{lk}(x)=2^{l/2}\psi(2^lx-k)$
obtained with orthonormal dilation-translations of  $\psi=I_{(0,1/2]}-I_{(1/2,1]}$.
Denote with $\bm X=(x_{ij})$ the $(n\times p)$ regression matrix of $p=2^{L_{max}}=n/2$ regressors constructed from Haar wavelets $\psi_{lk}$ up to the maximal resolution $\Lmax$, i.e.
\begin{equation}\label{design_mat}
x_{ij}=\begin{cases}
\psi_{-10}(x_i)=1& \quad\text{for}\quad j=1\\
\psi_{lk}(x_i)&\quad \text{for}\quad j=2^l+k+1.
\end{cases}
\end{equation}
We assume that the columns of $\bm X$ have been ordered according to the index $2^l+k$ (increasing ordering).
We denote with $F_0=(f_0(x_1),\dots, f_0(x_n))'$ the vector of realized values of the true regression function at design points. The non-parametric regression model \eqref{model} can be written in a matrix form 
\begin{equation}\label{eq:main_model}
Y=\bm X\b^*+ \bm\nu,\quad\text{where}\quad \bm\nu=F_0-\bm X\b^*+\bm\varepsilon\,\,\,\text{with}\,\,\,\bm\varepsilon\sim\mathcal N(0,I_n),
\end{equation}
where   $\bm \beta^*$ is an ordered vector of wavelet coefficients $\beta_{lk}^*=\langle\psi_{lk},f_0\rangle$.  Bayesian dyadic CART (with splits at dyadic rationals) corresponds to tree-shaped wavelet reconstructions \cite{castillo2021uncertainty}, as we re-iterate in Section \ref{section:bart_prior_only} below.

\begin{definition}(Tree) By a tree $\mT$, we understand a collection of hierarchically organized nodes $(l,k)$ such that
$(l,k)\in\mT\Rightarrow (j,\lfloor k/2^{l-j}\rfloor)\in\mT$  for  $j=0,\dots, l-1.$
We distinguish between two types of nodes: {\em internal} ones $\mT_{int}=\{(l,k)\in\mT: \{(l+1,2k),(l+1,2k+1)\}\in\mT\}\cup (-1,0)$  and {\em external} ones $\mT_{ext}=\mT\backslash\mT_{int}$ which are at the bottom of the tree. We define  a set of {\em pre-terminal} nodes  $\mP(\mT)=\{(l,k)\in\mT_{int}: \{(l+1,2k),(l+1,2k+1)\}\in\mT_{ext}\}$ as those internal nodes whose children are external. The null tree is defined as $\mT_{null}=\{(-1,0)\}$ and the full tree at the level 
$L$ is  defined as $\mT_{full}^L=\{(l,k):l<L\}$.
\end{definition}
\noindent We will often denote with  $\b_\mT=(\beta_{lk}:(l,k)\in\mT_{int})'$ the vector of ordered coefficients {\em inside} the tree\footnote{there are $|\mT_{ext}|$ of those} and with $\b_{\bmT}$ the complement. 
Similarly, for a given tree structure $\mT$ we often split the design matrix $\bm X$ into active covariates $\bm X_{\mT}$ (that correspond to $(l,k)\in\mT_{int}$) and the complementary inactive ones $\bm X_{\bmT}$.

\subsubsection{The Bayesian CART Posterior}\label{section:bart_prior_only}
The distinguishing feature of Bayesian CART, compared to selective wavelet reconstructions such as {\em RiskShrink} of \citep{donoho1994ideal}, is that the pattern of sparsity has a tree structure.
Namely,  for a chosen maximal tree depth $L\leq L_{max}$, we assume the tree-shaped wavelet shrinkage prior \citep{castillo2021uncertainty}
\begin{align}
\mT \qquad & \sim \qquad \Pi(\mT) \label{eq:prior_beta1} \\
\{\beta_{lk}\}_{l< L,k} \C \mT\ & \sim\  \Pi(\b_\mT) \, \otimes 
\bigotimes_{(l,k)\notin\mT_{int}} \delta_0(\beta_{lk}).\label{eq:prior_beta2} 
\end{align} 
Similarly as in \cite{rockova2021ideal}, we consider the unit information
$g$-prior $\b_\mT\sim\mathcal N(0,g_n(\bm X_\mT '\bm X_\mT)^{-1})$ with $g_n=n$ which coincides with the standard Gaussian prior $\Pi(\b_\mT)=\prod_{(l,k)\in\mT_{int}}\phi(\beta_{lk};0,1)$
 in regular designs.

The integral component of Bayesian CART is the tree prior  $\Pi(\mT)$ over a set  $\bT_L$ of all trees up to the maximal chosen depth $L\leq L_{max}$.
The Bayesian CART prior in \citep{chipman1998bayesian} uses the heterogeneous Galton-Watson (GW) process (see  Section 2.1 in \citep{castillo2021uncertainty} and  \citep{rovckova2019theory}) with node split probabilities
\begin{equation}\label{eq:split_prob}
p_{lk}=\mathbb{P}[ (l,k)\in\mT_{int}]
\end{equation}
which need to be small in order to prevent the trees from growing indefinitely. While \citep{chipman1998bayesian} suggest $p_{lk}=\alpha/(1+l)^\gamma$ for some $\alpha\in(0,1)$ and $\gamma>0$, we will assume that $p_{lk}$ decays faster, potentially depending on $n$. Given $p_{lk}$, the tree prior probability for $\mT\in\bT_L$ satisfies
\begin{equation}\label{eq:tree_prior_1dim}
\Pi (\mT) \propto \prod_{(l,k)\in \mT_{int}} p_{lk}\prod_{(l,k)\in\mT_{ext}}(1-p_{lk}).
\end{equation}

The conditional conjugacy of the Gaussian prior yields tractable posterior (up to multiplication) which is useful for Metropolis-Hastings implementations. In particular, for $\Sigma_{\mT}=c_n(\bm X_{\mT}'\bm X_{\mT})^{-1}$ with $c_n=n/(n+1)$
we have $\Pi(\mT\C Y)\propto \Pi(\mT)\times N_Y(\mT)$, where
\begin{equation}
N_Y(\mT)= \frac{\exp\left\{-\frac{1}{2} Y'[I-\bm X_\mT \Sigma_\mT \bm X_\mT']   Y\right\}}{(2\pi)^{n/2}(1+n)^{|\mT_{ext}|/2}}.\label{eq:compare_post}
\end{equation}
The Bayesian CART posterior has many favorable properties, such as near-minimax rate adaptation under the supremum loss  \citep{castillo2021uncertainty} for $\alpha$-H\"{o}lderian functions  with $\alpha\leq 1$. This work focuses on computational (not statistical) properties of Bayesian CART. The mixing rate of the Bayesian CART MCMC algorithm \citep{denison1998bayesian,chipman1998bayesian}, however,  ultimately depends on the structure of the underlying truth $f_0$. For clearer exposition of our findings, we focus on the following two assumptions on $f_0$, which are consonant with the tree (step function) model. 

\begin{ass}\label{ass:f_classic}
Assume that  $f_0$ in \eqref{model} satisfies
$f_0 (x) = \sum_{(l,k)\in \mathcal{B}} \psi_{lk}(x)\beta_{lk}^*$,
for some subset $\mathcal B \subseteq \{(l,k): l< L\}$ such that   $A\log n/\sqrt{n}<|\beta_{lk}^*|<C_{f_0}$ for all $(l,k)\in  \mathcal{B}$ for some $A>0$ and $C_{f_0}>0$. Define $\mT^*\in\bT_{L}$ as the {\em smallest} tree that contains all signal nodes in  $\mathcal B$ as internal nodes.
\begin{itemize}
\item[(a)] Assume that $\mathcal B\in \bT_{L}$, i.e. $\mT^*_{int}=\mathcal B$.
\item[(b)]  Assume that $\mathcal B\notin \bT_{L}$.
\end{itemize}
\end{ass}

\begin{remark}A class of tree-sparse functions compatible with Assumption \ref{ass:f_classic} (a) is discussed in \citep{baraniuk2010model}.
For example, signal discontinuity  gives rise to a chain of large wavelet coefficients connected  in the wavelet tree from the root to a leaf (\citep{baraniuk2010model}, Figure 2). The connected signal property has been leveraged in a myriad of wavelet-based processing and compression algorithms \citep{ shapiro1993embedded, crouse1998wavelet}. 
 Assumption \ref{ass:f_classic} (a) is  intentionally optimistic in the sense that Bayesian CART {\em is expected} do well on a tree-shaped truth compared to, for example, spike-and-slab priors that do not have structured regularization. We will see this superiority in both our numerical as well as theoretical study.
\end{remark}

\begin{remark}
Unlike previous investigations of Bayesian CART \cite{rockova2021ideal,castillo2021uncertainty}, we do not assume H\"{o}lderian $f_0$ which alone does not guarantee tree selection consistency.
Our results can be however replicated for {\em structured} H\"{o}lderian signals  under suitable signal gap assumption for coefficients inside and outside $\mT^*$. 
\end{remark}

An essential first step towards obtaining upper bounds on  Markov chain mixing times  is tree selection consistency.  The following Theorem shows that  under Assumption \ref{ass:f_classic} the posterior concentrates on $\mT^*$, the minimal  tree spanning over signal.  
Similar consistency requirements  (or log-concavity and asymptotic normality assumptions) have been required to obtain 
rapid convergence rate  statements for  Markov chains \citep{applegate1991sampling,lovasz2004hit, yang2016computational, zhou2021dimension}.
While our theory has been derived for the regular fixed design,  similar theoretical conclusions can be obtained also for fixed irregular design as in \cite{rockova2021ideal} using the unit information $g$-prior. 

\begin{theorem}(Tree Selection Consistency)\label{lemma:consist}
Assume the model \eqref{model}, the Bayesian CART prior from Section \ref{section:bart_prior_only} with $p_{lk}=n^{-c}$ for $ c>5/2$. Under Assumption \ref{ass:f_classic} for large enough $A>0$
we have  with probability at least $1-4/n$ 
$$
\Pi(\mT^*\C Y)\geq 1-{ \frac{1}{n^{c-5/2}-1}}-\frac{1}{n^{A^2/8\log n}}.
$$
\end{theorem}
\proof Section \ref{sec:proof_lemma_consist}.
\begin{remark}In the context of Bayesian inference with phylogenetic  trees, \citep{mossel2005phylogenetic}   show that when the data are generated by a mixture of two trees, many of the popular Markov chain take exponentially long to reach stationarity. Lemma \ref{lemma:consist} focuses on the less adverse situations when a single generative model is present that can be identified by the posterior.
\end{remark}

\begin{remark}
The consistency result in Theorem \ref{lemma:consist} is different from posterior concentration rate results in \cite{castillo2021uncertainty} and \cite{rockova2021ideal} for  H\"{o}lderian functions $f_0$
under the supremum loss. Due to the step function Assumption \ref{ass:f_classic}, we require a more aggressive split probability $p_{lk}=n^{-c}$ in Lemma \ref{lemma:consist} because we cannot leverage the decaying property of wavelet coefficients. 
\end{remark}

Much of the value of the optimality properties of the Bayesian CART posterior  (e.g. adaptation to local smoothness \cite{rockova2021ideal} and frequentist validity of inference about certain $f_0$ \cite{castillo2021uncertainty})
hinges on the ability to approximate this posterior well. 

\subsubsection{The Bayesian  CART Algorithm}\label{section:sampler}
The Bayesian CART algorithm is devised to explore the space of regression tree topologies by sequential sampling from the tree posterior distribution determined by \eqref{eq:compare_post}.
The two original algorithms \citep{denison1998bayesian,chipman1998bayesian} are based on Metropolis-Hastings ideas 
with an accept-reject proposal mechanism consisting of four basic proposal moves (add a node, delete a node, change a variable and change a split-point). Many variations were later proposed with more intricate moves, such as tree rotations \citep{gramacy2008bayesian, pratola2016efficient}, to better explore the tree space. 


The Bayesian CART algorithm generates a chain of trees $\mT^0,\mT^1,\dots$ which will gravitate toward regions charged with posterior probability.
Starting with an initial tree $\mT^0$,  transitions from $\mT^i$ to $\mT^{i+1}$ proceed in two steps:
(1) generate a candidate value $\wt\mT$  from a proposal distribution  $S(\mT^i\rightarrow\wt\mT)$ and 
(2) accept the proposal (i.e. $\mT^{i+1}=\wt\mT$) with a probability
\begin{equation}\label{eq:mh_accept}
\alpha(\mT^i,\wt\mT)=\min\left\{1, \frac{\Pi(\wt\mT\C Y)S(\wt\mT\rightarrow\mT^i)}{\Pi(\mT^i\C Y)S(\mT^i\rightarrow\wt\mT)}\right\}
\end{equation}
and set   $\mT^{i+1}=\mT^i$ otherwise.

Under weak conditions (Section 7.4 of \citep{robert1999monte}), the sequence obtained by this algorithm will be an irreducible and aperiodic Markov chain with a limiting distribution $\Pi(\mT\C Y)$. 
Below, we will describe a {\em dyadic} one-dimensional version of Bayesian CART \citep{chipman1998bayesian} which deploys a kernel  $S(\mT^i\rightarrow \wt\mT)$ that generates $\wt\mT$ from $\mT^i$ by randomly choosing among two steps (GROW and PRUNE). The algorithmic description of dyadic Bayesian CART we study is in Algorithm \ref{alg:original}. We describe the algorithm using our wavelet tree representation.

The GROW movement  chooses (uniformly at random) one terminal node, say $(\wt l, \wt k)$, and splits it. In particular, we have
$\wt \mT_{int}=\mT_{int}^i\cup \{(\wt l, \wt k)\}$ and  $\wt \mT_{ext}=\mT_{ext}^i\cup \{(\wt l+1, 2\wt k),(\wt l+1, 2\wt k+1) \}  \backslash \{(\wt l, \wt k)\}$ and 
\begin{equation} 
S_{GROW}(\mT^i\rightarrow\wt\mT)=\frac{1}{|\mT^i_{ext}|}\label{eq:grow}.
\end{equation}
The PRUNE movement reverses the GROW move by  choosing (uniformly at random) one pre-terminal node,   $(\wt l, \wt k)\in\mathcal P(\mT^i)$,  and by turning it into a terminal node.
In particular, we have $\wt \mT_{int}=\mT_{int}^i\backslash \{(\wt l, \wt k)\}$ and $\wt \mT_{ext}=\mT_{ext}^i{ \cup\{(\wt l, \wt k)\}}\backslash \{(\wt l+1, 2\wt k),(\wt l+1, 2\wt k+1) \}$
and
\begin{equation}
S_{PRUNE}(\mT^i\rightarrow\wt\mT)=\frac{1}{|\mP(\mT^i)|}.\label{eq:prune}
\end{equation} 
\smallskip
Combining the two moves, dyadic Bayesian CART has the following proposal distribution
\begin{equation}\label{eq:proposal}
S(\mT\rightarrow \wt\mT)=
 \Gamma(\mT) \times S_{GROW}(\mT\rightarrow \wt\mT)+  [1- \Gamma(\mT)]\times S_{PRUNE}(\mT\rightarrow \wt\mT), 
\end{equation}
where $\Gamma(\mT)$ is the grow binary indicator with $P[\Gamma(\mT)=1]=1/2$ for $\mT\notin\{ \mT_{null}, \mT_{full}^L\}$ and  $P[\Gamma(\mT_{null})]=1- P[\Gamma(\mT_{full}^L)]=1$.
The dyadic Bayesian CART algorithm was successfully deployed for estimating functions with spatially varying smoothness and for the construction of valid confidence sets \cite{rockova2021ideal}. Despite a simplified version of the full-blown Bayesian CART,   this toy algorithm will give us many useful insights  about computational bottlenecks. 
The GROW/PRUNE transition kernel performs only very local moves, not allowing bushy trees to be substantially restructured. This property makes this generic sampler susceptible to myopic encasement   
if initialized far away from high-posterior regions. While its poor mixing has been widely recognized in empirical studies 
 \citep{chipman1998bayesian,wu2007bayesian, pratola2016efficient, linero2017review,hill2020bayesian}, limited theoretical studies of the mixing times have been available  \citep{ronen2022mixing}.

\begin{figure}[t]
\includegraphics[width=1\linewidth]{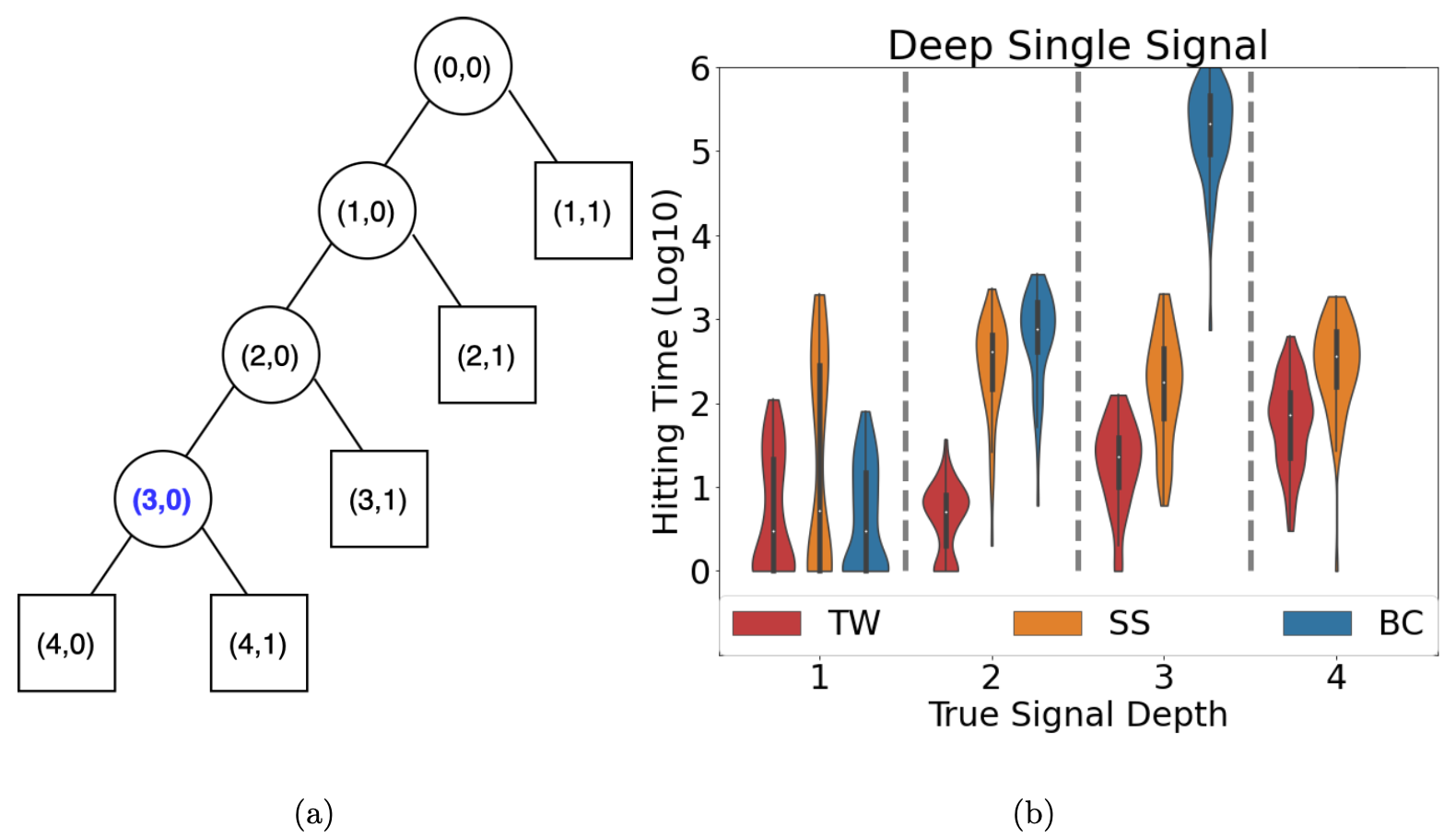}
\caption{ (a) An example of the minimal spanning tree $\mT^*$ from Example \ref{ex:caution} with oval internal and square external nodes. The signal node  is marked in blue. (b) The hitting time of 50 chains initialized at $(0,0)$. TW: Twiggy Bayesian CART, SS: Spike-and-Slab, BC: Bayesian CART. The deeper the signal is, the slower (exponentially) the hitting time of Bayesian CART. We investigate this phenomenon theoretically in Section \ref{sec:poor}.}\label{fig:motiv}
\end{figure}

\section{Bayesian CART with a Twist}\label{sec:patulous}
Local Metropolis-Hastings  proposals are known to induce poor mixing \citep{wu2007bayesian,pratola2016efficient} which may result in misleading  under-representations of uncertainty.  
In the context of trees, 
\citep{gramacy2008bayesian} remediate this issue by applying a rotation algorithm   \citep{sleator1988rotation} while \citep{pratola2016efficient} proposes various elaborate moves for radical restructuring (see also \citep{wu2007bayesian}). Another way to prevent single trees from getting stuck is by adding them up  and by performing Bayesian back-fitting (see the BART method of \citep{chipman2010bart}). Alternatives to MH samplers have also been recently explored, see \citep{lakshminarayanan2013top} for  Sequential Monte Carlo approach and  \citep{lakshminarayanan2015particle} for a particle Gibbs algorithm. We focus on the original Bayesian CART (dyadic version).
One source of mixing issues for Bayesian CART is illustrated in a cautionary tale example below.

\begin{example}{(The Pitfalls of  Bayesian CART)} \label{ex:caution}
Consider $f_0:[0,1]\rightarrow\R$  which satisfies Assumption \ref{ass:f_classic} (b) where $\mathcal B=\{(j,0)\}$ and  $\beta_{j,0}^*=2$. We also assume $n=2^{\Lmax+1}$ with  $\Lmax=8$. 
We consider the cases where the true signal depth grows, i.e.  $j\in\{1,2,3,4\}.$ We found that once the chain hits the signal node, it tends to stay around the minimal spanning tree $\mT^*$ (plotted in Figure \ref{fig:motiv}(a) for $j=3$). Therefore, as a proxy to mixing time, we measure the hitting time defined as $\tau=\min_{t\geq 0}\{ \mathcal B\subset \mT_{int}^t\}$. We run 50 chains for three algorithms: Bayesian CART, Twiggy Bayesian CART (to be introduced later), and Spike-and-Slab (one-site Metropolis-Hastings), where for all methods we use $p_{lk}=0.1$. All chains are initialized at the root node $(0,0)$. The violin plots of the  hitting times are in Figure \ref{fig:motiv} (b). We see how the hitting time of Bayesian CART slows down exponentially as the depth of the signal increases. When the signal depth is 4, none of the 50 Bayesian CART chains hit within 1,000,000 iterations. In conclusion, Bayesian CART may not be able to capture signal if there are layers of noise separating the initialization and the signal. We will prove this theoretically later in Section \ref{sec:poor}.  On the other hand, as Spike-and-Slab does not have a tree structure, its performance is consistent across different signal depth levels. 
\end{example}

Example \ref{ex:caution} may be unnecessarily pessimistic for Bayesian CART. The following example demonstrates that Bayesian CART actually mixes well when the signal is connected on a tree.

\begin{example}{(The Benefits of Bayesian CART)}\label{ex:caution2}
In contrast with Example \ref{ex:caution}, we now  consider Assumption \ref{ass:f_classic} (a) where $\mathcal B = \mT_{full}^j$ for $j\in\{1,2,3,4\}$ (plotted in Figure \ref{fig:motiv2}(a) for $j=3$). We consider the same simulation settings as in Example \ref{ex:caution}. The violin plots of  hitting times (for the entire set $\mathcal B$) in Figure \ref{fig:motiv2} (b) show superiority of tree-shaped regularization where spike-and-slab takes longer to hit the entire group of connected signals. The stable increase of the hitting time of Bayesian CART is in sheer contrast with the exponential slowdown in Figure \ref{fig:motiv} (b). We investigate mixing of Bayesian CART theoretically for situations like this one in Section \ref{sec:BCART_ass1}.

\begin{figure}
\includegraphics[width=1\linewidth]{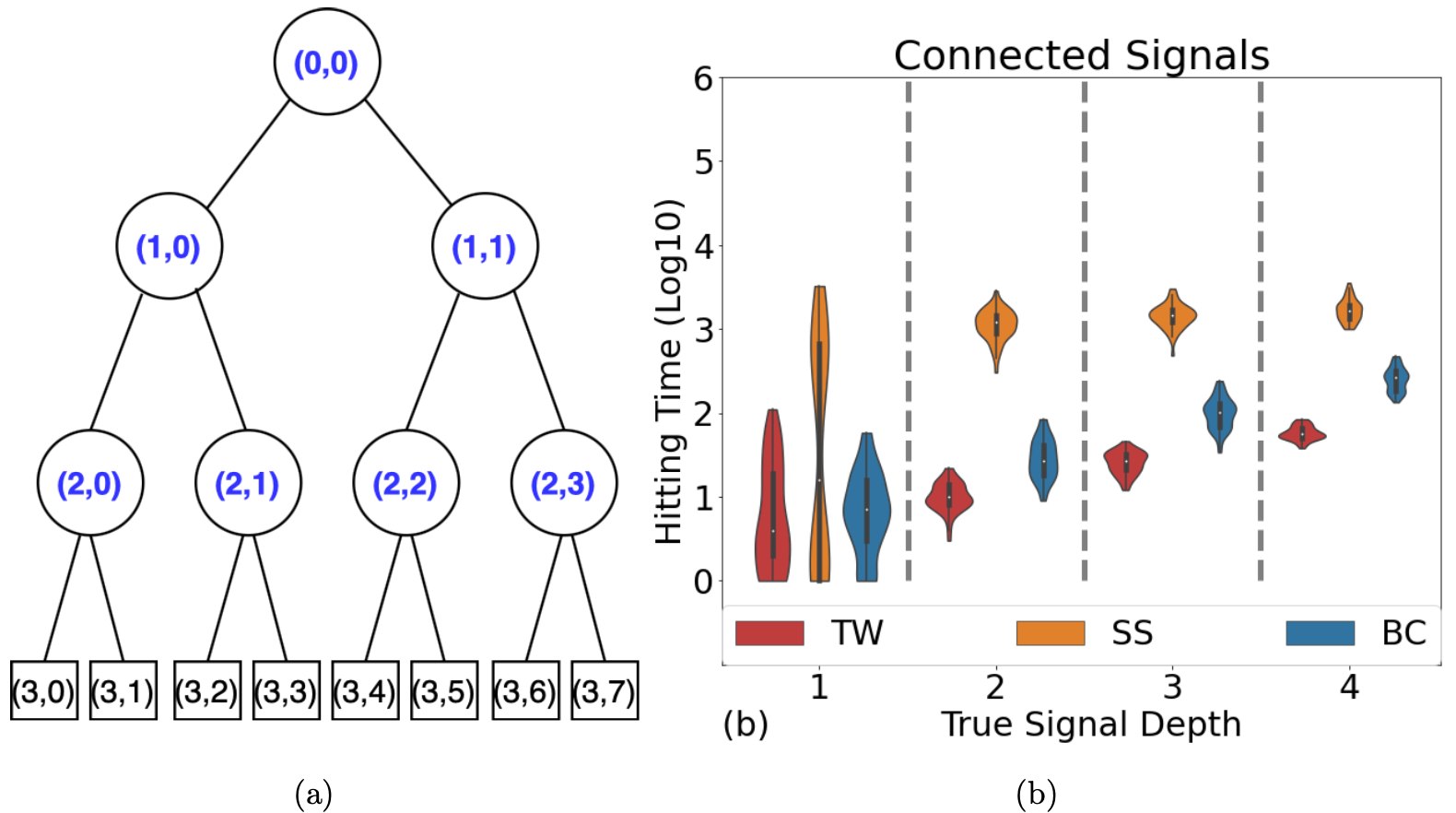}
\caption{ (a) An example of the minimal spanning tree $\mT^*$ from Example \ref{ex:caution2} with oval internal and square external nodes. (b) The hitting time of 50 chains initialized at $(0,0)$. TW: Twiggy Bayesian CART, SS: Spike-and-Slab, BC: Bayesian CART.}\label{fig:motiv2}
\end{figure}
\end{example}

This work is not necessarily aimed at establishing the new methodological gold standard for MH tree proposal distributions. Instead, it is aimed at formalizing computational bottlenecks of Bayesian CART by performing a   theoretical study of  the default approach used in practice.   During our theoretical investigation, however, one natural modification of the classical Bayesian CART resurfaced. 
In  Figure \ref{fig:motiv}(b), we showed a variant of Bayesian CART (called Twiggy Bayesian CART) which had more favorable hitting times. We now describe this new twist on an old classic. Later in Section \ref{sec:lit} we describe another enhancement using locally informed proposals.

\subsection{Twiggy Bayesian CART}\label{sec:twiggy}

To extend the reachability of trees in situations such as Example \ref{ex:caution}, we modify the   GROW and PRUNE proposals. 
The GROW proposal attaches a twig to a chosen terminal node (rather than just splitting it into two nodes). 
The reverse move is then removing an entire branch (twig) in a tree rather than just collapsing two sibling bottom nodes. We call this variant  Twiggy Bayesian CART.
{A twig is a portion of a tree that has at most one internal node for each level (as formalized below).

\begin{definition}(Twig) 
For two distinct nodes $(l,k)$ and $(l',k')$  where $(l,k)$ is an ancestor of $(l',k')$ and  $l\leq l'< L$,  we define a twig  $[(l,k)\leftrightarrow (l',k')]=\{(l,k),...(l'-1,\lfloor k'/2\rfloor),(l',k')\}$ 
as the collection of nodes on the unique shortest path connecting  $(l,k)$ and $(l',k')$ in a full tree $\mT_{full}^L$. A twig of length one is simply
$[(l,k)\leftrightarrow (l,k)]=\{(l,k)\}$. 
\end{definition}

\begin{definition}(Ancestors and Descendants) 
Given $\mT\in\bT_L$ and an internal node $(l,k)\in\mT_{int}$, we define ancestors of $(l,k)$ inside $\mT$ as 
$A_{lk}(\mT)=\{(l',k')\in\mT_{int}:\exists j\in \{0,1,\dots, L-1\}\,\,s.t.\,\, (l',k')=(l-j,\lfloor k/2^j\rfloor) \}$. Descendants of $(l,k)$ inside $\mT$ are defined as
$D_{lk}(\mT)=\{(l',k')\in\mT_{int}: (l,k)\in A_{l'k'}(\mT)\}$.
\end{definition}

Given the current state $\mT^i$, the GROW proposal of Twiggy Bayesian CART  picks a node $(l^*,k^*)$ in $\mT_{int}^{L\,full}\backslash {\mT_{int}^i}$
and grows a twig from an external  node $(\wt l,\wt k)\in\mT^i_{ext}$ that is closest  to $(l^*,k^*)$.
In particular,  we have 
$$
{\wt\mT}_{int}=\mT_{int}^i\cup [(\wt l,\wt k)\leftrightarrow (l^*,k^*)]. 
$$

If we considered a uniform proposal for $(l^*,k^*)$, we would often pick a node from the deepest allowed layer $L-1$ (because there are $2^{L-1}$ of nodes).
Instead, we penalize the inclusion of deep nodes by  first picking a layer $l^*$ from eligible layers $E^i=\{l{  <L}: \exists (l,k)\notin \mT^i_{int}\}$ with probabilities  $d_{l^*}=D^{-l^*}/\sum_{l\in E_l} D^{-l}$ for $D>1$ and by considering a uniform proposal within the chosen layer, i.e.  
\begin{equation}\label{eq:reweight_tw}
 \left(\frac{D-1}{D^{L}-1}\right)\frac{1}{2^{L-1}}\leq S_{GROW}(\mT^i\rightarrow\wt\mT)= \frac{d_{l^*}}{ |\mathcal{K}_{l^*}|} ,
\end{equation}
where $\mathcal{K}_{l^*}=\{k: (l^*,k)\notin \mT^i_{int}\}$. Note that larger $D$ provides a stronger penalty that prevents the proposal from stretching towards nodes that are too deep.


\begin{figure}
    \centering
    \includegraphics[width=.8\linewidth]{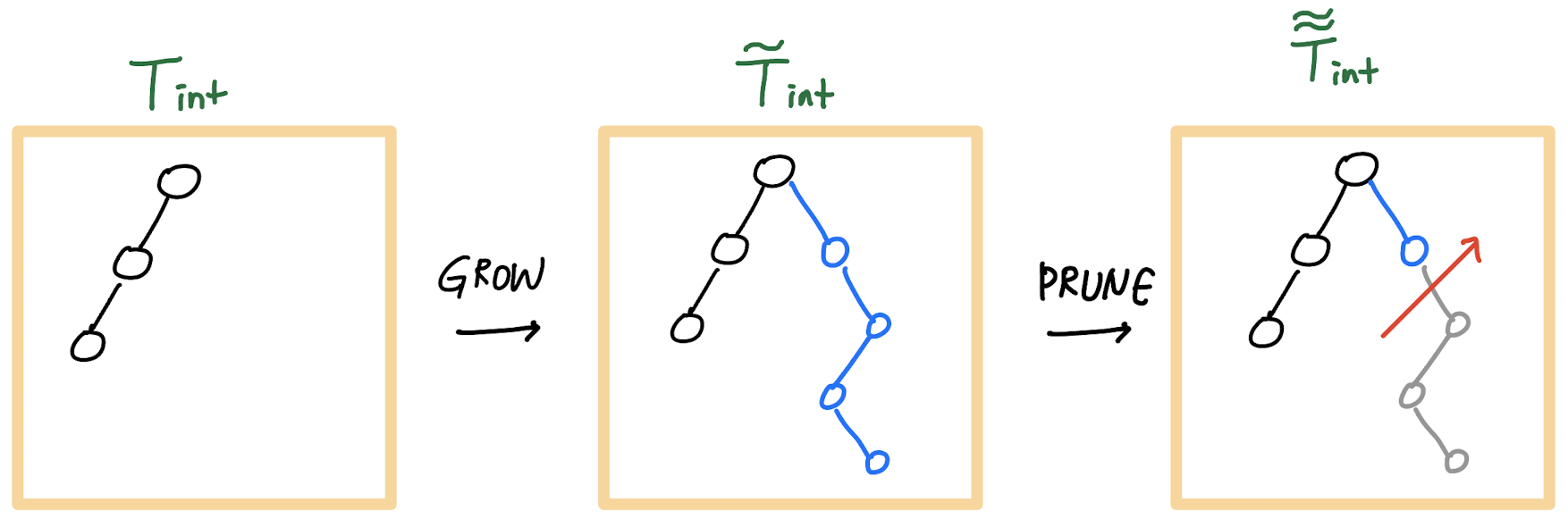}
    \caption{The Twiggy GROW and PRUNE.}
    \label{fig:twigg_example}
\end{figure}

The PRUNE step uniformly picks an internal node $(\wt l,\wt k)\in\mT_{int}^i$ such that its entire branch below is a twig, i.e. 
$D_{lk}=[(\wt l, \wt k) \leftrightarrow (l^*,k^*)]$ for some $(l^*,k^*)\in\mT_{int}^i$. The entire branch below the node is then removed to obtain $\wt\mT$. In particular,  we have 
$$
{\wt\mT}_{int}=\mT_{int}^i\backslash [(\wt l,\wt k)\leftrightarrow (l^*,k^*)] 
$$
 Since the proposal candidates are {all the nodes that have a twig below (including all per-terminal nodes $\mP(\mT^i)$}), the proposal probability is bounded by
\begin{equation}\label{eq:tw_prune}
\frac{1}{|\mT^i_{int}|} \leq S_{PRUNE}(\mT^i\rightarrow\wt\mT)\leq \frac{1}{|\mP(\mT^i)|}.
\end{equation}
{A cartoon of the twig proposals 
is depicted in Figure \ref{fig:twigg_example}.} It can be easily verified that the Twiggy Bayesian CART also yields an irreducible Markov Chain (i.e. all states communicate (see Section 6.3.1 in \citep{robert1999monte}). However, due to denser connectivity among trees we expect fewer bottlenecks.

\subsection {Locally Informed Bayesian CART}\label{sec:lit}
The proposal distribution $S(\cdot \rightarrow \cdot)$  for Bayesian CART and Twiggy Bayesian CART
 ignores posterior information which might be useful in guiding the chain towards high-posterior zones. 
To accelerate  MCMC over general discrete state spaces, \cite{zanella2020informed} proposed \emph{locally informed} proposal schemes  that leverage posterior information in the vicinity of the current state $\mT^i$ to propose the next state $\wt\mT$. In particular, the proposal assigns a weight to each neighboring state $\mT$ that depends on the posterior ratio $\Pi(\mT\C Y)/\Pi(\mT^i\C Y)$. Intuitively, we may expect that a large-posterior candidate is more likely to be accepted. Interestingly, \cite{zhou2021dimension} point  out that this expectation is not always met  and, as a remedy,    threshold  the posterior ratio in the proposal probability calculation. This approach is called LIT-MH (Metropolis-Hastings with Locally Informed and Thresholded proposal distributions). In the context of Bayesian variable selection, \cite{zhou2021dimension} show that LIT-MH significantly improves the mixing rate. Inspired by this finding, we also consider LIT-MH variants for Bayesian CART and Twiggy Bayesian CART and, later in Section \ref{sec:patul_mix_result}, show that their mixing rate is linear in  problem parameters.

Denote by $\mN_g(\mT^i)=\{\mT'\supset \mT^i: |\mT'_{int}\backslash\mT^i_{int}|=1\}$ and $\mN_p(\mT^i)=\{\mT'\subset \mT^i:|\mT^i_{int}\backslash\mT'_{int}|=1\}$ the GROW and PRUNE candidates from the current state $\mT^i$ of the Bayesian CART algorithm. For these neighbor candidate trees, we define an intelligent movement rule instead of just a random walk. The proposal distribution for the LIT-MH proposal for  Bayesian CART consists of
\begin{align}
S_{GROW}(\mT^i\rightarrow \wt\mT)=\frac{w_g(\wt\mT\C\mT^i)}{Z_g(\mT^i)}\mathbb{I}_{\mN_g(\mT^i)}(\wt\mT), \nonumber \\S_{PRUNE}(\mT^i\rightarrow \wt\mT)=\frac{w_p(\wt\mT\C\mT^i)}{Z_p(\mT^i)}\mathbb{I}_{\mN_p(\mT^i)}(\wt\mT),\label{eq:K_def}\end{align}
 where the weighting functions are defined for suitable $A>0$ and $c>3/2$ as
\begin{align}
w_g(\wt\mT\C\mT)&=\frac{\Pi(\wt\mT\C Y)}{\Pi(\mT\C Y)}\wedge n^{(A^2\log n )/8}\quad\text{and}\quad w_p(\wt\mT\C\mT)&=1\vee\frac{\Pi(\wt\mT\C Y)}{\Pi(\mT\C Y)}\wedge n^{c-3/2}\label{eq:wp}, 
\end{align}
and the corresponding normalizing constants are 
\begin{align*}
Z_g(\mT)=\sum_{\wt\mT\in\mN_g(\mT)}w_g(\wt\mT\C\mT)\quad\text{and}\quad
Z_p(\mT)=\sum_{\wt\mT\in\mN_p(\mT)}w_p(\wt\mT\C\mT).
\end{align*}
We call by \emph{Informed} (Twiggy) Bayesian CART the variant with proposal probabilities in \eqref{eq:K_def}. The Informed Twiggy Bayesian CART has more neighbors $\mN_g(\mT)$ and $\mN_p(\mT)$ compared to the Informed Bayesian CART.



\section{On Mixing Rates of Markov Chains}\label{sec:mixing}
This section revisits several known facts about Markov chains whose states are combinatorial structures.
In our work, bounds on mixing rates will be obtained by inspecting the eigenspectrum of the transition matrix. We denote with $ P$ the transition matrix on the state space $\bT_{L}$ whose entries  $ P(\mT_i,\mT_j)$ quantify the probability of the move $\mT_i\rightarrow\mT_j$.
For a given $\mT\in\bT_{L}$, we denote with $\mN(\mT)=\{{\mT'}: S(\mT\rightarrow{\mT'})\neq 0 \}$ the {\em neighborhood} of $\mT$ consisting of all trees ${\mT'}$ which can reach  $\mT$ in one step.  Under the MH algorithm, we can write
$$
 P(\mT,{\mT'})=
\begin{cases}
S(\mT\rightarrow {\mT'})\alpha(\mT,{\mT'}) &\quad \text{if}\quad {\mT'}\in\mN(\mT)\\
0&\quad\text{if}\quad {\mT'}\notin\mN(\mT)\cup\{\mT\},\\
1-\sum_{\wt \mT\neq\mT} P(\mT,\wt\mT)&\quad\text{if}\quad {\mT'}=\mT.
\end{cases}
$$
where $\alpha(\cdot,\cdot)$ is the MH acceptance probability and $S(\cdot\rightarrow\cdot)$ is the proposal probability in \eqref{eq:proposal}.
Moreover, the chain is reversible with respect to the probability distribution $\Pi(\mT\C Y)$ as it satisfies the detailed balance condition
\begin{equation*}
Q(\mT,{\mT'})\equiv \Pi(\mT\C Y) P(\mT,{\mT'})=\Pi({\mT'}\C Y) P({\mT'},\mT)\quad\text{for all}\quad \mT,{\mT'}\in\bT_{L}.
\end{equation*}

This condition ensures that $\Pi(\cdot\C Y)$ is the stationary distribution for $ P$.
It will be useful to associate the Markov chain with a weighted undirected graph on the vertex set $\bT_{L}$ where the weight between two connecting (neighboring) vertices 
$\mT$ and ${\mT'}$ equals $Q(\mT,{\mT'})$.
We denote such a weighted undirected graph by  $G$. Recall that  two vertices $\mT$ and ${\mT'}$ are connected if and only if 
$Q(\mT,{\mT'})>0.$ For an initial state $\mT$ of the Markov chain at time $t=0$, the total variation distance to the stationary distribution after $t$ iterations satisfies
\begin{equation}\label{eq:delta}
\Delta_\mT(t)=\|P^t(\mT,\cdot)-\Pi[\cdot\C Y]\|_{TV}\equiv\max_{S\subset\bT_{L}}|P^t(\mT,S)-\Pi[S\C Y]|,
\end{equation}
where $P^t(\mT,S)\equiv \sum_{{\mT'}\in S}P^t(\mT,{\mT'})$ and where $P^t(\mT,\cdot)$ denotes the distribution of the state at time $t$ with an initial condition $\mT$.
We now recall  the formal definition of a mixing time.

\begin{definition}\label{eq:mixing_time}
The $\epsilon$-mixing time of the Markov chain  is defined as 
\begin{equation}\label{eq:tau_def}
\tau_{\epsilon}\equiv\max_{\mT\in \bT_{L}}\min\{t\in\N: \Delta_\mT(t')\leq \epsilon\quad\text{for all}\quad t'\geq t\},
\end{equation}
where $ \Delta_\mT(t)$ is as in \eqref{eq:delta}.
\end{definition}
For an ergodic chain (whose states are aperiodic and positively recurrent),  the rate of convergence to $\Pi(\cdot\C Y)$ is governed by the spectral gap of $P$.
Defining $\lambda_{max}=\max\{\lambda_1,|\lambda_{|\bT_{L}|-1}|\}$, the spectral gap is defined as $Gap( P)=1-\lambda_{max}$.
The following sandwich relation shows that the mixing time $\tau_{\epsilon}$ and the spectral gap are related (\citep{sinclair1992improved}, equation 2.9 in \citep{woodard2013convergence})
\begin{equation}\label{eq:sandwich}
\frac{1-Gap(P)}{2\times Gap(P)}\log\left[\frac{1}{2\epsilon}\right]\leq \tau_{\epsilon}\leq\frac{\log[1/\min_{\mT\in\bT_{L}}\Pi(\mT\C Y)]+\log 1/\epsilon}{Gap(P)}.
\end{equation}

For our theoretical study, we will work with a modified transition matrix (as suggested in \citep{sinclair1992improved}) which adds self-loops of weight $1/2$ to each state.
This so called ``lazy" Markov chain   does not significantly affect the mixing times.
We denote with $\wt P$ the transition matrix of the original sampler and with $P\equiv \wt P/2+I/2$ the modified matrix. This modification ensures that all eigenvalues are non-negative where the spectral gap satisfies $Gap(P)=1-\lambda_1$. Beyond the connection  in \eqref{eq:sandwich},  the second eigenvalue $\lambda_1$ (or the spectral gap)  controls the information flow through the graph or, in other words,  the {\em conductance} of the Markov chain. 

\subsection{Canonical Paths and Conductance}
Some of the earliest spectral gap lower bounds were based on the concept of conductance \citep{jerrum1988conductance}. In particular, Theorem 2 in \citep{sinclair1992improved}  shows that in a reversible Markov chain
$$
\Phi^2/2\leq Gap(P)\leq 2\Phi,\quad\text{where}\quad\Phi=\min\limits_{\substack{A\subset\bT \\ 0<\Pi[A\C Y]\leq 1/2}}\frac{\sum_{\mT\in A,{\mT'}\in \bT\backslash A}\Pi(\mT\C Y)P(\mT,{\mT'})}{\Pi[A\C Y]}
$$
is the conductance which measures the ability of the chain to escape from any small region of the state space (and make rapid progress to equilibrium). 
The idea behind conductance is that chains with fewer bottlenecks will mix faster.  While conductance can sometimes be estimated directly, in many applications the better approach to upper-bound the spectral gap is with edge overload on {\em canonical paths} \citep{sinclair1992improved}.
\begin{definition}(Canonical Path Ensemble)
For any distinct  pair of trees $\mT,{\mT'}\in\bT_{L}$ we denote with $T_{\mT,{\mT'}}$ a simple path running from $\mT$ to ${\mT'}$ through adjacent states in the state space graph $G$.
A canonical path ensemble  $\mE=\{T_{\mT,{\mT'}}:(\mT,{\mT'})\in\bT_{L}\times\bT_{L}\}$ is then a collection of such simple paths, one for each (ordered) pair of distinct vertices in $G$.
\end{definition}
For any reversible Markov chain and any choice of a canonical path ensemble $\mE$, the spectral gap of $P$ can be lower-bounded with  (Corollary 6 of \citep{sinclair1992improved})
\begin{equation}\label{eq:gap_lb}
Gap(P)\geq \frac{1}{l(\mE)\rho(\mE)},
\end{equation}
where $l(\mE)$ is the length of the longest path in $\mE$ and
\begin{equation}\label{eq:congestion}
\rho(\mE)=\max_{e\in\mE}\frac{1}{Q(e)}\sum_{(\mT,{\mT'}):e\in T_{\mT,{\mT'}}}\Pi(\mT\C Y)\Pi({\mT'}\C Y)
\end{equation}
is the path congestion parameter. For the edge $e$ in between two adjacent states $\mT$ and ${\mT'}$, the quantity 
$Q(e)\equiv Q(\mT,{\mT'})=\Pi(\mT\C Y)P(\mT,{\mT'})$
 measures the natural capacity of the edge $e$ or, in other words, how much traffic it would normally experience in the stationary state. The sum  in \eqref{eq:congestion} then counts the flow of the edge in the given family of canonical paths. The congestion is the maximum load of any edge of the state space graph as a fraction of its capacity.
In order to find an upper bound on the mixing time using \eqref{eq:sandwich}, in Section \ref{sec:canonical} we construct a canonical path ensemble and find a lower bound on the conductance \eqref{eq:congestion}.

\section{Mixing Rates for Bayesian CART}\label{sec:Mixing_Rate}
This section presents some positive as well as negative findings for Bayesian CART in the context of Assumption \ref{ass:f_classic}.
The signal assumptions (a) and (b) are qualitatively rather different and we will be able to appreciate the importance of less myopic proposals in the structure-less signal (b). Without the tree skeleton, local moves of Bayesian CART may not be able to reach all signals.

\subsection{Bayesian CART Can Mix Poorly}\label{sec:poor}

We continue our cautionary tale from Example \ref{ex:caution} showing that isolated signals are out of reach for initializations which need grow through noise to catch them.  We now characterize the inability of the Markov chain to reach 
the posterior distribution reasonably (polynomially) fast. By finding an upper bound for the spectral gap in a
counterexample $f_0$ constructed according to the Example \ref{ex:caution}}, we  show that the mixing lower bound increases exponentially in $n$.

\begin{theorem}\label{theo:poor} Assume the model \eqref{model} with the Bayesian CART prior from Section \ref{section:bart_prior_only} with $L\geq 2$ and $p_{lk}=n^{-c}$ with $ c> 5/2$. There exists $f_0$ that satisfies Assumption \ref{ass:f_classic} (b) such that, with probability at least $1-4/n$, the Bayesian CART mixing time  satisfies for some $C>1$
\[
\tau_{\epsilon}> \log\left(\frac{1}{2\epsilon}\right)\frac{1}{4}\left[\left(\frac{n^{(c-3/2)}/4-1}{C}\right)^{L-2} -3\right].
\]
{This bound is exponential in $n$ when $L=L_{max}\sim \log (n/2)$.}
\end{theorem}
\proof {Section \ref{pf:poor_lower}}. 


\begin{remark}
In work independent  from ours, \citep{ronen2022mixing}  provided a lower bound result showing that a simplified version of BART \citep{chipman2010bart} mixes poorly (at least exponentially in $n$). In particular, \citep{ronen2022mixing} considered a single tree with prune and grow movements in  a multi-dimensional setting.  The key idea behind the lower bound is that the first split direction causes a serious bottleneck. To move between two trees that differ in their first split direction, one must prune all the way up to the root tree to replace the first split. We consider a perhaps more simplified scenario by exploiting wavelet representations and we show slow mixing even in a one-dimensional setting.
\end{remark}

\subsection{Bayesian CART Can Mix Well}\label{sec:BCART_ass1}
We now establish sufficient conditions for classical Bayesian CART to mix ``well", i.e the number of iterations required to converge to an $\epsilon$-ball of the stationary distribution grows only {\em polynomially} in the problem parameters. We will inspect various components of the sandwich relation presented earlier in \eqref{eq:sandwich}.

The following theorem provides a polynomial upper bound for the speed of MCMC convergence of classical Bayesian CART for connected signals (Assumption \ref{ass:f_classic} (a)).

\begin{theorem}\label{eq:theo1}
Assume the model \eqref{model} with the Bayesian CART prior with $p_{lk}=n^{-c}$ with $  c>5/2$.  Under  Assumption \ref{ass:f_classic} (a)  with a large enough constant $A>0$,
with probability at least $1-4/n$ the Bayesian CART algorithm from Section \ref{section:sampler} satisfies 
\begin{equation}\label{eq:tau_bound1}
\tau_\epsilon\leq 2^{2L+3} \left\{ n\left[\left(c+\frac{1}{2}\right)\log (1+n) +  |\mT^*_{int}|C_{f_0}^2+1\right]   + 4\,|\mT^*_{int}|\log n+ \log \left(\frac{2}{\epsilon}\right) \right\}
\end{equation}
where $C_{f_0}>0$ is the constant from Assumption \ref{ass:f_classic}.  
\end{theorem}
\proof 
See Section \ref{seq:proof_thm_theo1}.
 
\begin{remark}
In the bound \eqref{eq:tau_bound1}, we intentionally separated the influence of model complexity (captured by the maximal allowed depth $L$) and the sample size $n$. 
In practice, the most reasonable choice for $L$ is the maximal allowed resolution $L=L_{max}$ which will give us cubic mixing in $n$.
If we were confident that the posterior over trees deeper than $L$ goes to zero, we can always devise a Markov chain with a smaller state space (trees up to level $L$).
For example, for $\alpha$-H\"{o}lderian function choosing $L\propto (n/\log n)^{1/(2\alpha+1)}$ yields $\mathbb P_{f_0}(\bT_L^c)=o(1)$.

\end{remark}

\begin{remark}{(Comparison with Spike-and-Slab)}
 The tree-structured signal in Assumption \ref{ass:f_classic} (a) is particularly flattering for Bayesian CART. It is interesting to compare this approach with a spike-and-slab prior and a one-site Metropolis-Hastings proposal.
 \citep{yang2016computational} showed rapid mixing of the MH algorithm in a high-dimensional linear model (i.e. \eqref{eq:main_model} with $p$ covariates and $\bm \nu=\bm\varepsilon$) and  a $g$-prior (which coincides with our prior in orthogonal designs). We attempt to rephrase their result in the context of  our wavelet regression matrix where $p=n/2$.
The point-mass spike-and-slab prior in  \citep{yang2016computational} assumes $\Pi(\bm\gamma)\propto \Big(\frac{1}{p}\Big)^{\kappa|\bm\gamma|}\mathbb{I}[|\bm\gamma|\leq s_0]$, where $s_0$ is a chosen upper bound on the model complexity and $\kappa$ is a model-size penalty. We show ({Theorem \ref{theo:yang} in the Supplement Section \ref{sec:comp_yang}}) that our upper bound on Bayesian CART mixing time is tighter by $n$ than the upper bound for Spike-and-Slab MH due to a tighter bound on the congestion parameter. 

 \end{remark}
 

The proof  of Theorem \ref{eq:theo1} rests on the canonical path argument and the sandwich  relation \eqref{eq:sandwich}. Together with \eqref{eq:gap_lb}, this yields
$\tau_{\epsilon}\leq l(\mE)\rho(\mE)\big(\log[1/\min_{\mT\in\bT_{L}}\Pi(\mT\C Y)]+\log 1/\epsilon\big)$. In the next section, we show Lemma \ref{lem:classic_mE} and Lemma \ref{lemma:conduct}  which 
provide an upper bound for the first two terms on the   right side. The logarithmic term is handled by the posterior consistency result in Lemma \ref{lemma:consist}. 
In the next section, we provide details of the canonical path construction and  describe basic properties of our canonical path ensemble.
Similarly as in \citep{yang2016computational}, whose canonical path architecture was inspired by stepwise variable selection, our construction was inspired by the CART algorithm \citep{breiman2017classification}.

\subsubsection{Canonical Path Ensemble for Bayesian CART}\label{sec:canonical}
We denote with  $\mT^*$ the signal-spanning tree from Assumption \ref{ass:f_classic}. First, we construct a canonical path $T_{\mT,\mT^*}$ from any tree $\mT\in\bT_{L}\backslash \{\mT^*\}$ 
towards  $\mT^*$ {\em along edges} in the graph with a transition matrix $P$. To this end, we introduce the {\em transition function} $\mG:\bT_{L}\backslash \mT^*\rightarrow\bT_{L}$ that maps the 
current state $\mT\in\bT_{L}$ onto the next state $\mG(\mT)\in\bT_{L}$ that is {\em ``closer"} to $\mT^*$, where closeness is determined by the Hamming distance $h(\mT,\mT^*)$ between binary tree encodings\footnote{
A binary tree encoding consists of a $(2^L\times 1)$ ordered (according to $2^l+k$) binary vector indicating whether or not $(l,k)\in\mT_{int}$.}.
The canonical path $T_{\mT,\mT^*}=\{\mT^0,\mT^1,\dots,\mT^k\}$ is constructed by composing the transition function so that
$$
\mT^0\equiv\mT\rightarrow \mT^1\equiv\mG(\mT)\rightarrow  \cdots \rightarrow\mT^k\equiv\mG^k(\mT)\equiv\mT^*,
$$
where $\mG^k(\cdot)=\mG\circ\dots\circ\mG(\cdot)$ is a composition of  $\mG$.
Below, we describe one particular transition function $\mG(\mT)$ which reduces the (Hamming) distance after each step, i.e.
$h[\mG(\mT),\mT^*]<h(\mT,\mT^*)\,\,\forall \mT\in \bT_{L}\backslash\mT^*.$
The mapping corresponds to a deterministic version of the PRUNE and GROW steps of  the Bayesian CART algorithm from Section \ref{section:sampler}.
\begin{itemize} 
\item[(1)] Assume $\mT\supset \mT^*$ is {\bf overfitted}, i.e.  $\mT$ forms an envelope around  $\mT^*$ and contains at least one signal-less node. 
 The mapping $\mG(\cdot)$ finds the deepest rightmost redundant node, say $(l,k)\in\mT_{int}\backslash{\mT^*_{int}}$,  and turns it into a bottom node. More formally  $\mG(\mT)=\mT^-$ where
 \begin{equation}
\mT^-_{int}=\mT_{int}\backslash\{(l,k)\}\quad\text{and}\quad \mT^-_{ext}=\mT_{ext}\backslash \{(l+1,2k),(l+1,2k+1)\}\cup\{(l,k)\}\label{eq:Tminus}
 \end{equation}
where  $(l,k)=\arg\max\limits_{(l',k')\in \mT_{int}\backslash\mT^*_{int}} (2^{l'}+k')$. 
 

\item[(2)] Assume $\mT\not \supseteq \mT^*$ is {\bf underfitted}, i.e. $\mT$ misses at least one influential node in $\mT^*$.
\begin{itemize}
\item[(i)] If $\mT\subset \mT^*$, 
the mapping $\mG(\cdot)$ finds the deepest rightmost external node in $\mT_{ext}\backslash\mT^*_{int}$, say $(l,k)$,  and turns it into an internal node. More formally  $\mG(\mT)=\mT^+$ where
 \begin{equation}
\mT^+_{int}=\mT_{int}\cup\{(l,k)\}\quad\text{and}\quad \mT^+_{ext}=\mT_{ext}\cup \{(l+1,2k),(l+1,2k+1)\}\backslash\{(l,k)\}\label{eq:Tplus}
 \end{equation}
where  $(l,k)=\arg\max\limits_{(l',k')\in \mT_{ext}\backslash\mT^*_{int}} (2^{l'}+k')$. 

\item[(ii)] If $\mT\not\subset \mT^*$, the tree $\mT$ contains redundant internal nodes. The mapping $\mG(\cdot)$  again
 finds the deepest rightmost redundant node, say $(l,k)$,  and turns it into a bottom node. We have the same expression for $\mT^-=\mG(\mT)$ as in \eqref{eq:Tminus}.
\end{itemize}
\end{itemize}

\begin{definition}For $\mT'\in\bT_L$ let  $\bar{T}_{\mT,\mT^*}$ denote the reverse of a path $T_{\mT,\mT^*}$. The Bayesian CART canonical path ensemble  is defined as $\mE=\{T_{\mT,{\mT'}}:(\mT,{\mT'})\in\bT_{L}\times\bT_{L}\}$, where
for each canonical path $T_{\mT,{\mT'}}$ is obtained by collapsing the paths $T_{\mT,\mT^*}$ and $  \bar{T}_{{\mT'},\mT^*}$, i.e. {$T_{\mT,{\mT'}}=T_{\mT\backslash \mT'}\cup \bar{T}_{{\mT'}\backslash \mT}$, where $T_{\mT\backslash \mT'} := T_{\mT,\mT^*}\backslash ( T_{\mT,\mT^*}\cap T_{\mT',\mT^*})$\footnote{{ By construction each step in $T_{\mT,\mT^*}$ reduces the Hamming distance, and thus we can show similarly to \cite{yang2016computational} that $\mE$ is an ensemble of \emph{simple} paths.} }.} 

\end{definition}

Below, we characterize important properties of $\mE$ which are instrumental in the sandwich relation   \eqref{eq:sandwich} and in the proof of Theorem \ref{eq:theo1}.  

\begin{lemma}\label{lem:classic_mE}
Let $\mE$ be the canonical path ensemble for Bayesian CART and let $|T_{\mT,{\mT'}}|$ denote the length of the path $T_{\mT,\mT'}\in\mE$ between $\mT,{\mT'}\in\bT_{L}$.  For $\mT^*$ defined   in Assumption \ref{ass:f_classic}, we have
$
\ell(\mE)\equiv\max\limits_{\mT,{\mT'}\in\bT_{L}}|T_{\mT,{\mT'}}| \leq 2^{L+1}.
$
\end{lemma}
{\proof See Section \ref{sec:lem_proofs1}.}

\smallskip

The following lemma characterizes the behavior of the congestion parameter $\rho(\mE)$ for the canonical ensemble $\mE$ constructed above.
\begin{lemma}\label{lemma:conduct}
Assume the model \eqref{model} with the Bayesian CART prior  with $p_{lk}=n^{-c}$ with $c>1$.  Under  Assumption \ref{ass:f_classic} (a), the canonical path ensemble $\mE$ for the Bayesian CART algorithm from Section \ref{section:sampler} satisfies 
$
\rho(\mE)\leq  2^{L+1}[1+o(1)]\quad\text{with probability at least $1-4/n$}.
$
\end{lemma}
\proof  See Section \ref{sec:lem2_proof}.

\subsection{Twiggy Bayesian CART Mixes Well}\label{sec:patul_mix_result}
In Section  \ref{sec:BCART_ass1} we established encouraging results for   Bayesian CART   with PRUNE and GROW steps under the connected signal Assumption \ref{ass:f_classic} (a). 
We have also seen that under Assumption \ref{ass:f_classic} (b), where   signals are not connected, Bayesian CART can mix poorly (Theorem \ref{theo:poor}).
We now investigate mixing of Twiggy Bayesian CART in the context of the unstructured signal in Assumption \ref{ass:f_classic} (b). Moving from Bayesian CART to Twiggy Bayesian CART extends   signal reachability where trees can become more competitive with the Spike-and-Slab approach \citep{yang2016computational} when   signal is obscured by  layers of noise.

\begin{theorem}\label{theo:patul}
Assume the model \eqref{model} with the Bayesian CART prior from with $p_{lk}=n^{-c}$ {and $c>5/2+\log D$}. Under Assumption \ref{ass:f_classic} (b) with $|\mT^*_{int}|\lesssim \log^2 n$ and large enough $A>0$,  the Twiggy Bayesian CART algorithm in Section \ref{sec:twiggy} (with $D>1$)
satisfies with probability at least $1-4/n$  
\begin{equation}\label{eq:tau_bound2}
\tau_{\epsilon}\leq   {\frac{(D^L-1)}{D-1}}\times 2^{2L+3}\left\{ n\left[\left(c+\frac{1}{2}\right)\log (1+n) +  |\mT^*_{int}|C_{f_0}^2+1\right]   + 4\,|\mT^*_{int}|\log n+ \log \left(\frac{2}{\epsilon}\right) \right\}
\end{equation}
\end{theorem}
\proof {See Section \ref{sec:proof_theo_patul}.}


\subsection{Locally Informed Versions Mix Even Better}

For the informed versions of Bayesian CART and Twiggy Bayesian CART described in Section \ref{sec:lit}, we obtain the following  upper bound on the mixing time that is only linear in $2^L$, where $L\leq L_{max}$ is required to go to infinity as $n\rightarrow\infty$. This speedup is likely a consequence of the  posterior-informed proposal defined in \eqref{eq:K_def}. 

\begin{theorem}\label{thm:coupling} 
Assume the model \eqref{model} and the Bayesian CART prior with $p_{lk}=n^{-c}$ and $ c>3$. 
Consider the Twiggy Bayesian CART with an informed proposal in \eqref{eq:K_def}. 
Under Assumption \ref{ass:f_classic} (a) or (b), for a large enough constant $A>0$ and $ L\leq L_{max}$ such that $L\rightarrow \infty$ as $n\rightarrow\infty$, we have with probability at least $ 1-4/n-\e^{-n/8}$, 
\[ \tau_\epsilon\lesssim  \log(6/\epsilon) \max\left( \frac{9\,(C_{f_0}+2)^2}{A^2}\frac{2^L\,n}{\log ^2n}, 2^{L+5}\right).\] For the informed Bayesian CART, the same bound holds but only under Assumption \ref{ass:f_classic} (a).
\end{theorem}

\proof See Section \ref{pf:coupling}.
\begin{remark} \label{rm:myopic_bad} Theorem \ref{thm:coupling} provides at most linear in $n$ mixing for Bayesian CART {\em only} under Assumption \ref{ass:f_classic} (a).  The exponential mixing rate lower bound in Theorem \ref{theo:poor} still applies to the informed Bayesian CART\footnote{See that \eqref{eq:simple_bottle} in the proof is valid as long as the proposal neighbor is the same as the Bayesian CART.}. Therefore, the proposal informativeness  alone does not solve the myopic problem of   Bayesian CART.
\end{remark}

\begin{remark}\label{remark:coupling_app}
It is worthwhile to point out that the linear mixing in Theorem \ref{thm:coupling} is truly a consequence of the informed proposal as opposed to the proving technique (two-drift condition as opposed to canonical path argument). By using the   two-drift condition proving technique (for $c\geq 4$ and $D\leq \e$), as opposed to the canonical path argument, we can slightly improve the mixing rate upper bound in \eqref{eq:tau_bound1} for  Bayesian CART and for Twiggy Bayesian CART (the original non-informed versions) to $\tau_\epsilon\lesssim  \log(6/\epsilon) \times \max\left(\frac{C_{f_0}^2 }{\delta_1 A^2}\frac{2^{2L}\,n}{\log ^2n},{2^{2L+1}}\right),$ where $\delta_1=1$ for the Bayesian CART, and $\delta_1= \frac{2(D-1)}{D^{L}-1}$ for the Twiggy Bayesian CART. Compared with the bound \eqref{eq:tau_bound1} obtained by the  canonical path argument, this bound has a slight improvement by a logarithmic factor {when $|\mT_{int}^*|$ is fixed}. {For more explanation, see the discussion in Remark \ref{rmk:why_save}.} The proof is in Section \ref{pf:coupling_app}. 
\end{remark}

The  proof  of Theorem \ref{thm:coupling} rests on the two drift condition argument developed by \cite{zhou2021dimension}. In the next section, we provide details of the two drift functions chosen for our tree regression setting. 

\subsubsection{Two Drift Conditions}
Up to now, we have relied on the canonical path argument to upper-bound the mixing rates. 
In order to show linear mixing, we apply the two-drift condition framework developed by \cite{zhou2021dimension}. We say that a drift condition is satisfied on $A\subset\bT_L$ when there exists a function $V:\bT_L\rightarrow[1,\infty)$ and a constant $\lambda\in(0,1)$ such that \[ (PV)(\mT)\leq \lambda V(\mT)\quad\text{for all}\quad \mT\in A,\quad\text{where} \quad (PV)(\mT)  = \sum_{\wt\mT\in \bT_L}V(\wt\mT)P(\mT,\wt\mT).\]
Similarly to the canonical path construction \citep{yang2016computational}, \cite{zhou2021dimension} observe that in Bayesian variable selection, the chain tends to escape underfitted states. If it escapes to an overfitted state rather than the true covariate vector, then again the chain tends to escape to the true covariate vector. This idea was formalized by using two drift functions; drifting first to the non-underfittted states (an overfitted state or the true covariates), and then drifting to the true covariate vector. We also apply the same idea in the settings of Bayesian CART and Twiggy Bayesian CART. 
\begin{definition}\label{eq:drift_conds_classic}
We define  two drift functions as
\begin{align}
V_1(\mT)&=\exp\left\{\frac{1}{2^{L}\,(C_{f_0}+2)^2\,(n+1)}\left( Y'(I-P_{\mT}/n)Y\right)\right\},\label{eq:rescuer}\\
V_2(\mT)&=\exp\left\{\frac{1}{2^{L}}\big(|\mT_{int}\backslash\mT^*_{int}|+(|\mT^*_{int}\backslash\mT_{int}|\wedge1)\times (2^{L}-|\mT_{int}\cup\mT^*_{int}|)\big)\right\},\label{eq:v2}
\end{align}
where $P_{\mT}=\bm X_{\mT}\bm X_{\mT}'$ so that $P_{\mT}/n$ denotes the projection onto the column space spanned by $\bm X_{\mT}$ in regular designs. The drift ratios are defined as 
\[R_i(\mT,\wt\mT)=V_i(\wt\mT)/V_i(\mT)-1 \text{ for } i=1,2.\]
\end{definition}
\begin{remark}\label{eq:v2design}
The second drift function $V_2$ is designed so that for any overfitting $\mT$ we have $V_2(\mT)=\exp\{|\mT_{int}\backslash\mT^*_{int}|/2^{L}\}$, while $V_2$ is a constant function on non-overfitting (underfitting) trees as $V_2(\mT)=\exp\{1- |\mT^*_{int}|/2^{L}\}$. Therefore, {for any $\mT,\wt\mT\in \bT_L$ such that $\mT\supset \mT^*$ and $\wt\mT\not\supset \mT^*$}, we can guarantee $V_2(\mT)\leq V_2(\wt\mT)$ since $|\mT_{int}\backslash\mT^*_{int}|+|\mT^*_{int}| = |\mT_{int}|\leq 2^{L}$. The following lemma characterizes the chosen drift functions and the drift ratios for $V_1$ and $V_2$.
\end{remark}


The first drift condition guarantees that the chain will frequently visit overfitted states, while the second condition guarantees that within the overfitted states, the chain will consistently attempt to hit the true tree $\mT^*$. The following proposition is used to obtain the bound in Theorem \ref{thm:coupling}.

\begin{proposition}\label{prop:v2_approx} Under the same assumptions of Theorem \ref{thm:coupling}, with probability at least $ 1-4/n-\e^{-n/8}$ and with $ c>5/2$ we have the following properties of the drift functions: 
\begin{itemize}
\item[(i)] For any underfitted tree $\mT\in\bT_L$ such that $\mT\not\supseteq\mT^*$,
\[\frac{(PV_1)(\mT)}{V_1(\mT)}\leq 1-\frac{A^2}{2^{L+5}(C_{f_0}+2)^2}\frac{\log^2n}{n}+\frac{\e-1}{2n^{(A^2/8\log n-1)}}.\]
\item[(ii)] For any overfitted tree $\mT\in\bT_L$ such that $\mT\supset\mT^*$, 
\[\frac{(PV_2)(\mT)}{V_2(\mT)}\leq 1-\frac{1}{2^{L+2}}\frac{1}{(1+n^{5/2-c})}+\frac{M}{n^{c-3/2}}+n^{1- (A^2\log n )/8 },\]
where $M=1$ for the Bayesian CART and $M=2L$ for the Twiggy Bayesian CART. 
\end{itemize}
\end{proposition}
\proof See Section \ref{pf:prop1}.

\section{Performance Evaluation}\label{sec:neumerical}

We  compare Bayesian CART, Twiggy Bayesian CART (with $D=2$) and their informed versions\footnote{The upper thresholds of the informed algorithms in \eqref{eq:wp}, we use $\e^{10}$ without tuning.} on simulated data as well as a real dataset. To appreciate the effect of tree-shaped regularization, we also compare Bayesian CART with the Metropolis-Hastings one-site sampler for Spike-and-Slab priors. 
\subsection{Simulation Study}
\paragraph{Data} We generate simulated data from the model \eqref{model} with  three  true signal skeletons. Given the skeleton, all true coefficients are set equal to $2$. (1) Case 1 (fully connected signal) is a  full tree of internal depth $3$, i.e. $\mT_{int}^*=\mB =\{ (-1,0)\}\cup_{l=0}^{3}\cup_{k=0}^{2^l-1}\{(l,k)\}$. (2) Case 2 is a single, disconnected and isolated deep signal at $\mB =\{(4,0)\}$. (3) Case 3 is a mixed signal consisting of several isolated nodes with  $\mB =\{ (2,0),(2,3), (3,2),(3,3), (3,4), (3,5), (4,15)\}$. Recall that $\mT^*$  is the smallest tree that includes $\mB$ as its internal nodes. 

\paragraph{Prior} As the split probability in \eqref{eq:tree_prior_1dim} for Bayesian CART (using $L=L_{max}$), we use $p_{lk} = \alpha n^{-c}$, which was used in our theoretical studies {up to a constant factor}. We choose $\alpha$ so that $p_{lk} = 0.25/2^{\Lmax-6}$ for $\Lmax \in [6,...,11]$. 
 For the Spike-and-Slab prior, we consider  $\frac{\Pi(\mT\cup (l,k))}{\Pi(\mT)}=p^{ss}_l$, where we consider two cases {$p^{ss}_1 =\alpha n^{-c}= 0.25/2^{\Lmax-6}$ and $p^{ss}_2= 0.01 \,n^{1/4} 6^{-l}$.} Note that $p^{ss}_2$ penalizes deep node inclusion  more strongly than $p^{ss}_1$. As shown Figure \ref{fig:ss_tune} in Section \ref{sec:add_vis}, several MCMC runs on Case (3) reveal that $p^{ss}_1$ has a higher acceptance rate but $p^{ss}_2$ converges faster to the true signals. The two chosen split probabilities satisfy the sufficient conditions studied in \cite{rockova2021ideal} with which a Spike-and-Slab algorithm can achieve locally adaptive minimax rate.

\begin{figure}
    \centering
    \includegraphics[width=\linewidth]{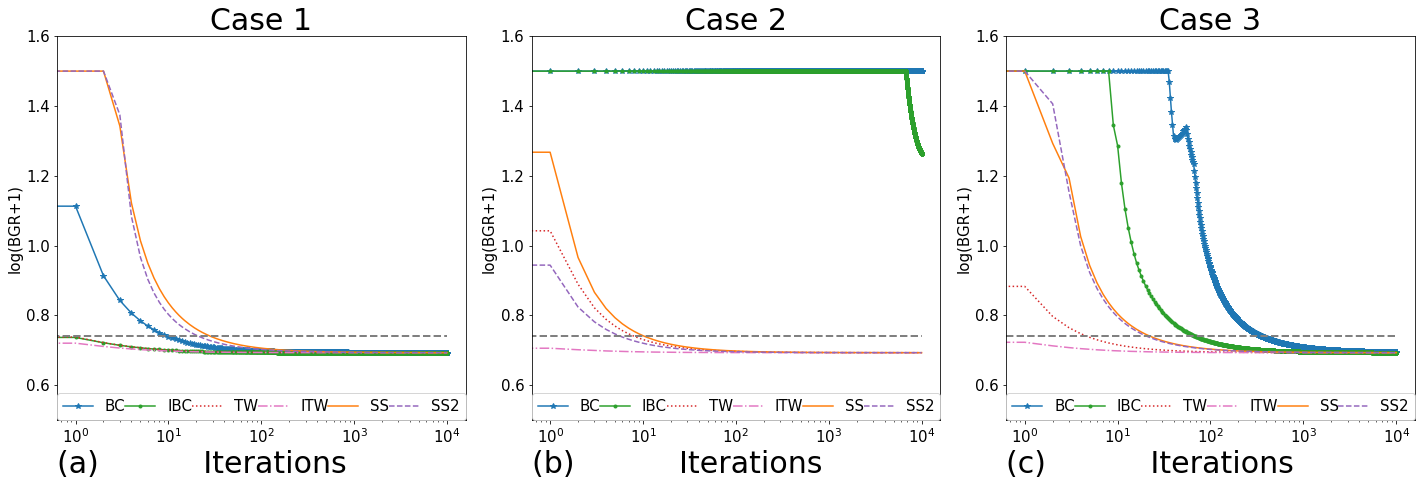}
    \caption{The local BGRs (log-transformed) when $n=2^7$. The local BGR values tend to decrease during the course of MCMC sampling. We have capped the values at a threshold  $log(Y+1) = 1.5$ for clearer visualization. The horizontal grey dotted line corresponds to the local BGR value 1.1. {(Legend) BC and IBC: original and informed Bayesian CART, TW and ITW: original and informed Twiggy Bayesian CART, SS and SS2: Spike-and-Slab with prior $p_1^{ss}$ and $p_2^{ss}$ respectively.}}
    \label{fig:bgr_support}
\end{figure}

\paragraph{Performance measure} We first define a proxy for the mixing time defined in \eqref{eq:tau_def}   using BGR (Gelman-Rubin diagnostic) \citep{gelman1992inference}. BGR measures the difference between in-chain  and across-chain variability when considering multiple initializations. Formally, consider a collection of $K$ chains $\mC=\{\mC_1,...,\mC_{K}\}$. Denoting the $L$ tree samples from each $\mC_k$ by $\{\mT_1^{\mC_k},...,\mT_L^{\mC_k}\}$ and the $j$-th coefficient sample from tree $\mT_i^{\mC_k}$ by $\beta_j(\mT_i^{\mC_k})$, the BGR for $\beta_j$ is 
\begin{equation*}
{\rm BGR}(\beta_j|\{\mT_1^{\mC_k},...,\mT_L^{\mC_k}\}_{k=1}^K) = \frac{\frac{L-1}{L}W_j+\frac{1}{L}B_j}{W_j},
\end{equation*}
where $W_j = \frac{1}{K}\sum_{k=1}^K \frac{1}{L-1}\sum_{i=1}^L (\beta_j(\mT_i^{\mC_k})-\bar{\beta}_{jk})^2$ and $B_j = \frac{L}{K-1}\sum_{k=1}^K(\bar{\beta}_{jk}-\bar{\beta}_j)^2$, given the in-chain and between-chain coefficient means $\bar{\beta}_{jk} = \frac{1}{L}\sum_{i=1}^L \beta_j(\mT_i^{\mC_k})$ and $\bar{\beta}_j =\frac{1}{K}\sum_{k=1}^K \bar{\beta}_{jk}.$ To reduce the computational cost of monitoring the BGRs for a large $L$,  we modify the BGR as follows. First, we measure the BGR of the estimated coefficient  $ \hat\beta_j(\mT_i^{\mC_k})$ instead of the sampled $\beta_j(\mT_i^{\mC_k})$. Given a tree $\mT$, the estimated coefficient of the $j^{th}$ node $ \hat\beta_j(\mT)=(X_j'X_j)^{-1}X_j'Y\times\mathbb{I}_{(l_j,k_j)\in \mT_{int}}$ only depends on whether the $j^{th}$ node is included in the tree. Therefore, it can be seen that  the BGR of $\hat\beta_j $ is equal to the BGR of an inclusion indicator, denoted by $I_j\in\{0,1\}$. Second, because BGR's tend to be smaller for noise coefficients, we monitor the BGR of signal nodes which have larger BGR values in general. Last, we compute BGRs locally; At time $t$, we consider the most recent $100$ samples to calculate BGR. Denote a $t^{th}$ sample of chain $\mC_k$ by $\mT_t^{\mC_k}$. Given $K=10$ chains, a local BGR of the $j^{th}$  indicator $I_j$ at time $t$ is defined as
\begin{equation}\label{eq:bgr_local}
{\rm BGR}(j,\mC|t)={\rm BGR}(I_j| \{\mT^{\mC_k}_{t},\mT^{\mC_k}_{t-1},...,\mT^{\mC_k}_{t-99}\}_{k=1}^{10}) .
\end{equation} As shown in Figure \ref{fig:bgr_support}, as the iteration proceeds, the ${\rm BGR}(j,\mC|t)$ values tend to decrease during the course of MCMC sampling. With this empirical observation, we define a proxy of the mixing time {called BGR $\alpha$-time by}
\begin{equation}\label{eq:measure_BGR}
\tau_\alpha^{\rm BGR} = 10^6 \wedge \arg\min_{t\geq 0} \{\max_{j\in S} {\rm BGR}(j,\mC|t)\leq \alpha\},
\end{equation}
where $S$ is the index set of true signals in the data generating mechanism. The inner maximum $\max_{j\in S} {\rm BGR}(j,\mC|t)$ quantifies the mixing quality of the chains at time $t$. In this paper, we consider $\alpha=1.1$, as this number used widely as a criterion of mixing \citep{gelman1992inference}.

We threshold the mixing time at $10^6$ because beyond this number, it is hard to see that a chain mixes in a reasonable time. In short, we measure the BGR in \eqref{eq:measure_BGR} for every $100$  out of $1,000,000$ iterations. Note that the chain may meander towards a poor local neighborgood, far from $\mT^*$. Therefore, to quantify the quality of each tree $\mT$, we can think  of $(l,k)\in \mT_{int}\backslash \mT^*_{int}$ as a \emph{false positive}  and $(l,k)\in \mT^*_{int}\backslash \mT_{int}$ as a \emph{false negative}. A natural quality measure for trees is the F1 score, the harmonic mean of precision and recall (a low precision indicates that the model is overfitting, while a low recall indicates that the model is underfitting).  If F1 equals $1$ (i.e. both precision and recall are $1$) then the tree equals $\mT^*$. The F1 values are obtained from the last 100 iterations of each chain. Note that \emph{not} all nodes in $\mT^*$ are necessarily signals. Therefore, when we calculate F1 for Spike-and-Slab, we consider only the true signals $\mathcal B$ as the true model. { In a similar spirit, we also measure hitting time defined by
$\tau_{hit} = \inf_{t\geq 0} \{\mT_t = \mB\}.$
} Lastly, we also present the acceptance rates, which may help understand the  stickiness of Markov chains. 

\begin{figure}[!t]
    \centering
    \includegraphics[width=1\linewidth]{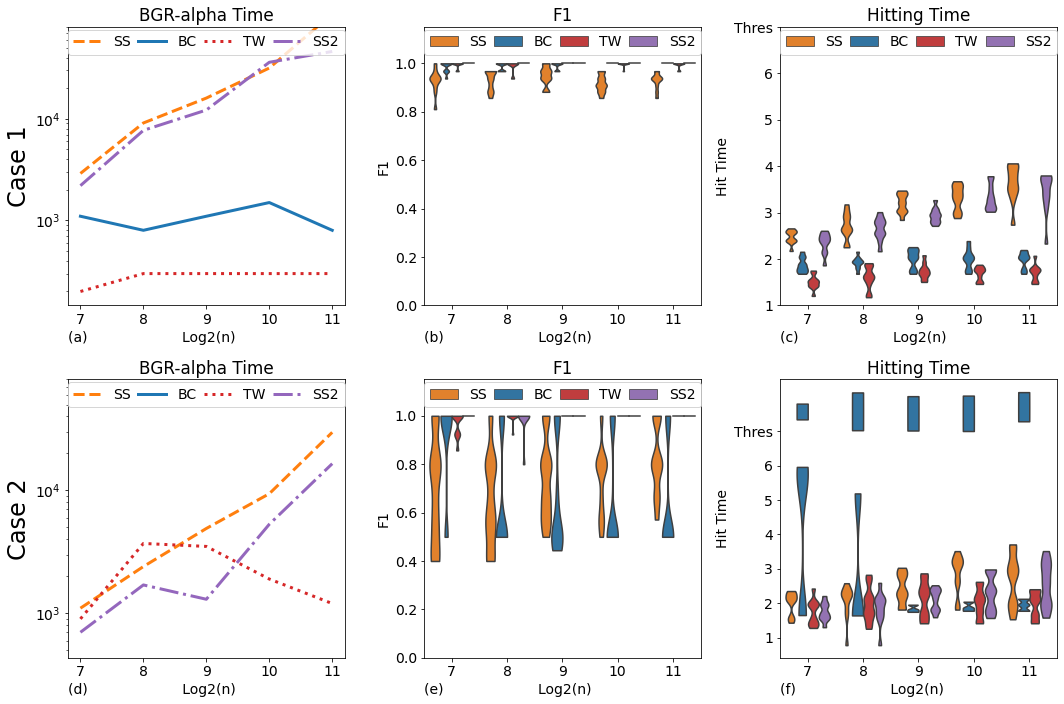}
    \caption{Plots (a) and (d) depict time for achieving the worst case local BGR below 1.1. Plots (b) and (e) show F1 scores. The small F1 value of Spike-and-Slab with $p_1^{ss}$ is due to low precision (overfit). Plots (c) and (f) show hitting times. {(Legend) BC: Bayesian CART, TW: Twiggy Bayesian CART, SS and SS2: Spike-and-Slab with prior $p_1^{ss}$ and $p_2^{ss}$ respectively.}}
    \label{fig:fastmix}
\end{figure}

\subsubsection{The Effect of Signal Structure} 
\paragraph{Connected signals.} The results on Case (1) numerically affirm the sufficient condition for rapid mixing of Bayesian CART in Theorem \ref{eq:theo1}, which says the Bayesian CART  mixes rapidly when all signals are connected. By increasing the size of data from $n=2^7$ to $n=2^{11}$ ($\Lmax$ from 6 to 10), we measure the BGR-$\alpha$ time $\tau_\alpha^{\rm BGR}$ for $\alpha=1.1$ in \eqref{eq:measure_BGR} and F1 as well as $\tau_{hit}$. We run 10 chains for  Bayesian CART, Twiggy Bayesian CART and Spike-and-Slab (both $p_1^{ss}$ and $p_2^{ss}$). For each method, 10 chains are initialized with randomly generated trees. The result is in Figure \ref{fig:fastmix} (a), (b), and (c). {We see that Bayesian CART hits the true tree faster than Spike-and-Slab.} In addition, we see that Bayesian CART achieves a good BGR-$\alpha$ time as the  sample size increases. We observe that  Bayesian CART is enjoying this favorable property over the Spike-and-Slab in two ways. First, we see that $\tau_\alpha^{\rm BGR}$ of Bayesian CART increases more slowly than Spike-and-Slab, and {second, the hitting time in Figure \ref{fig:fastmix} (c) is superior over that of Spike-and-Slab. Further investigation revealed that the source of the low F1 values of Spike-and-Slab with $p_1^{ss}$ was  low precision i.e., it often  overfitted. We notice that the deeper penalty through $p_2^{ss}$ (as opposed to) $p_1^{ss}$ improves Spike-and-Slab in all the performance measures (F1 and BGR-$\alpha$ time, and the hitting time).} Lastly, we observe that Twiggy Bayesian CART does similarly well as Bayesian CART. Note that in Figure \ref{fig:bgr_support}, when $n=2^7$, the local BGRs of the Twiggy Bayesian CART decrease faster compared to Bayesian CART.

\paragraph{Disconnected signals.} Now we numerically affirm the exponential lower bound of Bayesian CART in Theorem \ref{theo:poor} in the context of  deep isolate signals as in Example \ref{ex:caution}. On the other hand, Theorem \ref{theo:patul} says that Twiggy Bayesian CART still mixes rapidly. These theoretical results are affirmed by the results on Case (2) in Figure \ref{fig:fastmix} (d), (e), and (f). In terms of $\tau_\alpha^{\rm BGR}$, we see that Twiggy competes with Spike-and-Slab and then performs better when the sample size becomes larger.  Bayesian CART did not achieve local BGR under 1.1 in \eqref{eq:measure_BGR} in a given time range (1,000,000 iterations).
 This is why there is no line for Bayesian CART in Figure \ref{fig:fastmix} (d). Similarly, the hitting time would have to exceed the maximum number of $1,000,000$ iterations. This is marked by the histogram bars above the maximum allowed number of iterations. {Besides the BGR-$\alpha$ time displayed in Figure \ref{fig:fastmix} (d), in Section \ref{sec:add_vis} we additionally display  the smallest local BGRs in Figure \ref{fig:hit1} (e). This value is obtained by running all the chains for $10^6$ iterations and by taking the minimum over local BGRs defined in \eqref{eq:bgr_local} for each chain.} We observe that the minimum local BGR of Bayesian CART is exceedingly large. { This is related to the small F1 of Bayesian CART; further investigation revealed that low Recall values made F1 small, indicating underfit of Bayesian CART}. {On the other hand, Spike-and-Slab achieves local BGR smaller than 1.1, and using $p_2^{ss}$ resolves the  overfitting problem of $p_1^{ss}$ as in Figure \ref{fig:fastmix} (e). However, when $n$ increases, even using $p_2^{ss}$ does not bring up Spike-and-Slab to the speed of Twiggy Bayesian CART in terms of BGR-$\alpha$ time and hitting time (Figure \ref{fig:fastmix} (d) and (f)).} 

\begin{figure}[!t]
    \centering
    \vspace{-0.5cm}
    \includegraphics[width=\linewidth]{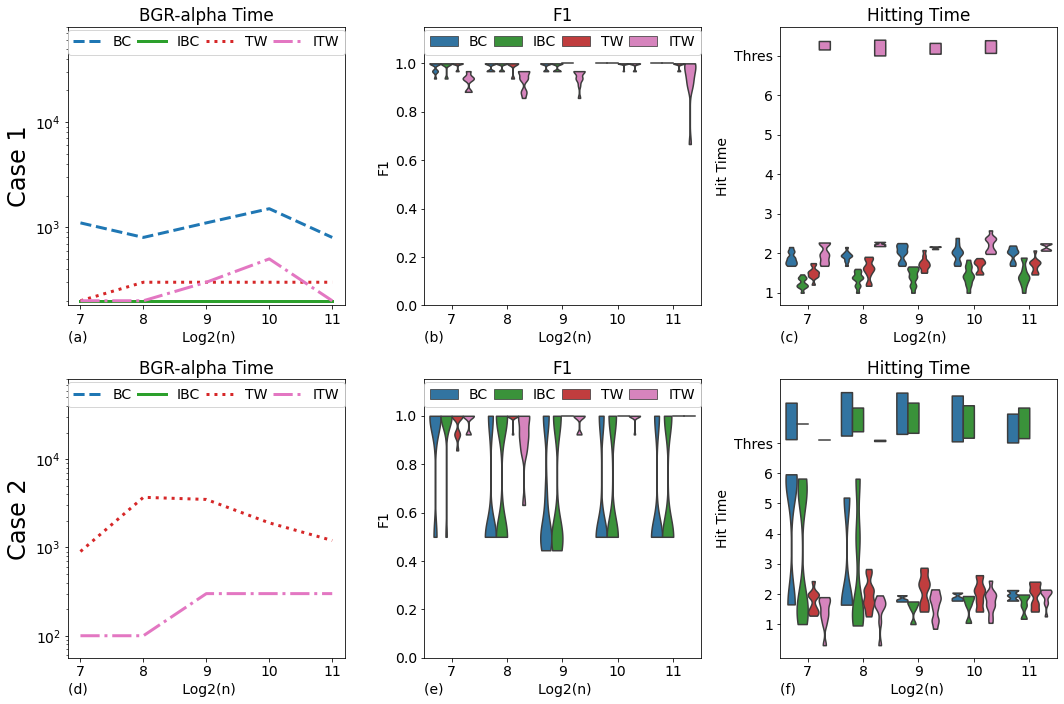}
    \caption{The improvement of the informed variants. (a) and (d) Time for achieving the worst case local BGR below 1.1. (b) and (e) Informed Twiggy Bayesian CART tends to overfit than its non-informed version. (c) and (f) Hitting times. {(Legend) BC and IBC: original and informed Bayesian CART, TW and ITW: original and informed Twiggy Bayesian CART.}}
    \label{fig:tw1}
\end{figure} 

\begin{figure}[!t]
    \centering
    \includegraphics[width=\linewidth]{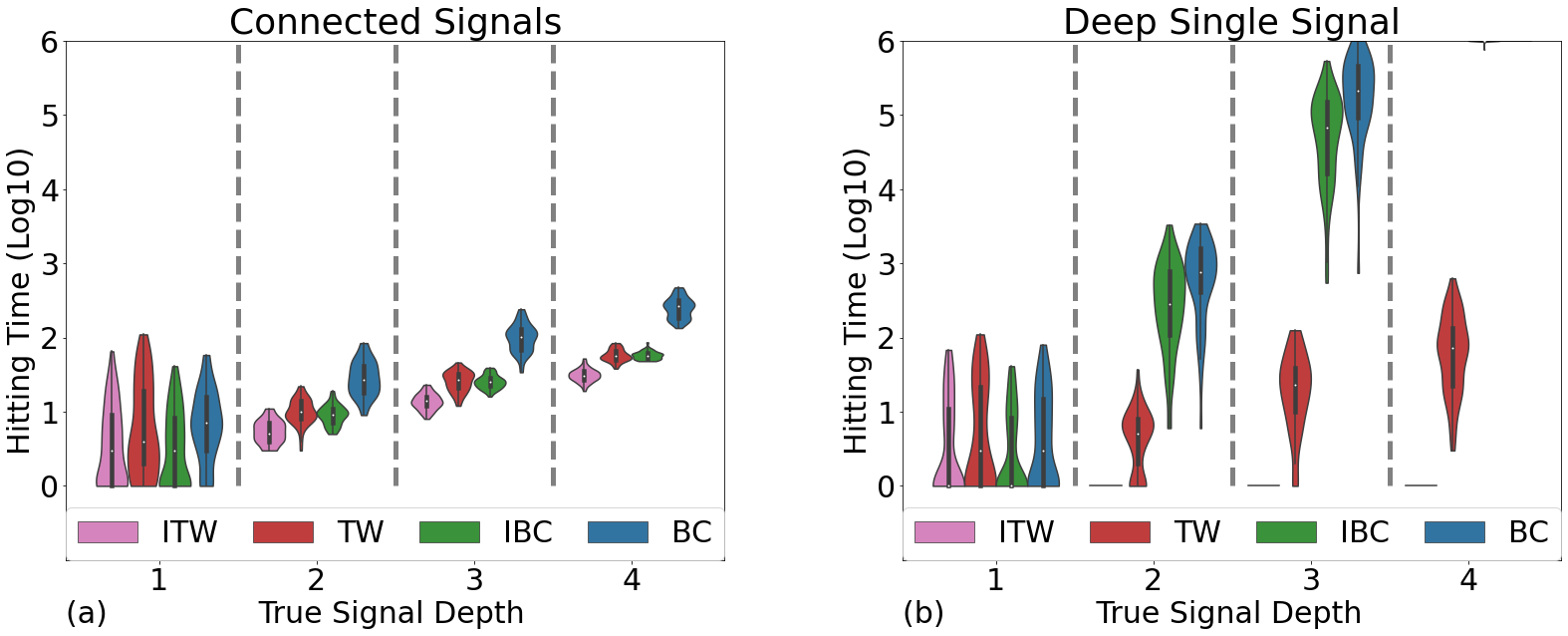}
    \caption{Hitting time when the true tree gets deeper. Informed Bayesian CART hits the true signals faster than Bayesian CART. Likewise, the informed Twiggy Bayesian CART is faster than Twiggy Bayesian CART. (b) However, for an isolated deep signal, informed Bayesian CART does not hit faster than Twiggy Bayesian CART. {(Legend) BC and IBC: original and informed Bayesian CART, TW and ITW: original and informed Twiggy Bayesian CART.}}
    \label{fig:informed1}
\end{figure}

\subsubsection{The Effect of Informed MCMC} 
We repeat the analyses to investigate the improvement brought by the posterior-informed proposals. From Figure \ref{fig:tw1} and Figure \ref{fig:tw3} in Section \ref{sec:add_vis}, we see that  informed versions overall improve on their non-informed counterparts. The BGR-$\alpha$ (Figure \ref{fig:tw1} (a) and (d), and Figure \ref{fig:tw3} (a)) and hitting times (Figure \ref{fig:tw1} (c) and (f), and Figure \ref{fig:tw3} (c)) become faster. However, in terms of hitting time, the informed Twiggy Bayesian CART does not outperform its non-informed version. This is related to that the informed Twiggy Bayesian CART tends to overfit, thereby having smaller F1 values (Figure \ref{fig:tw1} (b) and (e), and Figure \ref{fig:tw3} (b)). It might be understood as a trade-off of having a higher acceptance rate than Twiggy Bayesian CART (Figure \ref{fig:tw2}). However, the informed Bayesian CART also has an increased acceptance rate, but it does not result  in overfitting. We think that the decreased precision of informed Twiggy Bayesian CART is due to the combination of \emph{more flexible movement} and the larger acceptance rate (e.g., the lower bound in \eqref{eq:wp}). 
 
 When the signal depth deepens, in the same settings of Example \ref{ex:caution} and Example \ref{ex:caution2}, the hitting time results are in Figure \ref{fig:informed1}. We can see that informed versions hit generally faster than their non-informed versions. However, as in Figure \ref{fig:informed1} (b), the informed Bayesian CART falls short of resolving the hitting-time slow down, which is exponential in the signal depth. This result is consistent with Remark \ref{rm:myopic_bad}, implying that the problem of the myopic movement cannot be overcome even when using informed proposals. For an example where signals are disconnected but not very far from each other, see Figure \ref{fig:informed4} (b). { On the other hand, the informed Twiggy Bayesian CART includes the signal in a single step.}

\begin{figure}
    \centering
    \includegraphics[width=\linewidth]{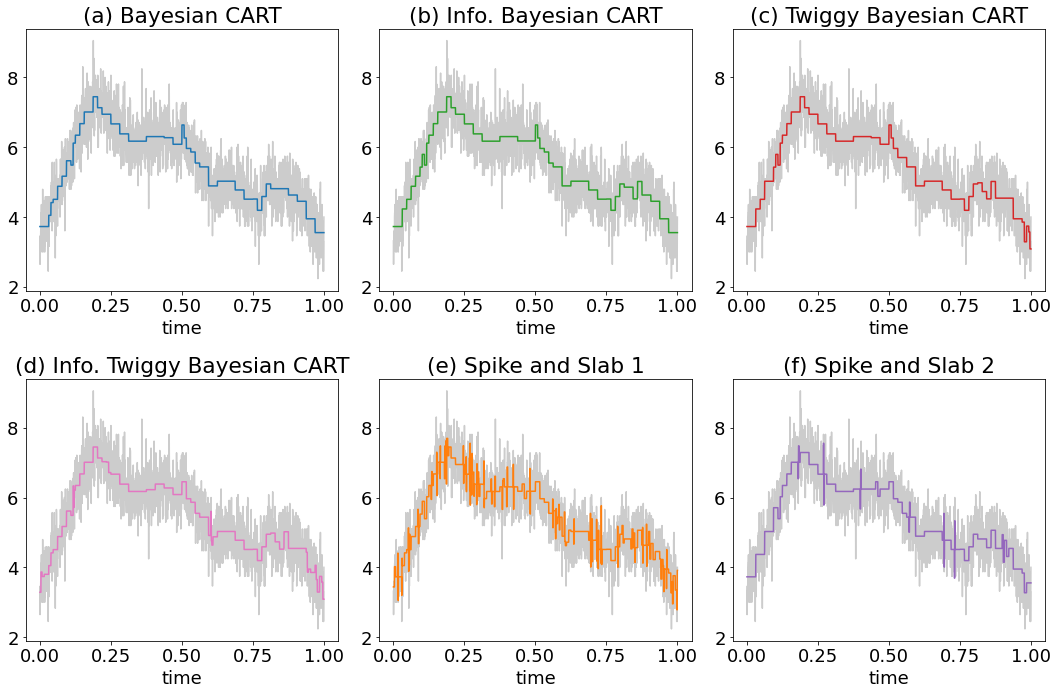}
    \caption{The visualization of MCMC chains on Call Center Data. {The colored lines are the median tree fit obtained from  1000 samples after 10,000 burn-in and the gray lines are the data.} (a) Bayesian CART (b) informed Bayesian CART (c) Twiggy Bayesian CART (d) informed Twiggy Bayesian CART (e) Spike-and-Slab (prior: $p_{lk}^{ss,1} = 0.01$) (f) Spike-and-Slab (prior: $p_{lk}^{ss,2} = 0.01 \times 6^{-l/2}$)}
    \label{fig:call1}
\end{figure}

\subsection{Call Center Data}
The data set, collected by a call center of an Israeli bank, contains arrival times, waiting times and service times. We focus on the arrival times, which can be seen  as an inhomogeneous Poisson process with a mean function $\mu(t)$  \cite{brown2005statistical}. This dataset was also studied in \cite{cai2014adaptive,rockova2021ideal} in the context of constructing adaptive confidence bands in non-parametric regression. In our analysis, we want to compare the mixing performance of each method studied in our simulations. We preprocess the data following \cite{rockova2021ideal} where the response is $Y_i = \sqrt{N_i+1/4}$ where $N_i$ is the number of calls arriving in the $i$-th time interval. We have $n=2048$ equispaced time intervals. As a proxy of mixing, we measure MSE  (distance from data to the posterior function draw) and the maximum local BGRs 
$\max_{j}{\rm BGR}(\beta_j| \{\mT^{\mC_k}_{t},\mT^{\mC_k}_{t-1},\ldots,\mT^{\mC_k}_{t-99}\}_{k=1}^{10})$ at time $t$. As the split probability of the tree based models, we use $p_{lk}=0.01$. For Spike-and-Slab, we again use two types of priors $p_{lk}^{ss,1} = 0.01$ and $p_{lk}^{ss,2} =  0.01 \times 6^{-l/2}$. As discussed in \cite{cai2014adaptive,rockova2021ideal}, the data approximately follows the model in \eqref{model} with the variance $\sigma^2 = 1/4$. Therefore, {we fix the variance at 0.25 in all methods.}  
\begin{figure}
    \centering
    \includegraphics[width=\linewidth]{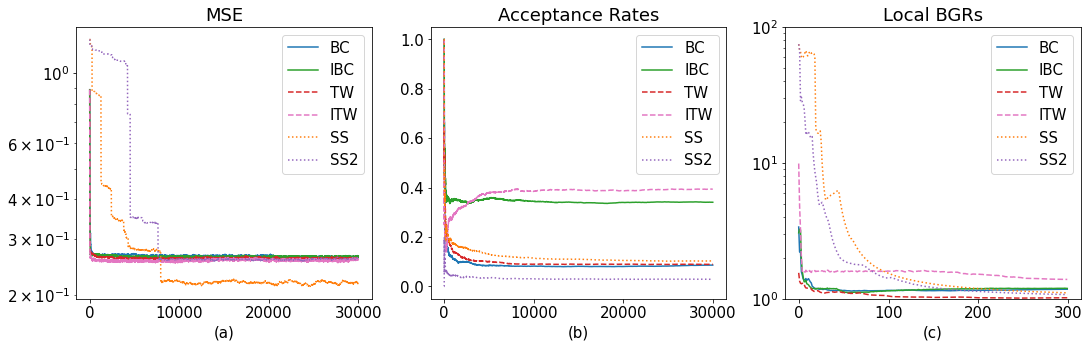}
    \caption{The performance measures on the Call Center data. (a) The MSE (log transformed) over the MCMC iterations (b) Acceptance rates. (c) The local BGRs for every 100 iterations on the Call center data. The minimum local BGRs achieved were Bayesian CART: 1.18, informed Bayesian CART: 1.19, Twiggy Bayesian CART: 1.02, informed Twiggy Bayesian CART: 1.39, Spike-and-Slab ($p_{lk}^{ss,1}$): 1.11, Spike-and-Slab ($p_{lk}^{ss,2}$): 1.07.}
    \label{fig:call4}
\end{figure}

In Figure \ref{fig:call1}, Spike-and-Slab shows a trade-off between overfitting and mixing. When the split probability is low ($p_{lk}^{ss,2}$), the chain may avoid overfitting (Figure \ref{fig:call1} (f)) compared with a high split probability $p_{lk}^{ss,1}$ (Figure \ref{fig:call1} (e)). However,   Figure \ref{fig:call4} (a) shows that the speed of the chain's ability to explain the data is much slower. The informed versions of the tree models catch more detailed signals than their non-informed counterparts. The speed of decreasing MSE is informed Twiggy Bayesian CART $<$  Twiggy Bayesian CART $<$ informed Bayesian CART $<$ Bayesian CART $<$ Spike-and-Slab (Figure \ref{fig:call4}).

\section{Concluding Remarks}\label{sec:conclusion}
This work is the first to have described upper bounds on mixing times for Bayesian CART, a simplified version of BART. We focused on one-dimensional setting and various proposal schemes, including our new Twiggy Bayesian CART proposal. We showed rapid mixing of Bayesian CART when the signal is connected on a tree. We also obtained rapid mixing for Twiggy Bayesian CART which does not require this assumption. We showed that without signal connectivity, Bayesian CART mixes poorly.
Extending our conclusions to more dimensions is an interesting problem. The first challenge   is the absence of  identifiability of the multi-dimensional trees. The non-identifiabiltiy prevents   from guaranteeing  posterior consistency (e.g., Lemma \ref{lemma:consist}), which is crucial in the canonical path argument. 
We believe that our results nevertheless serve as a valuable first step towards characterizing mixing of Bayesian Additive Regression Trees which have proven so useful in practice.

\pagebreak

\tableofcontents


\section{Bayesian CART Algorithm} \label{sec:algorithm_bcart}The algorithmic description of the original Bayesian CART (dyadic version) is in Algorithm \ref{alg:original}.
\begin{figure}[!ht]
\centering
\begin{minipage}{\linewidth}
\begin{algorithm}[H]
\caption{{\em Original Bayesian CART (Dyadic Version).}}\label{alg:original}
\centering
\spacingset{1.1}
\small
\resizebox{\linewidth}{!}{
\begin{tabular}{l l}
\hline
\multicolumn{2}{c}{\bf Input \cellcolor[gray]{0.6} }\\ 
\hline
\multicolumn{2}{l}{\textbf{Input}: The maximum iteration number $T_{max}$, the initial tree $\mT^0$, the posterior $\Pi(\mT\C Y)$}\\
\hline
\multicolumn{2}{c}{\bf Sampling \cellcolor[gray]{0.6}}\\ 
\hline
\multicolumn{2}{l}{For $i=1,...,T_{max}$}\\
\multicolumn{2}{l}{\qquad Sample $u_i\sim Unif(0,1)$}\\
\multicolumn{2}{l}{\qquad If $u_i>0.5$ or $\mT^i=\mT_{null}$, propose a new candidate tree by GROW}\\
\multicolumn{2}{l}{\qquad Else, propose a new candidate tree $\wt\mT$ by PRUNE}\\
\multicolumn{2}{c}{ \cellcolor[gray]{0.9}{\bf GROW}  }\\
\multicolumn{2}{l}{\qquad Randomly pick a terminal node $(l^*,k^*)\in\mT^i_{ext}$. }\\
\multicolumn{2}{l}{\qquad Split $(l^*,k^*)$ into two daughter nodes by splitting the interval $I_{lk}$ at a dyadic rational midpoint\footnote{This can be extended to the fullblown original version by first choosing a direction and a split point uniformly.} by }\\
\multicolumn{2}{l}{\qquad \qquad $\wt\mT_{int}\leftarrow\mT^i_{int}\cup\{(l^*,k^*)\}$}\\
\multicolumn{2}{l}{\qquad \qquad $\wt\mT_{ext}\leftarrow\mT^i_{ext}\backslash\{(l^*,k^*)\}\cup\{(l^*+1,2k^*),(l^*+1,2k^*+1)\} $}\\
\multicolumn{2}{l}{\qquad Set $\mT^{i+t}=\tilde{\mT}$ with probability $\alpha(\mT^i,\wt\mT)=\min\left\{1, \frac{\Pi(\wt\mT\C Y)|\mT^i_{ext}|}{\Pi(\mT^i\C Y)|\mP(\wt\mT)|}\right\}$}\\
\multicolumn{2}{c}{ \cellcolor[gray]{0.9}{\bf PRUNE}  }\\
\multicolumn{2}{l}{\qquad Randomly pick a parent of two terminal nodes $(l^*,k^*)\in\mP(\mT^i)$.}\\
\multicolumn{2}{l}{\qquad Collapse the nodes below it and turn it into a terminal node by}\\
\multicolumn{2}{l}{\qquad \qquad $\wt\mT_{int}\leftarrow\mT^i_{int}\backslash\{(l^*,k^*)\}.$}\\
\multicolumn{2}{l}{\qquad \qquad $\wt\mT_{ext}\leftarrow\mT^i_{ext}\backslash\{(l^*+1,2k^*),(l^*+1,2k^*+1)\}\cup\{(l^*, k^*)\}.$}\\
\multicolumn{2}{l}{\qquad Set $\mT^{i+t}=\tilde{\mT}$ with probability $\alpha(\mT^i,\wt\mT)=\min\left\{1, \frac{\Pi(\wt\mT\C Y)|\mP(\mT^i)|}{\Pi(\mT^i\C Y)|\wt\mT_{ext}|}\right\}$}
\end{tabular}}
\end{algorithm}
\end{minipage}
\end{figure}


\section{Proof of Theorem \ref{lemma:consist} (Establishing Consistency)}\label{sec:proof_lemma_consist}
We assume that the truth $f_0$ is a step function  as in Assumption \ref{ass:f_classic} (a) or (b) with signals $\mathcal B(A)\equiv \{(l,k): { C_{f_0}>}|\beta^*_{lk}|> A\log n/\sqrt{n} \}\subseteq \{(l,k):l<L\}$. Recall that $\mT^*$ is the smallest tree that includes $\mathcal B(A)$ as internal nodes and $\mT_{full}^L=\{(l,k): l<L\}$ is the full tree up to depth $L$.
Recall the $(n\times p)$ Haar wavelet regression matrix $\bm X$ with wavelets up to the maximal resolution $L_{max}$ (i.e. $p=n/2$). We will work conditionally on the event space $\mA_n$ defined as
\begin{equation}
\mA_n\equiv\{\bm\varepsilon: \|\bm X'\bm\varepsilon\|_\infty\leq 2\|\bm X\|\sqrt{\log p}\},\label{eq:event_set}
\end{equation} 
where $\|\bm X\|=\max\limits_{1\leq j\leq p}\|X_j\|_2.$ It is known  that  $\P(\mA_n^c)\leq 2/ p=4/n\rightarrow 0.$

We split the set of eligible  trees $\bT=\bT_L$  into 
$$
\bT=\mT^*\cup \bT_U \cup \bT_O,
$$ 
where $\bT_U=\{\mT\in\bT: \mathcal \mT^*\not\subseteq \mT \}$ are all under-fitted trees that miss at least one internal node inside $\mT^*_{int}$ and
$\bT_O=\{\mT\in\bT: \mathcal \mT^*\subset \mT \}$ are all over-fitted trees that include inside at least one redundant internal node  in $\mT_{full}\backslash \mT^*_{int}$.  We show below that on the event $\mA_n$ we have $\Pi[\bT_O\C Y]=o(1)$ and $\Pi[\bT_U\C Y]=o(1)$ for $c>5/2$.

\subsubsection{Trees do not overfit.}\label{sec:do_not_overfit}
We decompose the overfitted set $\bT_O=\bigcup_{K=1}^{2^L} \Lambda(\mT^*,K)$ into shells depending on how many extra internal nodes the overfitted tree $\mT\in\bT_O$ has relative to $\mT^*$,
where
$$
 \Lambda(\mT^*,K)=\{\mT\in\bT_O:   |\mT_{int}|-|\mT^*_{int}|=K\}.
$$
We can write
\begin{equation}
\frac{\Pi[ \Lambda(\mT^*,K)\C Y]}{\Pi(\mT^*\C Y)}= {\sum_{\mT\in  \Lambda(\mT^*,K)}\frac{\Pi(\mT)N_\mT(Y)}{\Pi(\mT^*)N_{\mT^*}(Y)}} \label{eq:ratio-post}
\end{equation}
where the  marginal likelihood ratio can be written as  (using the expression in \eqref{eq:compare_post})
$$
\frac{ N_\mT(Y)}{ N_{\mT^*}(Y)}=(1+n)^{-K/{2}}\exp\left\{ \frac{1}{2(n+1)} Y'[X_{\mT} X_{\mT}' -X_{\mT^*} X_{\mT^*}']Y\right\}.
$$
For $\mT\in \Lambda(\mT^*,K)$
we denote with $\mT^0\equiv\mT^*\rightarrow\mT^1\rightarrow\dots\rightarrow\mT^K\equiv\mT$ the sequence of nested trees obtained from $\mT^*$ by growing one internal node (at a depth $l_j$) at a time towards  reaching $\mT$. We will use a shorthand notation $p_l=p_{lk}$ for the split probability.
The prior ratio of two consecutive trees in this sequence satisfies
$$
\frac{\Pi(\mT^{j})}{\Pi(\mT^{j-1})}=\frac{p_{l_j}}{1-p_{l_j}}\times (1-p_{l_j+1})^2
$$
Then  we find
\begin{align} 
\frac{\Pi(\mT)N_Y(\mT)}{\Pi(\mT^*)N_{Y}(\mT^*)}&=(1+n)^{-K/2}\prod_{j=1}^K\frac{p_{l_j}}{1-p_{l_j}}\times (1-p_{l_j+1})^2\times\exp\left\{ -\frac{ Y'(P_{ {j-1}}-P_{j})Y}{2(n+1)} \right\}\notag\\
&=(1+n)^{-K/2}\prod_{j=1}^K\frac{p_{l_j}}{1-p_{l_j}}\times (1-p_{l_j+1})^2\times \exp\left\{ \frac{  |X_{[j]}' Y|^2}{2(n+1)} \right\}, \label{eq:ratio_post_overfit}
\end{align}
where 
$$
P_j= X_{\mT^j} X_{\mT^j}'=P_{j-1}+X_{[j]}X_{[j]}' 
$$
and where $X_{[j]}$ is the column added at the $j^{th}$ step of  branch growing. Since $\mT^*_{int}$ contains {\em all} signals, we have $\b_{\bmT^*}^*=\bm0$.
Then for any $j=1,\dots, K$ we have on the event $\mathcal A_n$ (since $\bm \nu=\bm\epsilon$ under Assumption \ref{ass:f_classic} and due to orthogonality of $X$)
$$
 |X_{[j]}' Y|=|X_{[j]}'(X_{\mT^*}\b^*_{\mT^*}+X_{\bmT^*}\b^*_{\bmT^*}+\bm\nu )|\leq { 2}\sqrt{n\log n}
$$
Using $p_l=p_{lk}=n^{-c}<1/2$ we obtain   
\begin{equation}
\frac{\Pi(\mT)N_Y(\mT)}{\Pi(\mT^*)N_{Y}(\mT^*)} \leq \exp\left(-\frac{K}{2} \log(1+1/n)-K\, (c-3/2) \log n \right)\leq \e^{- K(c-3/2)\log n}.\label{eq:ratio_post_overfit2}
\end{equation}
Noting that the cardinality of $\Lambda(\mT^*,K)$ can be for each $K$ bounded by
$$
\mathrm{card}[\Lambda(\mT^*,K)]\leq\prod_{j=1}^{K}(|\mT^*_{ext}|+j-1)
$$
we find an upper bound for \eqref{eq:ratio-post}
\begin{equation}
\frac{\Pi[ \Lambda(\mT^*,K)\C Y]}{\Pi(\mT^*\C Y)}\leq (|\mT^*_{ext}|+K-1)^K \e^{-K(c-3/2)\log n}.\label{eq:bound_overtit_by_k}
\end{equation}
Since
\begin{align}
\frac{ \Pi(\bT_O\C Y)}{\Pi(\mT^*\C Y)}&\leq \sum_{K=1}^{2^L}
  \frac{\Pi(\Lambda(\mT^*,K)\C Y)}{\Pi(\mT^*\C Y)}< \sum_{K=1}^{2^L}\e^{ K\log [(|\mT^*_{ext}|+K-1)]}  \e^{-K \,(c-3/2)\,\log n}\notag\\
& <\sum_{k=1}^{2^L} \e^{-K(c-5/2)\log n}\leq{n^{5/2-c}\frac{1-[n^{(5/2-c)}]^{n/2}}{1-n^{5/2-c}}}<\frac{1}{n^{c-5/2}-1}\label{eq:bound_overfit}
\end{align}
we obtain that $\Pi[\bT_O\C Y]< \frac{1}{n^{c-5/2}-1}$.

\subsubsection{Trees do not underfit.}\label{sec:do_not_underfit}

We now show that the probability of trees that miss at least one signal goes to zero. In particular,  we show  that (on the event $\mA_n$)  we have   
\begin{equation}\label{eq:catch_signal_np}
\Pi\left[ \mT\in\bT: \mathcal B(A)\not\subseteq \mT_{int} \C Y\right]\rightarrow 0\quad\text{as $n\rightarrow\infty$}
\end{equation}
for
\begin{equation}
\mathcal B(A)\equiv \{(l,k): { C_{f_0}> }|\beta_{lk}^*|>A\log n/\sqrt{n}\}\label{eq:setSnew}.
\end{equation}
The proof of \eqref{eq:catch_signal_np} follows the route of Lemma 3 in   \citep{castillo2021uncertainty} and Section 9.0.2 in \cite{rockova2021ideal}. 
For simplicity, we have focused in this work on the regular design case where the regression matrix is orthogonal and thereby $\Sigma_{\mT}=c_n (X_{\mT}'X_{\mT})^{-1}=\frac{1}{n+1}I_{|\mT_{ext}|}$. Suppose that 
$ (l_S,k_S)\in \mathcal B(A)$
is a signal node for some $A>0$ and let $\mT$ be such that $(l_S,k_S)\notin\mT$. 
We grow a branch from $\mT$ that extends towards $(l_S,k_S)$ to obtain an enlarged tree $\mT^+\supset\mT$. In other words $\mT^+$ is the smallest tree that contains $\mT$ and $(l_S,k_S)$ as an internal node. For details, we refer to Lemma 3 in \citep{castillo2021uncertainty}.
 We define $K=|\mT^+_{int}\backslash \mT_{int}|$ and write  (using the expression in \eqref{eq:compare_post})
\begin{equation}\label{eq:ratio2}
\frac{N_Y(\mT)}{N_{Y}(\mT^+)}=(1+n)^{K/{2}}\exp\left\{ \frac{1}{2(n+1)} Y'[X_\mT X_\mT' -X_{\mT^+} X_{\mT^+}' ]Y\right\}.
\end{equation}
We denote with $\mT^0=\mT\rightarrow\mT^1\rightarrow\dots\rightarrow\mT^K=\mT^+$ the sequence of nested trees obtained by adding one additional internal node $(l_j,k_j)$ towards $(l_S,k_S)$.
Then  we find
\begin{align} 
\frac{N_Y(\mT)}{N_{Y}(\mT^+)}&=(1+n)^{K/2}\prod_{j=1}^K\exp\left\{ \frac{ Y'(P_{ {j-1}}-P_{j})Y}{2(n+1)} \right\}\notag\\
&=(1+n)^{K/2}\prod_{j=1}^K\exp\left\{- \frac{  |X_{[j]}' Y|^2}{2(n+1)} \right\},\label{eq:posterior_ratio_overfit}
\end{align}
where $P_j= X_{\mT^j} X_{\mT^j}'=P_{j-1}+X_{[j]}X_{[j]}' $
and where $X_{[j]}$ is the column added at the $j^{th}$ step of  branch growing.
Let $X_{[K]}$ be the {\em last} column to be added to $X_{\mT^+}$, i.e. the {\em signal} column associated with $(l_S,k_S)$. We will be denoting simply $\beta_{[K]}^*\equiv \beta^*_{(l_S,k_S)}$ the coefficient associated with $X_{[K]}$.
Then (from the orthogonality of $X$)
$$
|X_{[K]}' Y|^2=|X_{[K]}' X_{[K]}\beta_{[K]}^*   +X_{[K]}' \bm \nu|^2
$$
Using the inequality $(a+b)^2\geq a^2/2-b^2$   we find that
$$
{|X_{[K]}'Y|^2} \geq n^2|\beta_{[K]}^*|^2/2 -  |X_{[K]}' \bm \nu|^2.
$$
On the event $\mathcal A_n$ and using the fact that $F_0-X\b^*=0$ under the step function Assumption  \ref{ass:f_classic} we find that 
$$
|X_{[K]}' \bm \nu| = |X_{[K]}'\bm  \varepsilon |\leq { 2}\sqrt{n\log n}
$$  
which yields 
$$
\frac{ |X_{[K]}' Y|^2}{2(n+1)}\geq \frac{n^2|\beta_{[K]}^*|^2}{4(n+1)}- \frac{{ 4}n\log n}{2(n+1)}.
$$
From the signal assumption $(l_S,k_S)\in\mathcal B(A)$, we have $|\beta_{[K]}^*|>A\log n/\sqrt{n}$ for some $A>0$ and thereby 
\begin{equation}
\frac{N_Y(\mT)}{N_{Y}(\mT^+)}\leq \exp\left\{\frac{K}{2}\log (1+ n)- \frac{  nA^2\log^2 n }{4(n+1)} +  \frac{{ 4}n\log n}{2(n+1)}\right\}  
\end{equation}
The prior ratio satisfies  (using again the notation $p_l=p_{lk}$)
\begin{equation}
\frac{\Pi(\mT)}{\Pi(\mT^+)}=\frac{1-p_{l_0}}{p_{l_0}}\times\left(\prod_{j=1}^{K-1}\frac{1}{p_{l_j}(1-p_{l_j})}\right)\times\frac{1}{(1-p_{l_K})^2}.\label{eq:prior_ratio_twig}
\end{equation}
Defining
$$
b(n)\coloneqq \frac{\Pi(\mT\C Y)}{\Pi(\mT^+\C Y)}
$$
with $p_l=p_{lk}=n^{-c}<1/2$, we have  $\Pi(\mT)/\Pi(\mT^+)\leq 2^{K}\e^{c\,K\log n}$, and thereby  (since $K\leq L\leq L_{max}=\log_2 [n/2]$) 
\begin{align}
b(n) &\leq 2^{K} \exp\left\{c\,K\log n+ \frac{K}{2}\log (1+ n)- \frac{  nA^2\log^2 n }{4(n+1)} +  \frac{{ 4}n\log n}{2(n+1)}\right\}.\label{eq:posterior_ratio_overfit2}
\end{align}
Following the proof technique in Lemma 2 in  \citep{castillo2021uncertainty} we conclude that {for some sufficiently large $A>0$ }
$$
\Pi[(l_S,k_S)\notin \mT_{int}\C Y]\leq l_S\times b(n)\leq \e^{ -A^2/4\log^2 n}.
$$
Thereby,  
$$
\Pi[\mathcal B(A)\not\subseteq\mT_{int}\C Y]\leq \sum_{(l_S,k_S)\in \mathcal B(A)}\Pi[(l_S,k_S)\notin \mT_{int}\C Y]\leq \e^{ -A^2/4\log^2 n}2^{L}\leq \e^{-A^2/8\log^2 n} \rightarrow 0.
$$
This concludes the proof of \eqref{eq:catch_signal_np}. Because $\mT^*$ is the minimal tree that contains $\mathcal B(A)$ as its internal nodes, this implies 
$\Pi[\mT\in\bT: \mathcal \mT^*\not\subseteq \mT\C Y]=\Pi[\bT_{U}\C Y]=o(1)$.

\section{Proof of Theorem \ref{theo:poor} (Bayesian CART Mixing Lower Bound)}\label{pf:poor_lower}
We  assume  that the true signal  $f_0(x)=\psi_{l*k*}(x) $ consists of just one deepest leftmost  wavelet coefficient  with {$0<l^*< L$} and $k^*=0$, where $|\beta^*_{l^*k^*}|>A\log n/\sqrt{n}$ according to Assumption \ref{ass:f_classic} (b). Figure \ref{fig:motiv} (a) illustrates a special case when $l^*=3$. During the proof, we take advantage of the bottleneck ratio bound  \citep{sinclair1992improved} 
\begin{equation}\label{eq:lb}
Gap(P)\leq 2\Phi,\quad\text{where}\quad\Phi=\min\limits_{\substack{A\subset\bT \\ 0<\Pi[A\C Y]\leq 1/2}}\frac{\sum_{\mT\in A,{\mT'}\in \bT\backslash A}\Pi(\mT\C Y)P(\mT,{\mT'})}{\Pi[A\C Y]}
\end{equation}
is the conductance which measures the ability of the chain to escape from any small region of the state space (and make a rapid progress to the equilibrium).

We now choose $A\subset\bT$ that gives a small value of the ratio inside the minimum in \eqref{eq:lb}, thereby providing a small upper bound of the conductance. Intuitively, among   trees without  the signal,
 the posterior is smaller for deeper trees. Recall that in Bayesian CART, the transition probability is non-zero only between trees that differ by one internal node. The signal  node $(l^*,k^*)$ is only reachable from trees that include $(l^*-1,0)$. The set of  trees that include $(l^*-1,0)$ thus comprises a bottleneck between trees that capture the signal node  $(l^*,0)$ and those that do not. Using this intuition, we will calculate the bottleneck ratio w.r.t. 
\begin{equation}
A_{\backslash (l^*-1,0)}\equiv\{\mT\in\mathbb{T}\C (l^*-1,0)\not\in \mT_{int}\}\label{eq:Aset}
\end{equation}
to  bound the conductance. Note that $\Pi[A_{\backslash (l^*-1,0)}\C Y]< 1/2$ since the posterior is concentrated to the true tree with $(l^*,0)$ (See, Lemma \ref{lemma:consist})\footnote{By consistency established in Lemma \ref{lemma:consist}, on the event space $\mA_n$ {with $c>5/2$}, we have $\Pi(\mT^*\C Y)> 1/2$ {with probability at least $1-4/n$ when the signal is large enough}. Since $\mT^*\not\in A_{\backslash (l^*-1,0)}$, we have $\Pi[A_{\backslash (l^*-1,0)}\C Y]\leq 1/2$.}. A tree $\mT\in A_{\backslash (l^*-1,0)}$ must contain $(l^*-2,0)$ to have a non-zero transition probability $P(\mT, {\mT'})$ for ${\mT'}\in A_{\backslash (l^*-1,0)}^c$. Therefore, denoting 
$$
 B_{l^*-1}\equiv A_{\backslash (l^*-1,0)}\cap \{\mT\in\mathbb{T}\C (l^*-2,0)\in \mT_{int}\},
$$ the bottleneck ratio w.r.t. $A_{\backslash (l^*-1,0)}$   bounds  the conductance  from above simply by 
\begin{equation}
\Phi\leq\frac{\sum_{\mT\in  B_{l^*-1}}\Pi(\mT\C Y)P(\mT,A_{\backslash (l^*-1,0)}^c)}{\Pi[A_{\backslash (l^*-1,0)}\C Y]}\leq \frac{\Pi[  B_{l^*-1}\C Y]}{\Pi[A_{\backslash (l^*-1,0)}\C Y]}.\label{eq:simple_bottle}
\end{equation}

We now show that the tightest upper bound in \eqref{eq:simple_bottle} is obtained when $l^*=L-1$. Namely, we first derive a bound w.r.t. a general $l^*$, and show that  the bound in \eqref{eq:simple_bottle}   becomes smaller as $l^*$ increases. We will  work conditionally  on the set $\mA_n$ defined in \eqref{eq:event_set}. {Recall the definition of $\mT^*$ from Assumption \ref{ass:f_classic} (b)}  as the minimal  tree that contains the signal $\mathcal B=\{(l^*,0)\}$.

To bound the ratio in \eqref{eq:simple_bottle}, we decompose $A_{\backslash (l^*-1,0)}$ into $l^*$ disjoint subsets that contain the leftmost node at a certain level and exclude the leftmost node at the next level:
$$
B_{i} \equiv\{\mT\in\mathbb{T}\C (i-1,0)\in \mT_{int},  (i,0) \not\in \mT_{int}\}\quad\text{for}\quad i=0,1,\dots, l^*-1.\label{eq:b_set}
$$
It is easy to see that $A_{\backslash (l^*-1,0)} = \bigcup_{i=0}^{l^*-1} B_i$, so that 
\begin{equation}
 \Pi[A_{\backslash (l^*-1,0)}\C Y]  = \sum_{i=0}^{l^*-1}\Pi[B_i\C Y].\label{eq:denominator_decomp}
\end{equation}
Therefore, the bound in \eqref{eq:simple_bottle} can be rewritten as
\begin{equation*}\label{eq:intermed}
\Phi\leq \frac{\Pi[B_{l^*-1}\C Y]}{\Pi[A_{\backslash (l^*-1,0)}\C Y]}=\frac{\Pi[B_{l^*-1}\C Y]}{\Pi[B_{l^*-1}\C Y]+\sum_{i=0}^{l^*-2}\Pi[B_i\C Y]}.
\end{equation*}

To see how small $\Pi[B_{l^*-1}\C Y]$ is compared to $\sum_{i=0}^{l^*-2}\Pi[B_i\C Y]$, we will scrutinize the posterior ratio $\Pi[B_{i-1}\C Y]/\Pi[B_{i}\C Y]$ for each 
$B_i$. We first characterize a simple relationship between $B_i$ and $B_{i-1}$ where each tree in $B_i$ can be obtained by attaching a ``mini-tree"  to some tree inside $B_{i-1}$. For each $\mT\in B_i$, we denote with $M(\mT)$ the operator that removes all descendants $D_{i-1, 0}(\mT)$ of the node $(i-1,0)$, i.e. for $\mT'=M(\mT)$
$$
\mT'_{int}=\mT_{int}\backslash  D_{i-1, 0}(\mT).
$$
The mapping $M(\cdot)$ is many to one and for each $\mT'\in B_{i-1}$
we denote 
$$
\mathcal N(\mT')=\{\mT\in B_i: \mT'=M(\mT)\}
$$
the nonempty set of those $\mT\in B_i$ that map onto the same $\mT'$. Then we can write
$$
\Pi(B_i\C Y)=\sum_{\mT\in B_i}\frac{\Pi(\mT\C Y)}{\Pi(M(\mT)\C Y)}\Pi(M(\mT)\C Y)=\sum_{\mT'\in B_{i-1}}\Pi(\mT'\C Y)
\sum_{\mT\in \mathcal N(\mT')}
\frac{\Pi(\mT\C Y)}{\Pi(\mT'\C Y)}.
$$
Each tree $\mT\in \mathcal N(\mT')$ differs from $\mT'$ by addition of at least one node without signal. 
We decompose $\mathcal N(\mT')=\cup_{K=1}^{2^{L_{max}}} \mathcal N(\mT',K)$ into shells according to how many extra noise nodes the trees have, 
where 
$
 \mathcal N(\mT',K)=\{\mT\in \mathcal N(\mT'): |\mT_{int}\backslash \mT'_{int}|=K \}.
$
We can use the posterior ratio expression 
\eqref{eq:ratio_post_overfit2} for nested models to conclude that  for any $\mT\in  \mathcal N(\mT',K)$ we have
$$
\frac{\Pi(\mT\C Y)}{\Pi(\mT'\C Y)}\leq  \e^{- K(c-3/2)\log n}
$$
The cardinality of the set $\mathcal N(\mT',K)$ is at most the number of all binary trees with  $K$ nodes. This corresponds to the Catalan number   $\mathbb C_{K}$, which
according to Lemma  S-3 in \citep{castillo2021uncertainty}, satisfies
$\mathbb C_{K}\asymp 4^{K}/{K}^{3/2}$.
Then
$$
\Pi(B_i\C Y)\lesssim  \sum_{\mT'\in B_{i-1}}\Pi(\mT'\C Y)
 \times \sum_{K=1}^{n/2}  4^K\e^{- K(c-3/2)\log n}\leq \frac{1}{ {n^{(c-3/2)}}/4-1}\times \Pi(B_{i-1}\C Y).
$$
Denoting with $\gamma_n= \frac{C}{n^{(c-3/2)}/4-1}$ the ``shrinkage factor" for some $C>1$, the posterior of $A_{\backslash (l^*-1,0)}$ satisfies 
\begin{align}
 \Pi[A_{\backslash (l^*-1,0)}\C Y]  &= \sum_{i=0}^{l^*-1}\Pi[B_i\C Y] \geq \Pi[B_{l^*-1}\C Y] \sum_{i=0}^{l^*-1}\left(\frac{1}{\gamma_n}\right)^i\nonumber \\&= \Pi[B_{l^*-1}\C Y] \frac{\gamma_n}{1-\gamma_n}\left[\left(\frac{1}{\gamma_n}\right)^{l^*} -1\right]. \label{eq:denominator_final}
\end{align}
Therefore, it follows from \eqref{eq:denominator_final} and \eqref{eq:simple_bottle} that
\begin{align*}
\Phi^{-1}&\geq\frac{\Pi[A_{\backslash (l^*-1,0)}\C Y]}{\Pi[B_{l^*-1}\C Y]} 
\geq \frac{\gamma_n}{1-\gamma_n}\left[\left(\frac{1}{\gamma_n}\right)^{l^*} -1\right]\\
&\geq  \frac{C}{n^{(c-3/2)}/4-C}\left[\left(\frac{n^{(c-3/2)}/4-1}{C}\right)^{l^*}-1\right]>\left(\frac{n^{(c-3/2)}/4-1}{C}\right)^{l^*-1}-1
\end{align*}
where we used the fact that $\gamma_n> 4C/n^{(c-3/2)}$ and that $C/ (n^{(c-3/2)}/4-1)<1$ for large enough $n$.
Therefore, we have
\[\frac{1}{Gap(P)}\,\,\geq \,\,\frac{1}{2\Phi} \,\,{\geq}\,\, \frac{1}{2} \left(\left(\frac{n^{(c-3/2)}/4-1}{C}\right)^{l^*-1}-1\right).\] As this quantity increases with $l^*$,  the maximum is reached for $l^*=L-1$. 
By applying the mixing time lower bound in \eqref{eq:sandwich} using the spectral gap, we obtain
$$
\tau_{\epsilon}\geq \log\left(\frac{1}{2\epsilon}\right)\frac{1}{2}\left[\frac{1}{Gap(P)}-1\right]> \log\left(\frac{1}{2\epsilon}\right)\frac{1}{4} \left[\left(\frac{n^{(c-3/2)}/4-1}{C}\right)^{L-2} -3\right].
$$

\section{Proof of Theorem \ref{eq:theo1} (Bayesian CART Mixing Upper Bound)}\label{pf:theo1}

\subsection{Proof of Lemma \ref{lem:classic_mE}} \label{sec:lem_proofs1}
We want to upper bound  the length of the longest canonical path constructed in Section \ref{sec:canonical}.
Let us first  bound $|T_{\mT,\mT^*}|$ when $\mT \supset\mT^*$.  In order to reach $\mT^* $ from $\mT$ on a canonical path, we remove  one redundant node at a time. There are at most $2^L$ nodes of which
$(2^L-|\mT_{int}^*|)$ are redundant. 
Thereby, we have $\max\limits_{\mT:\mT\supset\mT^*}\{|T_{\mT,\mT^*}|\}\leq
(2^L-|\mT_{int}^*|)$.  Conversely, for any $\mT\subset\mT^*$, the canonical path from $\mT$ towards $\mT^*$ adds one node in $\mT^*_{int}\backslash\mT_{int}$ at a time. 
This means $\max\limits_{\mT:\mT\subset\mT^*}|T_{\mT,\mT^*}|\leq
 |\mT_{int}^*|$. When $\mT\not\subset\mT^*$ and $\mT\not\supset \mT^*$, the path from $\mT$ towards $\mT^*$ follows by first deleting redundant nodes and then adding nodes towards reaching $\mT^*$. This can be achieved in at most $(2^L-|\mT^*_{int}|+|\mT^*_{int}|)$ steps. Finally, for any two trees $\mT,\mT'\in\bT$ the canonical path $T_{\mT,\mT'}$ is obtained by collapsing $T_{\mT,\mT*}$ and $ \bar{T}_{\mT',\mT^*}$.
Thereby, we have $\max_{\mT,\mT'\in\bT}|T_{\mT,\mT'}|\leq  2^{L+1}$.

\subsection{Proof of Lemma \ref{lemma:conduct}}\label{sec:lem2_proof}

We will work conditionally on the set $\mA_n$ defined in \eqref{eq:event_set}, where $p=2^{\Lmax}=n/2$. We know that the complement of this set has a vanishing probability $\P(\mA_n^c)\leq 2/ p\rightarrow 0$. 
We denote by $T_{\mT,\mT'}\in\mathcal E$ a canonical path between two nodes $\mT,\mT'\in\bT$. We will  find an upper bound for the congestion parameter $\rho(\mE)$ defined in \eqref{eq:congestion} as
$$
  \rho(\mE)=\max_{e\in\mE}\frac{1}{Q(e)}\sum_{(\bar \mT,\bar {\mT}'):e\in T_{\bar \mT,\bar {\mT}'}}\Pi(\bar \mT\C Y)\Pi(\bar {\mT}'\C Y),
$$
where for an edge $e$ between $(\mT,\mT')$ we have 
$$
Q(e)\equiv Q(\mT, \mT')=\Pi(\mT\C Y)P(\mT,\mT').
$$
First, we denote with 
\begin{equation}\label{eq:precedent}
\Delta(\mT')=\{{\mT}:\mT'\in T_{\mT,\mT^*}\}
\end{equation}
a set of precedents of a tree $\mT'$ that lie on a canonical path towards $\mT^*$. Note that $\mT'\in\Delta(\mT')$. For any given  edge $e_{\mT,\mT'}=T_{\mT,\mT'}$ between two {\em adjacent} trees     $\mT$  and  $\mT'$ where $\mT\in\Delta(\mT')$, we have
$$
N(e)\equiv\{(\bar \mT,\bar {\mT}')\C e\in T_{\bar\mT,\bar{\mT}'}\}\subset\Delta(\mT)\times\bT.
$$ 
Then we can find an upper bound for the congestion parameter in \eqref{eq:congestion} as 
\begin{equation}\label{eq:bound_rho}
\rho(\mE)\leq \max_{e_{\mT,\mT'}\in\mE}\frac{\Pi[\Delta(\mT)]}{Q(e_{\mT,\mT'})} \leq  \max_{(\mT,\mT')\in\bG^*}\frac{\Pi[\Delta(\mT)]}{Q(\mT,\mT')},
\end{equation}
where 
$$
\bG^*=\{(\mT,\mT')\in\bT\times\bT\C e_{\mT,\mT'}=T_{\mT,\mT'}\quad\text{and}\quad \mT\in\Delta(\mT')\}
$$
Now we find a lower bound for $Q(\mT,\mT')$ for an {\em adjacent} pair $(\mT,\mT')$ such that $\mT\in\Delta(\mT')$ or, equivalently, for $\mT'=\mG(\mT)$, where $\mG(\cdot)$ is the mapping introduced in Section \ref{sec:canonical}. For the ``lazy"  walk explained in Section \ref{sec:mixing} with a transition matrix $P=\wt P/2+I/2$ where $\wt P$ is the original transition matrix, we have
\begin{align*}
Q(\mT,\mT')&=\frac{1}{2}\Pi(\mT\C Y)\wt P(\mT,\mT')=\frac{1}{2}\Pi(\mT\C Y)S(\mT\rightarrow\mT') \min\left\{1, \frac{\Pi(\mT'\C Y)S(\mT'\rightarrow\mT)}{\Pi(\mT\C Y)S(\mT\rightarrow\mT')}\right\}.
\end{align*}
Plugging this into \eqref{eq:bound_rho} we obtain
\begin{equation}\label{eq:ratios}
\rho(\mE)\leq  2 \max_{(\mT,\mT')\in\bG^*} \left\{\frac{\Pi[\Delta(\mT)]}{\Pi(\mT\C Y)S(\mT\rightarrow\mT')}\times \max\left[1, \frac{\Pi(\mT\C Y)S(\mT\rightarrow\mT')}{\Pi(\mT'\C Y)S(\mT'\rightarrow\mT)}\right]\right\}.
\end{equation}
We now bound the ratio $\rho(\mE)$ assuming that $\mT$ is either {\em overfitting} or {\em underfitting}. We continue using the notation $\bT_O=\{\mT: \mT\supset\mT^*\}$ and $\bT_U=\{\mT: \mT\not\supset\mT^*\}$.



\subsubsection{When $\mT^*\subset \mT$ (The Overfitted Case)} 
When $\mT\in\bT_O$ subsumes the tree $\mT^*$, the mapping $\mG(\cdot)$  picks the deepest rightmost internal node, say $(l_S,k_S)\in\mT_{int}\backslash \mT^*_{int}$,  and turns it into a bottom node.
We denote with $\mT^-=\mG(\mT)$ such a pruned tree. We also define the collection of {\em pre-terminal nodes} of a tree $\mT\in\bT$ as
$$
\mP(\mT)=\{(l,k)\in\mT_{int}: \{(l+1,2k),(l+1,2k+1)\}\in\mT_{ext}\},
$$ 
 i.e. these are  internal nodes whose children are the bottom nodes. Using the posterior ratio expression in \eqref{eq:ratio_post_overfit} and \eqref{eq:ratio_post_overfit2} for overfitted trees with $K=1$ we obtain (for $j=2^{l_S}+k_S+1$)
$$
\frac{\Pi(\mT\C Y)}{\Pi(\mT^-\C Y)} =(1+n)^{-1/2} \frac{p_{l_S }}{1-p_{l_S }}\times (1-p_{l_S+1})^2\times \exp\left\{ \frac{  |X_{[j]}' Y|^2}{2(n+1)} \right\}\leq  \e^{-(c-3/2)\log n}
$$
{ Since we cannot preclude that $\mT = \mT_{full}^L$,} the proposal ratio satisfies
$$
\frac{S(\mT\rightarrow\mT^-)}{S(\mT^-\rightarrow\mT)}\leq  \frac{2|\mT^-_{ext}|}{|\mP(\mT)|}<4.
$$ 
Then   the ratio inside the Metropolis-Hastings acceptance probability  in \eqref{eq:ratios} satisfies 
\begin{equation}
\frac{\Pi(\mT\C Y)S(\mT\rightarrow{ \mT^-})}{\Pi({\mT^-}\C Y)S({\mT^-}\rightarrow\mT)}\leq 4\e^{ -(c-3/2)\log n}=o(1)\quad\text{for $ c>3/2$}.\label{eq:ratio_acceptance1}
\end{equation}
We now focus on the second  ratio in the product in \eqref{eq:ratios}.
When $\mT^*\subset\mT$,  all precedents ${\mT'}\in\Delta(\mT)$ (recall the definition of $\Delta(\mT)$ in \eqref{eq:precedent}) are also {\em overfitted} models, i.e. $\mT^*\subset{\mT'}$ and $\mT\subset\mT'$ for all ${\mT'}\in\Delta(\mT)\backslash\{\mT\}$.
We decompose $\Delta(\mT)=\cup_{K=1}^{2^L}\Delta(\mT,K)$ into shells depending how many steps away each tree $\mT'\in\Delta(\mT)$ is on a canonical path towards $\mT$.
Namely, for $K\in\N$, we denote with  
\begin{equation}
\Delta(\mT,K)=\{{\mT'}\in \Delta(\mT): |T_{\mT',\mT}|=K\}\label{eq:delta_shell}
\end{equation}
the set of precedents that are $K$ steps away from $\mT$ on some canonical path. Using again the posterior ratio for overfitted models in  \eqref{eq:ratio_post_overfit} and \eqref{eq:ratio_post_overfit2}, we obtain
for $\mT'\in\Delta(\mT,K)$
$$
\frac{\Pi(\mT'\C Y)}{\Pi(\mT\C Y)}\leq \e^{-K(c-3/2)\log n}.
$$
Moreover, the cardinality of $\Delta(\mT,K)$  for $K\geq 1$ satisfies $\mathrm{card}[\Delta(\mT,K))]\leq\prod_{j=1}^{K}(|\mT_{ext}|+j-1)$ and $\mathrm{card}[\Delta(\mT,0))]=1$.
Thereby
\begin{align}
\frac{\Pi[\Delta(\mT)\C Y]}{\Pi(\mT\C Y)}&=1+ \sum_{K=1}^{2^L}
\sum_{{\mT'}\in\Delta(\mT,K)} \frac{\Pi({\mT'}\C Y)}{\Pi(\mT\C Y)}\nonumber\leq 1+\sum_{K=1}^{2^L}\e^{ K\log [(|\mT_{ext}|+K-1)]}  \e^{- (c-3/2)\,K \log n}\\
&<1+\frac{1}{n^{(c-5/2)}-1}\label{eq:frac}.
\end{align}
Finally, because $\mT\neq \mT_{null}$, we have from \eqref{eq:proposal}
\[S(\mT\rightarrow{\mT^-}) = \frac{1}{2} \frac{1}{|\mP(\mT)|}.\]
Then we obtain
$$
 \frac{\Pi[\Delta(\mT)\C Y]}{\Pi(\mT\C Y)S(\mT\rightarrow{\mT'})}\times \max\left[1, \frac{\Pi(\mT\C Y)S(\mT\rightarrow{\mT'})}{\Pi({\mT'}\C Y)S({\mT'}\rightarrow\mT)}\right] 
\leq 2|\mP(\mT)|\left(1+\frac{1}{n^{(c-5/2)}-1}\right).
$$

\subsubsection{When $\mT^*\not\subset \mT$ (The Underfitted Case)}\label{sec:canonical_underfit}
We consider two cases of underfitting: {(1) when $\mT\not\subset\mT^*$ and, at the same time, $\mT\not\supset\mT^*$ and (2) when $\mT\subset\mT^*$. } First, if the tree underfits and contains extra nodes, those are deleted first which coincides with the previous case.

We now focus on the second case when $\mT\subset\mT^*$. Then $\mG(\mT)$ proceeds by adding an additional node towards completing $\mT^*$. 
We denote the resulting enlarged tree by $\mT^+=\mG(\mT)$. Using the expression of posterior ratio in \eqref{eq:posterior_ratio_overfit} and  \eqref{eq:posterior_ratio_overfit2} with $K=1$ we find that
$$
\frac{S(\mT\rightarrow\mT^+)}{S(\mT^+\rightarrow\mT)}\frac{\Pi(\mT\C Y)}{\Pi(\mT^+\C Y)} \leq\frac{{ 2}|\mathcal P(\mT^+)|}{|\mT_{ext}|} \e^{-A^2/8\log^2 n}=o(1).
$$
Next, note that precedents $\Delta(\mT)$ in \eqref{eq:precedent} of an underfitted model $\mT$ are also underfitted  models and can be divided into two (besides the singleton set $\{\mT\}$) mutually exclusive categories, i.e. $\Delta(\mT)=\{\mT\}\cup \mU_1(\mT)\cup \mU_2(\mT)$.
The first set, denoted with $\mU_1(\mT)$, consists of all precedents   $\Delta(\mT)$ that are also subsets of $\mT^*$  , i.e. 
$\mU_1(\mT)\equiv \{{\mT'}\in\Delta(\mT):{\mT'}\subset\mT^*\}$. 
The second set, denoted with $\mU_2(\mT)$, are all the precedents that have some redundant nodes and are {\em not} included in $\mT^*$, i.e. $\mU_2(\mT)=\{{\mT'}\in\Delta(\mT): {\mT'}\not\subseteq\mT^*\}$.
We denote with $\Delta(\mT,K)\subset\Delta(\mT)$ those precedents that are $K$ steps away  from $\mT$ on a canonical path (i.e.  all trees inside $\mU_1(\mT)$ that have $K$ fewer internal nodes compared to $\mT$ 
and all trees inside $\mU_2(\mT)$ that have $K$ extra internal nodes compared to $\mT$), where the cardinality satisfies $\mathrm{card}[\Delta(\mT,K)]\leq 2^{LK}$. Because under the Assumption \ref{ass:f_classic} (a), {\em all} internal nodes in  $\mT^*$ are signals, we can modify the expressions in \eqref{eq:posterior_ratio_overfit} to include all $K$ signals (not just one) to obtain for large enough $A$
\begin{align}
\frac{\Pi[\mU_1(\mT)\C Y]}{\Pi(\mT\C Y)}&\leq\sum_{K=1}^{|\mT^*_{int}|}\sum_{{\mT'}\in\Delta(\mT,K)\cap\mU_1(\mT)}\frac{\Pi({\mT'}\C Y)}{\Pi(\mT\C Y)}
< \sum_{K=1}^{|\mT^*_{int}|}\e^{-KA^2/8 \log^2 n}<\frac{1}{n^{A^2/8\log n}-1}.\label{eq:bound_U1}
\end{align}
We now consider the second type of underfitting precedents  $\mU_2(\mT)$. For each such $\mT'\in \mU_2(\mT)$, the canonical path $T_{{\mT'},\mT}$ first proceeds by removing redundant nodes 
and at some point reaches a tree $U(\mT')$ which already underfits. In other words, $U(\mT')\in \mU_1$ is defined as the largest subtree obtained from ${\mT'}$ by removing all redundant branches (without signal). 
This means that $U({\mT'})$ is the largest tree that satisfies $U({\mT'})\subset{\mT'}$ and, at the same time, $U({\mT'})\subset \mT^*$. The mapping ${\mT'}\rightarrow U({\mT'})$ is many-to-one and for any $\wt\mT\in\mU_1(\mT)$ such that there exists ${\mT'}\in\mU_2(\mT)$ so that $U({\mT'})=\wt\mT$ we have 
$$
\mathcal N(\mT,\wt\mT)\equiv\{{\mT'}\in\mU_2(\mT):U({\mT'})=\wt\mT \}\subseteq \bT_O(\wt\mT),
$$
where $\bT_O(\wt\mT)=\{\mT:\wt\mT\subset\mT\}$ are all trees that contain $\wt\mT$. Using the same logic as in \eqref{eq:bound_overtit_by_k} and \eqref{eq:bound_overfit} we find that 
$$
\frac{\Pi[\mathcal N(\mT,\wt\mT)\C Y]}{\Pi(\wt\mT\C Y)}\leq  \frac{1}{n^{c-5/2}-1}
$$
and thereby using \eqref{eq:bound_U1}
\begin{align*}
\frac{\Pi[\mU_2(\mT)\C Y]}{\Pi(\mT\C Y)}
&=\sum_{{\mT'}\in \mU_2(\mT)}\frac{\Pi[U({\mT'})\C Y]}{\Pi(\mT\C Y)}\frac{\Pi({\mT'}\C Y)}{\Pi[U({\mT'})\C Y]}\\
&=\sum_{\substack{\wt\mT\in U_1(\mT):\\\mathcal N(\mT,\wt\mT)\neq \emptyset}}\,\,\sum_{{\mT'}\in\mathcal N(\mT,\wt\mT)}
\frac{\Pi(\wt\mT\C Y)}{\Pi(\mT\C Y)}\frac{\Pi({\mT'}\C Y)}{\Pi(\wt\mT\C Y)}\\
&\leq \sum_{\substack{\wt\mT\in U_1(\mT):\\\mathcal N(\mT,\wt\mT)\neq \emptyset}}\frac{\Pi(\wt\mT\C Y)}{\Pi(\mT\C Y)} \,\,\sum_{{\mT'}\in \bT_O(\wt\mT)}\frac{\Pi({\mT'}\C Y)}{\Pi(\wt\mT\C Y)} \\
&\leq { \frac{1}{n^{c-5/2}-1}}\sum_{\substack{\wt\mT\in U_1(\mT):\\\mathcal N(\mT,\wt\mT)\neq \emptyset}}
\frac{\Pi(\wt\mT\C Y)}{\Pi(\mT\C Y)} \\
&\leq  { \frac{1}{n^{c-5/2}-1}}\frac{\Pi[\mU_1(\mT)\C Y]}{\Pi(\mT\C Y)}<
 { \frac{1}{n^{c-5/2}-1}}\times\frac{1}{n^{A^2/8\log n}-1}
\end{align*}
Putting it all together, we have 
$$
\frac{\Pi[\Delta(\mT)\C Y]}{\Pi(\mT\C Y)S(\mT\rightarrow\mT^+)}\leq 2 |\mT_{ext}|\left(1+o(1)\right).
$$
The bound for second underfitting case (b) when $\mT\not\subset\mT^*$ and, at the same time, $\mT\not \supset\mT^*$ proceeds analogously, only without the set $\mU_1(\mT)$ that is empty.
 
\smallskip
Putting it all together, and noting that $|\mathcal P(\mT)|\leq 2^L$ and $|\mT_{ext}|\leq 2^L$, the bound in \eqref{eq:ratios} yields $\rho\leq 2^{L+1}(1+o(1))$ for $ c>5/2$.

\section{Proof of Theorem \ref{eq:theo1}}\label{seq:proof_thm_theo1}
We start with the sandwich relation  \eqref{eq:sandwich} and find a lower bound to $\min\limits_{\mT\in\bT}\Pi(\mT\C Y)$. 
By   consistency established in Lemma \ref{lemma:consist}, for any $\mT\in \bT$ we have (with probability at least $1-4/n$)
\begin{align*}
\Pi(\mT\C Y)&= \Pi(\mT^*\C Y)\frac{\Pi(\mT\C Y)}{\Pi(\mT^* \C Y)}\geq\frac{1}{2}\frac{\Pi(\mT\C Y)}{\Pi(\mT^*\C Y)}.
\end{align*} 
We will again split the eligible trees $\bT$ into overfitted $\bT_O$ and underfitted $\bT_U$.
For any tree $\mT\in \bT_O$ with $K$ extra internal nodes, we know from  Section \ref{sec:do_not_overfit}   that (using the shorthand notation $p_l=p_{lk}$)
\begin{align} 
\frac{\Pi(\mT\C Y)}{\Pi(\mT^*\C Y)}&=(1+n)^{-K/2}\prod_{j=1}^K\left(\frac{p_{l_j}}{1-p_{l_j}}\times (1-p_{l_j+1})^2\right)\times \exp\left\{ \frac{  |X_{[j]}' Y|^2}{2(n+1)} \right\}.\label{eq:subset_tree_ratio}
\end{align}
With $p_l=n^{-c}<1/2$ we obtain 
\begin{equation}
\min_{\mT\in \bT_O}\Pi(\mT\C Y)>\frac{1}{2}\min_{\mT\in \bT_O}\frac{\Pi(\mT\C Y)}{\Pi(\mT^*\C Y)}>\frac{1}{2} \left(\frac{1}{2n^c\sqrt{1+n}}\right)^{K}>\frac{ \e^{-\frac{n}{2}[\log 2+(c+1/2)\log (1+n)]}}{2}. \label{eq:overfit_minT}
\end{equation}

Similarly as in Section \ref{sec:canonical_underfit}, we consider two under-fitted cases $\mT\in\bT_U$. First, assume that $\mT\in\bT_U$ and at the same time $\mT\subset\mT^*$. This means that $\mT$ misses at least one signal node, e.g. $(l_S,k_S)\in\mathcal B(A)$.  
We denote with $\mT^0=\mT\rightarrow\mT^1\rightarrow\dots\rightarrow\mT^K=\mT^+$ the sequence of nested trees obtained by adding one additional internal node $(l_j,k_j)$ towards $(l_S,k_S)$.
As in Section \ref{sec:do_not_underfit}, we find that
\begin{align}
\frac{N_Y(\mT)}{N_{Y}(\mT^+)} &=(1+n)^{K/2}\prod_{j=1}^K\exp\left\{- \frac{  |X_{[j]}' Y|^2}{2(n+1)} \right\},
\end{align}
We have 
$$
 |X_{[j]}' Y|^2=|X_{|j|}'(X_{\mT^*}\b^*_{\mT^*} +\bm\varepsilon)|^2\leq 2n^2|\beta^*_{l_jk_j}|^2 +{ 8}n\,\log n 
$$
and thereby
$$
\frac{N_Y(\mT)}{N_{Y}(\mT^+)}>(1+n)^{K/2}\exp\left\{-\frac{n^2\, \|\b^*_{\mT^+\backslash \mT}\|^2}{n+1}   - \frac{{ 4}n\,K \log n}{n+1}\right\} 
$$
If $\mT^+=\mT^*$ we stop tree growing, otherwise we repeat the same process with $\mT^+$, extending a branch towards missing signal to create $\mT^{++}$. We stop after $M$ steps where $\mT^{+...+}=\mT^*$
We then bound
\begin{align}
\frac{N_Y(\mT)}{N_{Y}(\mT^*)} &= \frac{N_Y(\mT)}{N_{Y}(\mT^+)}\times \frac{N_Y(\mT^+)}{N_{Y}(\mT^{++})}\times\dots\times \frac{N_Y(\mT^{+...+})}{N_{Y}(\mT^{^*})}\\
&>(1+n)^{1/2}\exp\left\{-\frac{n^2\, \|\b^*_{\mT^*}\|^2}{n+1}   - \frac{{ 4}n\,|\mT^*_{int}|\log n}{n+1}\right\}.
\end{align}
The prior ratio satisfies (with $p_l=n^{-c}<1/2$)
\begin{equation}
\frac{\Pi(\mT)}{\Pi(\mT^+)}=\frac{1-p_{l_1}}{p_{l_1}}\times\left(\prod_{j=2}^{K-1}\frac{1}{p_{l_j}(1-p_{l_j})}\right)\times\frac{1}{(1-p_{l_K})^2}>n^{c (K-1) } 
\end{equation}
This yields
\begin{equation}
\min_{\mT\in \bT_U:\mT\subset\mT^*}\Pi(\mT\C Y)>(1+n)^{1/2}\exp\left\{-\frac{n^2\, \|\b^*_{\mT^*}\|^2}{n+1}   - \frac{{ 4}n\,|\mT^*_{int}|\log n}{n+1}\right\}.\label{eq:lower_bound_underfit}
\end{equation}
Now we focus on the under-fitted trees that are not necessarily contained inside $\mT^*$. Consider  $\mT\in\bT_U$ such that $\mT\not\subset\mT^*$.
Then we combine growing and pruning operations from the previous steps. First, we prune the tree $\mT$ into the largest tree $\mT_U$ that underfits, i.e. $\mT_U$ is the largest tree such that $\mT_U\in\bT_U$ and $\mT_U\subset\mT$,
and write
$$
\frac{\Pi(\mT\C Y)}{\Pi(\mT^*\C Y)}=\frac{\Pi(\mT\C Y)}{\Pi(\mT^U\C Y)}\frac{\Pi(\mT^U\C Y)}{\Pi(\mT^*\C Y)}
$$
and combining the expression \eqref{eq:overfit_minT} with \eqref{eq:lower_bound_underfit}  we find
\begin{align*}
\min_{\mT\in\bT_U:\mT\not\subseteq\mT^*}\frac{\Pi(\mT\C Y)}{\Pi(\mT^*\C Y)}&> \min_{\mT\in\bT_U:\mT\not\subseteq\mT^*}\frac{\Pi(\mT\C Y)}{\Pi(\mT_U\C Y)}\times \min_{\mT\in\bT_U:\mT\subset\mT^*}\frac{\Pi(\mT\C Y)}{\Pi(\mT^*\C Y)}\\
&> \frac{1}{2} \exp\left\{{-n\left[1+{\left(c+\frac{1}{2}\right)}\log (1+n)\right]} -\frac{n^2\,  |\mT^*_{int}|C_{f_0}^2}{n+1}   - \frac{{ 4}n\,|\mT^*_{int}|\log n}{n+1} \right\}
\end{align*}
Now, by Lemma \ref{lem:classic_mE} we have $l(\mE)\leq 2^{L+1}$  and 
by Lemma \ref{lemma:conduct} we have $\rho(\mE)\leq 2^{L+1}[1+o(1)]$
for $c>1$ on the event space $\mA_n$. Plugging these into \eqref{eq:sandwich}, 
we obtain
\begin{align*}
\tau_{\epsilon}&\leq {l(\mE)\rho(\mE)}\big({\log\left[\frac{1}{\min_{\mT\in\bT}\Pi(\mT\C Y)}\right]+\log (1/\varepsilon)}\big)\\
&\leq 2^{2(L+1)+1} \left\{ n\left[\left(c+\frac{1}{2}\right)\log (1+n) +  |\mT^*_{int}|C_{f_0}^2+1\right]   + 4\,|\mT^*_{int}|\log n+ \log \left(\frac{2}{\epsilon}\right) \right\}.
\end{align*}

\section{Proof of Theorem \ref{theo:patul} (Twiggy Bayesian CART Mixing Upper Bound)}\label{sec:proof_theo_patul}

We follow the same recipe as in the proof of Theorem \ref{eq:theo1}. We first need to show Lemma \ref{lem:classic_mE} and Lemma \ref{lemma:conduct} for the canonical path ensemble for Twiggy Bayesian CART constructed in Section \ref{sec:canonical_twiggy} below.

\subsubsection{Canonical Path Ensemble for Twiggy Bayesian CART}\label{sec:canonical_twiggy}
We again construct a canonical path $T_{\mT,\mT^*}$ between any $\mT\in\bT\backslash{\mT^*}$ and the spanning tree $\mT^*$ from Assumption \ref{ass:f_classic}. 
Recall the definition of signals $\mB(A)=\{(l,k):C_{f_0}> |\beta_{lk}^*|>A\log n/\sqrt{n}\}$, where $\mT^*=\mB(A)$ under Assumption  \ref{ass:f_classic} (a) and $\mB(A)\subseteq \mT^*$
Assumption \ref{ass:f_classic} (b).
The transition function $\mG(\mT)$ for Twiggy Bayesian CART is defined as follows:
\begin{itemize}
\item[(1)] Assume $\mT\supset \mT^*$ is {\bf overfitted},  i.e. $\mT$ forms an envelope around  $\mT^*$ and contains at least one redundant node. 
Denote  the set of all redundant internal nodes whose descendants form a twig  as
$$
S(\mT)\equiv \left\{ (l,k) \in\mT_{int}\backslash\mT^*_{int}: \exists (l^*,k^*)\in\mT_{int}\,\,s.t.\,\, D_{lk}(\mT)=[(l, k) \leftrightarrow (l^*,k^*)] \right\}.
$$
Note that this set contains all pre-terminal nodes.  The mapping $\mG(\cdot)$ finds the most shallow leftmost node inside $S(\mT)$, say $(\wt l,\wt k)$, and turns it into a bottom node with the twig below removed. More formally,  
we define  $\mG(\mT)=\mT^-$ as
 \begin{equation}
\mT^-_{int}=\mT_{int}\,\backslash\, [(\wt l,\wt k)\leftrightarrow (l^*,k^*)]\quad\text{where}\quad (\wt l,\wt k)=\arg\min \limits_{(l,k)\in S(\mT)}(2^{l}+k). 
 \end{equation}
Picking the shallowest (as opposed to deepest) node for removal gives us an opportunity to remove more than one node at a time, thereby shortening the path towards $\mT^*$.
 
\item[(2)] Assume $\mT\not\supseteq \mT^*$ is {\bf underfitted}, i.e. $\mT$ misses at least one  node in $\mT^*$.
\begin{itemize}
\item[(i)] If $\mT\subset \mT^*$, the mapping $\mG(\cdot)$ finds the deepest rightmost  node inside $\mT^*$ missed by $\mT$, say $(l^*,k^*)$, and grows a twig towards it. 
More formally, we define $\mT^+=\mG(\mT)$ where
$$
\mT^+_{int}=\mT_{int}\cup[(\wt l, \wt k)\leftrightarrow (l^*,k^*)]\quad\text{where}\quad (l^*,k^*)=\arg\max_{(l,k)\in\mT^*_{int}\backslash \mT_{int}}(2^l+k)
$$ 
and where $(\wt l, \wt k)$ is the closest node to $(l^*,k^*)$ inside $\mT_{ext}$. Note that taking the deepest rightmost node gives us an opportunity to add more than one signal node at a time.

\item[(ii)] If $\mT\not\subset \mT^*$, i.e., $\mT$ contains redundant nodes and the mapping $\mG(\cdot)$ is the same as in the overfitting the case (1).
\end{itemize}

\end{itemize}
It is easy to see that this transition function  reduces the Hamming distance after each step. Compared to Bayesian CART, however, it may take larger leaps.
It can be shown that the canonical path  ensemble for Twiggy Bayesian CART satisfies the statements of Lemma \ref{lem:classic_mE} and Lemma \ref{lemma:conduct} for the unstructured signal Assumption \ref{ass:f_classic} (b) (see proof of Theorem  \ref{theo:patul} in Section \ref{sec:proof_theo_patul}).

\subsection{Version of Lemma \ref{lem:classic_mE} for Twiggy Bayesian CART}  
The proof is similar to the Bayesian CART version. 
Let us first  bound $|T_{\mT,\mT^*}|$ when $\mT \supset\mT^*$.  In order to reach $\mT^* $ from $\mT$ on a canonical path, we remove  at least one redundant node at a time.
There are at most $2^L$ nodes of which $(2^L-|\mT_{int}^*|)$ are redundant.  
Thereby, we have $\max\limits_{\mT:\mT\supset\mT^*}\{|T_{\mT,\mT^*}|\}\leq
(2^L-|\mT_{int}^*|)$. Using a more complicated argument,  one could take advantage of removals of entire twigs to show that the removal can be achieved in up to  $|\mathcal P(\mT)\backslash \mT_{int}^*|$ steps. Conversely, for any $\mT\subset\mT^*$, the canonical path from $\mT$ towards $\mT^*$ adds a twig towards a node in $\mathcal P(\mT^*)\backslash\mT_{int}$ at a time. 
This means $\max\limits_{\mT:\mT\subset\mT^*}|T_{\mT,\mT^*}|\leq
 |\mathcal P(\mT^*)|$. When $\mT\not\subset\mT^*$ and $\mT\not\supset \mT^*$, the path from $\mT$ towards $\mT^*$ follows by first deleting redundant nodes and then adding nodes towards reaching $\mT^*$. This can be achieved in at most $(2^L-|\mT^*_{int}|+|\mathcal P(\mT)|)$ steps. Finally, for any two trees $\mT,\mT'\in\bT$ the canonical path $T_{\mT,\mT'}$ is obtained by collapsing $T_{\mT,\mT*}$ and $ \bar{T}_{\mT',\mT^*}$.
Thereby, we have $\max_{\mT,\mT'\in\bT}|T_{\mT,\mT'}|\leq  2^{L+1}$. The bound can be sharpened to  $2^L$ using a more complicated argument.

\subsection{Version of Lemma \ref{lemma:conduct} for Twiggy Bayesian CART}  
We will again work on the event space $\mA_n$ in \eqref{eq:event_set}, which has probability at least $1-4/n $. The strategy is the same as in the proof of Lemma  \ref{lemma:conduct}.
We again split the considerations into overfitted and underfitted trees.
 
\subsubsection{When $\mT^*\subset \mT$ (The Overfitted Case)} 

When $\mT$ subsumes the tree $\mT^*$, the mapping $\mG(\cdot)$ finds the shallowest leftmost node inside $\mT_{int}\backslash\mT_{int}^*$, say $(l,k)$, such that the entire branch below $(l,k)$ is a twig, and removes  the twig, turning $(l,k)\in\mT_{int}$ into an external node. In other words, $(l,k)$ has been converted to a bottom node and its descendants erased. More formally,  
recall the definition of ancestors of $(l,k)$ inside $\mT$ as 
$$
A_{lk}(\mT)=\{(l',k')\in\mT_{int}:\exists j\in \{0,1,\dots, L-1\}\,\,s.t.\,\, (l',k')=(l-j,\lfloor k/2^j\rfloor) \}
$$ 
and descendants of $(l,k)$ as
$$
D_{lk}(\mT)=\{(l',k')\in\mT_{int}: (l,k)\in A_{l'k'}(\mT)\}.
$$
Moreover, $(l,k)$ is such that $\exists (\wt l,\wt k)\in\mT_{int}$ such that  $D_{lk}=[(l,k)\leftrightarrow (\wt l,\wt k)]$.
Writing $\mT^-=\mG(\mT)$, we have
$$
\mT^-_{int}=\mT_{int}\,\backslash\, [(l,k)\leftrightarrow (\wt l,\wt k)]. 
$$
We now provide bounds for the two terms in \eqref{eq:ratios}. We denote the length (number of nodes) of the twig $[(l,k)\leftrightarrow (\wt l,\wt k)]$ by $k$. This means that 
$\mT^-_{int}$ has $k$ fewer internal nodes compared to $\mT$ and from the construction all of them are signal-less nodes. With the proposal distribution described in Section \ref{sec:twiggy}, {and since we cannot preclude that $\mT=\mT_{full}^L$,} we have  
$$
\frac{S(\mT\rightarrow\mT^-)}{S(\mT^-\rightarrow\mT)}\leq  \frac{2^{L}\frac{D^L-1}{D-1}}{|\mP(\mT)|}\leq n\, \frac{D^L-1}{D-1}.
$$
Using the posterior ratio expression in \eqref{eq:ratio_post_overfit} and \eqref{eq:ratio_post_overfit2} for overfitted trees  we obtain
$$
\frac{\Pi(\mT\C Y)}{\Pi(\mT^-\C Y)}\frac{S(\mT\rightarrow\mT^-)}{S(\mT^-\rightarrow\mT)}\leq{ {n\,{ \frac{D^L-1}{D-1}}\e^{- k(c-3/2)\log n}}}\leq {\frac{D^L-1}{D-1}}\e^{ -k(c-5/2)\log n}.
$$
 This means that the second maximum quantity  in  \eqref{eq:ratios} is {no greater than a constant multiple of { $\e^{-(c-5/2-\log D)\log n}$ which is $o(1)$ for $c>5/2+\log D$.}
We now focus on the first  ratio in the product in \eqref{eq:ratios}. When $\mT^*\subset\mT$,  all precedents ${\mT'}\in\Delta(\mT)$ (recall the definition of $\Delta(\mT)$ in \eqref{eq:precedent}) are also {\em overfitted} models, i.e. $\mT^*\subset{\mT'}$ and $\mT\subset\mT'$ for all ${\mT'}\in\Delta(\mT)\backslash\{\mT'\}$. Similarly as in the proof of Lemma \ref{lemma:conduct} in Section \ref{sec:lem2_proof}, we decompose $\Delta(\mT)=\cup_{K=1}^{2^L}\Delta(\mT,K)$ into shells $\Delta(\mT,K)=\{{\mT'}\in \Delta(\mT): |T_{\mT',\mT}|=K\}$ defined as in \eqref{eq:delta_shell}.  The difference now is that each tree $\mT'\in \Delta(\mT,K)$ can have {\em more than $K$ redundant nodes}. We denote with $\bm \kappa=(k(1),\dots,k(K))'\in(\N\backslash\{0\})^K$ the vector of numbers of redundant nodes deleted at each of the $K$ steps on the canonical path from $\mT'\in \Delta(\mT, K)$ towards $\mT^*$.
Using again the posterior ratio for overfitted models in  \eqref{eq:ratio_post_overfit} and \eqref{eq:ratio_post_overfit2}, we now obtain 
for $\mT'\in\Delta(\mT,K)$
$$
\frac{\Pi(\mT'\C Y)}{\Pi(\mT\C Y)}\leq \e^{(3/2-c)\sum_{j=1}^Kk(j)\log n}.
$$
Moreover, we define
$$
\Delta(\mT, K,\bm \kappa)=\{{\mT'}\in \Delta(\mT,K): |\mT^j_{int}\backslash\mT^{j-1}_{int}|=k(j),\,\, \forall j=1,\dots, K\}
$$ 
all the trees that are $K$ steps away from $\mT$ and that differ from $\mT$ by adding  exactly $k(j)$ nodes at each step.
When $K=1$, the number  of such precedents is at most the number of binary trees with $k(1)$ internal nodes. This corresponds to the Catalan number   $\mathbb C_{K}$, which
according to Lemma  S-3 in \citep{castillo2021uncertainty}, satisfies
$\mathbb C_{k(1)}\asymp 4^{k(1)}/{k(1)}^{3/2}$.
Then it is easy to see that for $K\geq 1$ we have
$$
\mathrm{card}[\Delta(\mT,K,\bm \kappa)]\leq \prod_{j=1}^K\e^{k(j)\log 4-3/2\log k(j)}.
$$
This yields (since $K\leq \sum_j k(j)\leq {2^L}$ and $ c>5/2$) for $n\geq 8$  
\begin{align*}
\frac{\Pi[\Delta(\mT)\C Y]}{\Pi(\mT\C Y)}&=1+ \sum_{K=1}^{{2^L}} \,\,\,\sum_{\bm \kappa: \sum_j k(j)\leq {2^L}}\,\,\,
\sum_{{\mT'}\in\Delta(\mT,K,\bm \kappa)} \frac{\Pi({\mT'}\C Y)}{\Pi(\mT\C Y)}\nonumber\\
&\lesssim 1+\sum_{K=1}^{{2^L}}\,\,\sum_{\bm \kappa: \sum_j k(j)\leq {2^L}} \frac{ \e^{\sum_j k(j)[\log 8- c\log n]}}{\prod_j k(j)^{3/2}}\\
&\leq 1+\sum_{K=1}^{{2^L}} \left(\frac{n}{2}\right)^K  \e^{-K[{(c-3/2)}\log n -\log 8]} \leq 1+\sum_{K=1}^{{2^L}}  \e^{-K[{(c-5/2)}\log n -\log 8]}\\&=1+ \frac{1}{n^{c-5/2}/8-1}
\end{align*}
and thereby 
$$	
\left\{\frac{\Pi[\Delta(\mT)\C Y]}{\Pi(\mT\C Y)S(\mT\rightarrow\mT^{-})}\times \max\left[1, \frac{\Pi(\mT\C Y)S(\mT\rightarrow\mT^{-})}{\Pi(\mT^{-}\C Y)S(\mT^{-}\rightarrow\mT)}\right]\right\}
\leq  2|\mT_{int}|(1+o(1)).
$$

\subsubsection{When $\mT^*\not\subset \mT$ (The Underfitted Case)}\label{subsub:underfit1}
If the tree underfits and contains extra nodes, those are deleted first which coincides with the previous case. For those underfitted trees such that $\mT\subset\mT^*$, the internal nodes $\mT_{int}$ do not include at least one pre-terminal node  $\mP(\mT^*)$. According to the Assumption \ref{ass:f_classic}, the pre-terminal nodes have large enough signal, where $|\beta_{lk}^*|>A\log n/\sqrt{n}$ for some $A>0$ for all $(l,k)\in\mP(\mT^*)$.
Denote with $(l,k)\in \mP(\mT^*)\backslash\mT_{int}$ the deepest rightmost signal pre-terminal node missed by $\mT$. Let $(l^*,k^*)\in\mT_{ext}$ be the  external node of $\mT$ that is closest to the signal node $(l,k)$. Then $\mG(\mT)$ is formed by growing   a twig $[(l^*,k^*)\leftrightarrow(l,k)]$. 
In other words,  $\mT^+=\mG(\mT)$ is the smallest tree that contains nodes $(l,k)\cup\mT_{int}$ inside and
$\mT^+_{int}=\mT_{int}\cup[(l^*,k^*)\leftrightarrow(l,k)]$ has $k$ more internal nodes relative to $\mT_{int}$.
Then we can write 
$$
\frac{S(\mT\rightarrow\mT^+)}{S(\mT^+\rightarrow\mT)}\leq  2|\mT^+_{int}| \leq n.
$$
Using the expression for the posterior ratio in \eqref{eq:posterior_ratio_overfit} and  \eqref{eq:posterior_ratio_overfit2}  we  again find that 
$$
\frac{\Pi(\mT\C Y)}{\Pi(\mT^+\C Y)}\frac{S (\mT\rightarrow\mT^+)}{S(\mT^+\rightarrow\mT)} { \leq}  n \e^{ -A^2/8\log^2 n}=o(1).
$$
We now proceed similarly as in Section \ref{sec:canonical_underfit}. 
The precedents $\Delta(\mT)$ in \eqref{eq:precedent} of an underfitted model $\mT$ are  again divided into mutually exclusive categories defined in Section \ref{sec:canonical_underfit}, i.e. $\Delta(\mT)=\{\mT\}\cup \mU_1(\mT)\cup \mU_2(\mT)$.
We again denote with $\Delta(\mT,K)\subset\Delta(\mT)$ those precedents that are $K$ steps away  from $\mT$ on a canonical path. 
Note that trees inside $\mU_1(\mT) \cap\Delta(\mT,K)$ have at least $K$ fewer internal nodes compared to $\mT$ 
and   trees inside $\mU_2(\mT) \cap\Delta(\mT,K)$   have at least $K$ extra internal nodes compared to $\mT$.

Each tree in $\Delta(\mT,K)\,\cap\, \mathcal U_1(\mT)$ misses $K$ preterminal nodes in $\mathcal P(\mT^*)$ (and thereby at least $K$ internal nodes relative to $\mT^*$). 
The cardinality   $\mathrm{card}[\Delta(\mT,K)\,\cap\, \mathcal U_1(\mT)]$  is thereby at most the number of binary trees with $|\mT_{int}^*|-K$ internal nodes which equals the Catalan number $\mathbb C_{|\mT_{int}^*|-K}$. By Lemma 6 in \cite{castillo2021uncertainty}, we have $\mathbb C_K\asymp 4^K/K^{3/2}$ and
using the expressions in \eqref{eq:posterior_ratio_overfit}  we obtain for large enough $A>0$ and $|\mT^*_{int}|\lesssim \log^2 n$  
\begin{align}
\frac{\Pi[\mU_1(\mT)\C Y]}{\Pi(\mT\C Y)}&\leq\sum_{K=1}^{|\mathcal P(\mT^*)|}\sum_{{\mT'}\in\Delta(\mT,K)\cap\mU_1(\mT)}\frac{\Pi({\mT'}\C Y)}{\Pi(\mT\C Y)}
\lesssim |\mT^*_{int}|\times 4 ^{|\mT^*_{int}|}\times \e^{-A^2/8 \log^2 n} =o(1). 
\end{align}
For the second type of underfitting precedents, we follow the same arguments as in Section  \ref{sec:canonical_underfit} to conclude
\begin{align*}
\frac{\Pi[\mU_2(\mT)\C Y]}{\Pi(\mT\C Y)} \leq  {\frac{1}{n^{c-5/2}/8-1}}\frac{\Pi[\mU_1(\mT)\C Y]}{\Pi(\mT\C Y)}
\end{align*}
and thereby 
\begin{align*}
\frac{\Pi[\Delta(\mT)\C Y]}{\Pi(\mT\C Y)S(\mT\rightarrow\mT^+)}&\leq {{ \frac{2^{L}(D^L-1)}{D-1}}}\,\left(1+o(1)\right).
\end{align*}

The bound for the second underfitting case   when $\mT\not\subset\mT^*$ and, at the same time, $\mT\not \supset\mT^*$ proceeds analogously, only without the set $\mU_1(\mT)$ that is empty. 

\smallskip
Putting it all together,   the bound in \eqref{eq:ratios} yields $\rho(\mE)\leq { \frac{(D^L-1)}{D-1}}2^{L+1}(1+o(1))$ for $c>{ 5/2}+\log D$.

\medskip
The mixing bound \eqref{eq:tau_bound2} for Twiggy Bayesian CART is then obtained by the sandwich relation  \eqref{eq:sandwich} with \eqref{eq:congestion} in the same way as in the proof of Theorem \ref{eq:theo1}.

\section{Proof of Theorem \ref{thm:coupling} (Mixing Upper Bound for Locally Informed Versions)}\label{pf:coupling}
\subsection{General two-stage drift condition} 
We will use a similar strategy as in  \cite{zhou2021dimension}. We first state the general two-stage drift condition theorem and its corollary, which are a slight modification from  \cite{zhou2021dimension}. 
\begin{theorem}\label{thm:drift2_modif}Consider a Markov chain $(X_t)_{t\in \N}$ on a state space $(\mathcal{X},\mathcal{E})$ where the $\sigma$-algebra $\mathcal{E}$ is countably generated. Assume a transition kernel $P$ that is reversible with respect to a stationary distribution $\pi$ and  that $P$ has a non-negative eigenspectrum. Suppose that there exist two drift functions $V_1,V_2:\mX\rightarrow[1,\infty)$ with constants $\lambda_1,\lambda_2\in (0,1)$, a set $A\in \mE$ and a point $x^*\in A$ such that 
\begin{enumerate}
\item[(i)] $PV_1\leq \lambda_1 V_1$ on $A^c$,
\item[(ii)] $PV_2\leq \lambda_2 V_2$ on $A\backslash \{x^*\}$, and 
\end{enumerate}
Further, suppose that $A$ satisfies the following conditions for some finite constants $K_1\geq 2$, $K_2\geq 1$.
\begin{enumerate}
\item[(iii)] For any $x\in A$, $V_1(x)\leq K_1/2$, and if $P(x,A^c)>0,~\E_x[V_1(X_1)|X_1\in A^c]\leq K_1/2$.
\item[(iv)] For any $x\in A$, $V_2(x)\leq K_2$, and if $P(x,A^c)>0,~\E_x[V_2(X_1)|X_1\in A^c]\geq V_2(x)$.
\item[(v)]For any $x\in A$, $P(x,A^c)\leq q$ for some constant $q<\min\{1-\lambda_1,(1-\lambda_2)/K_2\}$.
\end{enumerate}
Then, for every $x\in \mX$ and $t\in \N$, we have 
\[\|P^t(x,\cdot)-\pi\|_{TV} \leq 4\alpha^{t+1}\left(1+\frac{V_1(x)}{K_1}\right),\]
where $\alpha$ is a constant that satisfies
\[\alpha=\frac{1+\rho^m}{2}=\frac{1+K_1^m/u}{2}, ~\rho=\frac{qK_2}{1-\lambda_2},~u=\frac{1}{1-q/2},~ m=\frac{\log u}{\log (K_1/\rho)}.\]

\end{theorem}

\begin{corollary}\label{cor:drift2}
Recall the definition of the $\epsilon$-mixing time in \eqref{eq:tau_def}. In the setting of Theorem \ref{thm:drift2_modif}, assume that $\lambda_1,\lambda_2\rightarrow 1$ and $q\leq \min\{ 1-\lambda_1, (1-\lambda_2)/C_2K_2\}$ for some universal constant $C_2>1$. With $K_1 = 2\sup_{x\in \mX} V_1(x)$, for sufficiently large $n$, we have 
$$ 
\tau_\epsilon \lesssim \frac{4\log (6/\epsilon)}{\log C_2} \log (C_2 K_1) \max\left\{\frac{1}{1-\lambda_1},\frac{C_2K_2}{1-\lambda_2}\right\}.
$$

\end{corollary}

\subsubsection{Application of general two-stage drift condition} \label{sec:application_two_drift}
First, we introduce some notation. For $\mT,\wt\mT\in\bT_L,$ denote by $B(\mT,\wt\mT)$ the posterior ratio $\Pi(\wt\mT\C Y)/\Pi(\mT\C Y)$. If $\wt\mT\subset\mT$ and $|\mT\backslash\wt\mT|=K$, we say $\wt\mT$ is a $K$-node sub tree of $\mT$.

In what follows, we show how the general two drift conditions of \cite{zhou2021dimension} can be  applied in the context of regression trees. First, we consider the case of Bayesian CART. We will check the conditions in Theorem \ref{thm:drift2_modif}. To ensure a non-negative spectrum of the transition matrix, we consider the lazy version $P_{lazy}$, defined as $(P+I)/2$. We can account for this by scaling the terms added to 1 in Proposition \ref{prop:v2_approx} by a factor of $1/2$\footnote{When $(PV)(\mT)\leq (1-\delta) V(\mT)$ for some $\delta\in(0,1)$ and a drift function $V$, we have $(PV)(\mT)/2\leq (1/2-\delta/2) V(\mT)$. Therefore, $(P_{lazy}V)(\mT)=(PV)(\mT)/2+V(\mT)/2 \leq  (1-\delta/2) V(\mT)$. }. Therefore, the bounds in Proposition \ref{prop:v2_approx} become in a following manner. For any underfitted tree $\mT\in \bT_L$, 
\[\frac{(P_{lazy}V_1)(\mT)}{V_1(\mT)}\leq 1-\frac{A^2}{2^{L+6}(C_{f_0}+2)^2}\frac{\log^2n}{n}+\frac{\e-1}{4n^{(A^2/8\log n-1)}},\]
and for any overfitted tree $\mT\in\bT_L$ such that $\mT\neq \mT^*$, 
\begin{equation}\label{eq:m_case}
\frac{(P_{lazy}V_2)(\mT)}{V_2(\mT)}\leq 1-\frac{1}{2^{L+3}}\frac{1}{(1+n^{5/2-c})}+\frac{M}{2n^{c-3/2}}+\frac{1}{2}n^{1- (A^2\log n )/8 },
\end{equation}
where $M$ is set to 1. We assign the values $\lambda_1$ and $\lambda_2$ to correspond to $V_1$ and $V_2$, where
\begin{align}
\lambda_1&=1-\frac{A^2}{72(C_{f_0}+2)^2}\frac{\log^2n}{2^L\,n},\nonumber\\
\lambda_2&=1-\frac{1}{2^{L+7/2}}.\label{eq:lambda_2_large_n}
\end{align} 
Let $A\subset \bT_L$ be a set of overfitted trees and $\mT^*\in A$ is the true tree. Then, conditions (i) and (ii) of Theorem \ref{thm:drift2_modif} are satisfied with the above $\lambda_1$ and $\lambda_2$ for a large enough $n$ and $ c>3$. Let $K_1=2\e, K_2=\e,$ and then by Lemma \ref{lem:v_property} (i), conditions (iii) and (iv) are satisfied; The second condition of (iv) is satisfied because for any proposal $\wt\mT\in A^c$ from the state $\mT\in A$, we have $V_2(\mT)\leq V_2(\wt\mT)$ (see, Remark \ref{eq:v2design}). To check condition (v), we set $C_2=2\e$ as a universal constant in Corollary \ref{cor:drift2}. We want to see if $P_{lazy}(\mT,A^c)$ for $\mT\in A$ is smaller than $q= \min(1-\lambda_1,(1-\lambda_2)/C_2K_2)$. Note that the only transition that makes an overfitted tree to underfitted one is the PRUNE movement. Also, by \eqref{eq:trans_bound} and by the definition of $P_{lazy}$, we have $ P_{lazy}(\mT,A^c)\leq B(\mT, A^c)/2$. Therefore, by applying \eqref{eq:lem_b}, we have
\begin{align}
\frac{B(\mT, A^c)}{2}&=\sum_{\wt\mT\in A^c\cap N_p(\mT)} \frac{B(\mT, \wt\mT)}{2}\leq \sum_{\wt\mT\in A^c\cap N_p(\mT)}\frac{1}{2n^{ (A^2\log n )/8 }}\nonumber\\
&\leq \frac{1}{2n^{ (A^2\log n )/8 -1}}\leq q=\min(\frac{A^2}{72(C_{f_0}+2)^2}\frac{\log^2n}{2^L\,n},\frac{1}{\e^2 \,2^{L+9/2} }),\label{eq:q_condition}
\end{align}
for large enough $A$ and $n$. Therefore, condition (v) is satisfied. Now, by applying Corollary \ref{cor:drift2}, we have
 \begin{align*}
\tau_\epsilon&\lesssim \frac{4\log(6/\epsilon)}{\log C_2}\log( 2C_2\e)\max\left\{\frac{72(C_{f_0}+2)^2}{A^2}\frac{2^L\,n}{\log^2n}, C_2 2\e 2^{L+7/2}\right\}\\
&\lesssim  \log(6/\epsilon) \max\left( \frac{9\,(C_{f_0}+2)^2}{A^2}\frac{2^L\,n}{\log ^2n}, 2^{L+5}\right).
\end{align*}
Lastly, when it comes to the Twiggy Bayesian CART, the only difference is that we have $M=2L$ in \eqref{eq:m_case} instead of $M=1$. This change does not affect the above proof because $\lambda_2$ in \eqref{eq:lambda_2_large_n} is still valid. Therefore, we finish our proof.

\subsection{Proof of Proposition \ref{prop:v2_approx}}\label{pf:prop1}
The proof is based on the key decomposition characterized in \cite{zhou2021dimension} as (for $i=1,2$) 
\begin{align*}
\frac{(PV_i)(\mT)}{V_i(\mT)}&=1+\sum_{\wt\mT\neq\mT}R_i(\mT,\wt\mT)P(\mT,\wt\mT)\\
&=1+\sum_{\star=g,p}\sum_{\wt\mT\in \mN_\star(\mT)}R_i(\mT,\wt\mT)P(\mT,\wt\mT),
\end{align*} and on a useful bound for the transition probability (for any $\wt\mT\neq\mT\in \bT_L$)
\begin{equation}\label{eq:trans_bound}
P(\mT,\wt\mT)=\min\{S(\mT\rightarrow \wt\mT),B(\mT,\wt\mT)S(\wt\mT\rightarrow \mT)\}\leq B(\mT,\wt\mT).
\end{equation}

\medskip 

Here, we consider both the Bayesian and Twiggy CART together in one place. This is possible by observing the following commonalities. (1) The neighbor sizes for both algorithms can be bounded by $|\mN_p(\mT)|\leq 2^L\leq n/2$ and $|\mN_g(\mT)|\leq 2^L\leq n/2$. For the Bayesian CART, $|\mN_g(\mT)| = |\mT_{ext}|\leq 2^{L-1}\leq n/2$. In the case of the Twiggy CART, $|\mN_g(\mT)| = |\mT_{full,int}^L\backslash \mT_{int}|\leq 2^L\leq n/2$. (2) The internal tree size difference between the existing tree and the proposed one is $k\geq 1$ for the Twiggy CART, while the Bayesian CART is a special case with $k=1$. These commonalities allow for a unified framework to prove both algorithms.

The unimodal shape of the posterior is crucial for guaranteeing the linear mixing rate of LIT-MH. Therefore, we first characterize the posterior landscape, which implies the posterior unimodality given (Twiggy) GROW and PRUNE movements. Recall that on the event $\mA_n$ defined in \eqref{eq:event_set}, we have two prior ratios. First, similar to \eqref{eq:ratio_post_overfit2} for any overfitted trees $\mT\subset \wt\mT\in \bT_L$ such that $\mT\supseteq\mT^*$ and $|\wt\mT\backslash\mT|=K$, 
\begin{equation}\label{eq:lem_a}
\frac{\Pi(\wt\mT\C Y)}{\Pi(\mT\C Y)}\leq n^{-K(c-3/2)}.
\end{equation}
 Second, due to Assumption \ref{ass:f_classic}, for any underfitted tree $\mT\in \bT_L$, there exists a tree $\wt\mT\in \mN_g(\mT)$ containing (at least) one extra signal node, which may not be unique. Such $\wt\mT$ should have one extra node than $\mT$ for the Bayesian CART or $k\geq 1$ extra nodes for the Twiggy Bayesian CART. For any such $\wt\mT$, from \eqref{eq:posterior_ratio_overfit2} we have 
\begin{equation}\label{eq:lem_b}
\frac{\Pi(\wt\mT\C Y)}{\Pi(\mT\C Y)}\geq n^{  (A^2\log n )/8}.
\end{equation}
Now, we characterize the properties of the two drift functions.
\begin{lemma}\label{lem:v_property} Under the same assumptions of Theorem \ref{thm:coupling}, for any $\mT,\wt\mT\in \bT_L$, the following statements hold with probability at least $1-4/n-\e^{-n/8}$.
\begin{itemize}
\item[(i)] $1\leq V_1(\mT)\leq \e$ and $1\leq V_2(\mT)\leq \e$.
\item[(ii)] When $\wt\mT\supset\mT$, 
\[R_1(\mT,\wt\mT)\leq 0,~~R_1(\wt\mT,\mT)\geq 0.\]
\item[(iii)] When $\wt\mT\supset\mT$ and $|\wt\mT_{int}\backslash \mT_{int}|=k$, where $\mT\supseteq \mT^*$,
\[R_2(\mT,\wt\mT)\leq \frac{2k}{2^{L}},~~R_2(\wt\mT,\mT)\leq -\frac{k}{2^{L+1}}.\]
\end{itemize}
\end{lemma}
\begin{proof}
For part (i), we first show the upper bound of $V_1(\mT)$. We will work on $\mA_n' = \mA_n\cap \{ \varepsilon: \|\varepsilon\|_2^2\leq 2n\},$ where $\mA_n$ is defined in \eqref{eq:event_set}. Since $\|\varepsilon\|_2^2\sim \chi^2(n),$ by applying the tail bound in \cite{ghosh2021exponential} (Theorem 1), we have $\mathbb{P}( \|\varepsilon\|_2^2> 2n)\leq \e^{-n/8}.$ Therefore, $\mathbb{P}(\mA_n')\geq 1-4/n-\e^{-n/8}$. As $\bm\nu=\varepsilon$, we obtain the bound by observing that on the event $\mA_n'$ with $p:=2^L,$ the following holds.
\begin{align*}
Y'Y&=\|\bm X\b^*\|_2^2 + \|\bm \nu\|_2^2+2\bm\nu'\bm X\b^*\\
&\leq n \|\b^*\|_2^2+ 2n+2|\bm\nu'\bm X\b^*|\\
&\leq n\,p\,C_{f_0}^2 + 2n+4\|\b^*\|_2\sqrt{n^2\log p}\\
&\leq n\,p\,\left(C_{f_0}^2+ 2/p + 4C_{f_0} \sqrt{\frac{\log p}{p}} \right)\leq  n\,2^{L}\,(C_{f_0}+2)^2, 
\end{align*}
where we use the assumption that $|\beta^*_{lk}| \leq C_{f_0}$. The other upper bound in part (i) is trivial since for any tree $\mT\in \bT_L$ we have $\mT_{int}\leq 2^L$. 
For part (ii), we observe that the column space spanned by $\bm X_{\mT}$ is a subspace of the column space spanned by $\bm X_{\wt \mT}$. Therefore,
\begin{align*}
V_1(\wt\mT)/V_1(\mT)=\exp\left\{\frac{1}{2^L\,(C_{f_0}+2)^2(n+1)}\left( Y'(P_{\mT}/n-P_{\wt\mT}/n)Y\right)\right\}\leq 1.
\end{align*} For part (iii), we have $ |\wt\mT\backslash\mT^*|-|\mT\backslash\mT^*|= k\leq 2^{L}$, and $V_2(\wt\mT) /V_2(\mT) = e^{k/2^{L}}$. The result follows by using the two inequalities as in \cite{zhou2021dimension}
\begin{equation}\label{two_ineq}
\e^x\leq 1+2x,~~\e^{-x}\leq 1-\frac{x}{2},~~\forall x\in[0,1].
\end{equation}

\end{proof}

\subsubsection{Drift condition for overfitted models ($R_2$)}
\begin{lemma} \label{lem:overfit_help} Recall the definition of $w_p$ and $w_g$ in \eqref{eq:wp}. Under the same assumptions of Theorem \ref{thm:coupling}, for any overfitted tree $\mT\in \bT_L$, 
\begin{itemize}
\item[(i)] $Z_g(\mT)\leq  \frac{n^{-(c-5/2)}}{2}$.
\item[(ii)] For any subtree $\wt\mT\subset \mT$, $w_p(\wt\mT|\mT)=n^{c-3/2}$ if $\wt\mT$ contains all the signal nodes, i.e., $\wt\mT\supset\mT^*$, and otherwise, $w_p(\wt\mT|\mT)=1$.
\item[(iii)] $Z_p(\mT)\leq  |a_{\mT}|+|b_{\mT}| \, n^{c-3/2},$ where $a_\mT$ and $b_\mT$ in the decomposition $N_p(\mT) = a_\mT\cup b_\mT$ are defined as follows.
\begin{itemize}
\item[(a)] (Classical) $a_{\mT} =\mP(\mT)\cap \mT^*$ and $b_{\mT}= \mP(\mT)\backslash\mT^*$. 
\item[(b)] (Twiggy) Denote by $W(\mT)$ all the twigs existing on $\mT$ that end at a pre-terminal node (the Twiggy prune candidates). \newline $a_{\mT} =\{W\in W(\mT)| W\cap \mT^* \neq \emptyset\}$ and $b_{\mT} =\{W\in W(\mT)| W\cap \mT^* = \emptyset\}$. 
\end{itemize}
\end{itemize}
\end{lemma}

\begin{proof} 
(i) By \eqref{eq:lem_a}, \[Z_g(\mT)=\sum_{\wt\mT\in\mN_g(\mT)}\left(B(\mT,\wt\mT)\wedge  n^{ (A^2\log n )/2 }\right)\leq |\mN_g(\mT)| n^{-(c-3/2)}\leq \frac{n^{-(c-5/2)}}{2}.\] (ii) When $\wt\mT\supset\mT^*$, by applying \eqref{eq:lem_a}, we get $B(\mT,\wt\mT)\geq n^{c-3/2}$, and thus $w_p(\wt\mT|\mT)=n^{c-3/2}$ by definition in \eqref{eq:wp}. Likewise, when $\wt\mT$ loses a signal node compared with $\mT$, by \eqref{eq:lem_b}, we have $B(\mT,\wt\mT)\leq n^{- (A^2\log n )/2 }\leq 1$. Therefore, it follows from  definition \eqref{eq:wp} that $w_p(\wt\mT|\mT)=1$. (iii) (a) is apparent by definition \eqref{eq:wp} and that the prune candidates are in $\mP(\mT)$; For $\wt\mT = \mT\backslash\{(l,k)\}$, $B(\mT,\wt\mT) = 1$ if $(l,k)\in \mP(\mT)\cap \mT^*$ and $B(\mT,\wt\mT) = n^{c-3/2}$ if $(l,k)\in \mP(\mT)\backslash\mT^*$. Similarly, for (b), we apply the same reasoning to the twiggy candidate pool $W(\mT)$.
\end{proof}

\begin{lemma} \label{lem:key_R2} Under the same assumptions of Theorem \ref{thm:coupling}, for any overfitted tree $\mT\in \bT_L$ such that $\mT\neq\mT^*$,
\begin{align*}
\sum_{\wt\mT\in \mN_p(\mT)}R_2(\mT,\wt\mT)P(\mT,\wt\mT)&\leq -\frac{1}{2^{L+2}}\frac{1}{(1+n^{5/2-c})}+n^{1- (A^2\log n )/8 },\\
\sum_{\wt\mT\in \mN_g(\mT)}R_2(\mT,\wt\mT)P(\mT,\wt\mT)&\leq \frac{M}{n^{c-3/2}},
\end{align*}
where $M$ is defined as 1 for the Bayesian CART and $2L$ for the Twiggy CART.
\end{lemma}
\begin{proof}
\noindent\textbf{The PRUNE movement.} Recall the definitions of $\alpha_\mT$ and $b_\mT$ in Lemma \ref{lem:overfit_help} (iii). First, consider $\wt\mT \in b_\mT$. We know that $b_\mT$ is non-empty because $\mT\neq\mT^*$, which means there exists in $\mN_p(\mT)$ a 1-node ($k$-node for Twiggy) subtree $\wt\mT\subset\mT$ such that $\wt\mT_{int}\supseteq\mT^*_{int}$. By \eqref{eq:lem_a}, we have $B(\wt\mT,\mT)\leq n^{-(c-3/2)}\leq n^{ (A^2\log n )/8}$. Therefore, for such $\wt\mT$, $w_g(\mT|\wt\mT)=B(\wt\mT,\mT)$, and thus by applying Lemma \ref{lem:overfit_help} (i), we have
\begin{align*}
B(\mT,\wt\mT)S(\wt\mT\rightarrow \mT)&\geq B(\mT,\wt\mT)\frac{w_g(\mT|\wt\mT)}{2Z_g(\wt\mT)}= \frac{1}{2Z_g(\wt\mT)}\geq {n^{c-5/2}}\geq 1.
\end{align*}
Therefore, by \eqref{eq:trans_bound}, we have $P(\mT,\wt\mT)=S(\mT\rightarrow \wt\mT)$. Since the true signals contained in $\mT$ and $\wt\mT$ are the same, by definition \eqref{eq:K_def} and \eqref{eq:wp}, we have $S(\mT\rightarrow \wt\mT)\geq S_{PRUNE} (\mT\rightarrow \wt\mT)/2 = n^{c-3/2}/2Z_p(\mT)$. Then, applying Lemma \ref{lem:v_property} (iii), and then Lemma \ref{lem:overfit_help} (iii), we find that \[-R_2(\mT,\wt\mT)P(\mT,\wt\mT)\geq \frac{(|\mT|-|\wt\mT|)}{2^{L+1}}S(\mT\rightarrow \wt\mT)\geq \frac{n^{c-3/2}}{2^{L+2}(|a_{\mT}|+|b_{\mT}| \, n^{c-3/2})}.\] 
Since $|b_{\mT}|\geq 1$, we have for $c>5/2$,
\begin{equation}\label{eq:contrast1}
    -\sum_{\wt\mT\in b_{\mT}}R_2(\mT,\wt\mT)P(\mT,\wt\mT)\geq  \frac{|b_\mT| n^{c-3/2}}{2^{L+2}(n+|b_{\mT}| \, n^{c-3/2})} \geq\frac{1}{2^{L+2}}\frac{1}{1+n^{5/2-c}}.
\end{equation}
Note that from \eqref{eq:trans_bound} and \eqref{eq:lem_b}, we have 
\[\sum_{\wt\mT\in a_{\mT}}R_2(\mT,\wt\mT)P(\mT,\wt\mT)\leq |a_\mT|\,(\e-1) \, n^{- (A^2\log n )/8 }\leq n^{1- (A^2\log n )/8 },\]where we used $|a_\mT|\leq 2^L\leq n/2$. Since $\mN_p(\mT) = a_\mT\cup b_\mT$, we have the result of the lemma.

\medskip
\medskip

\noindent\textbf{The GROW movement.} There is no additional signal node that can be added by GROW when the current state $\mT\in\bT_L$ is overfitted. Therefore, for any $\wt\mT\supset\mT$, from \eqref{eq:trans_bound} and \eqref{eq:lem_a}, 
\begin{equation}\label{eq:application_of_trans_bound}
P(\mT,\wt\mT)\leq B(\mT,\wt\mT)\leq \frac{1}{n^{c-3/2}}.
\end{equation}
With Lemma \ref{lem:v_property} (iii) with $k=|\wt\mT_{int}|-|\mT_{int}|$, we obtain that
\[\sum_{\wt\mT\in \mN_g(\mT)}R_2(\mT,\wt\mT)P(\mT,\wt\mT)\leq\sum_{\wt\mT\in \mN_g(\mT)}\frac{2k}{2^{L}}\frac{1}{n^{c-3/2}}.\]
Since $|\mN_g(\mT)|\leq 2^{L-1}$, and $k=1$ in the Bayesian CART ($M=1$), and $|\mN_g(\mT)|\leq 2^L$ and $k\leq L$ in the Twiggy CART ($M=2L$), we get the results.
\end{proof}

\subsubsection{Drift condition for underfitted models ($R_1$)}\label{eq:sec_R1}
This section shares the same proof process for both the Twiggy CART and Bayesian CART algorithms, based on the observation at the beginning of Section \ref{pf:prop1}.

\begin{lemma} \label{lem:overfit_help2} 
Under the same assumptions of Theorem \ref{thm:coupling}, for any underfitted tree $\mT\in \bT_L$ i.e., for any $\mT\not\supset \mT^*$, 
\begin{itemize}
\item[(i)] $Z_g(\mT)\geq  n^{ (A^2\log n )/8 }$.
\item[(ii)] $Z_p(\mT)\leq n^{c-1/2}$. 
\end{itemize}
\end{lemma}

\begin{proof} 
(i) By \eqref{eq:lem_b}, we can always find a proposal $\wt\mT\in\mN_g(\mT)$ such that $B(\mT,\wt\mT)\geq n^{ (A^2\log n )/8 }$. (ii) is apparent by definition \eqref{eq:wp} and that $|\mN_p(\mT)|\leq n$ for both Bayesian and Twiggy CART.

\end{proof}

\begin{lemma}\label{lem:b_connect} Suppose
$B(\mT,\wt\mT)\geq n^{a}$ for some $a\in \R$ and define 
\begin{equation}\label{eq:b}
    b=\frac{1}{2^{L-1}(C_{f_0}+2)^2n}\left(a\log n -\log \left( \frac{\Pi(\wt\mT)}{\Pi(\mT)}\right) - \frac{|\mT_{ext}|-|\wt\mT_{ext}|}{2}\log (1+n) \right).
\end{equation} If $b\in[0,1],$ then $-R_1(\mT,\wt\mT)\geq b/2$.
\end{lemma}

\begin{proof} 
From the posterior in \eqref{eq:compare_post}, we relate $B_1$ to $R_1$ by 
\[B(\mT,\wt\mT)=\frac{\Pi(\wt\mT)}{\Pi(\mT)}(1+n)^{\frac{|\mT_{ext}|-|\wt\mT_{ext}|}{2}}\left(\frac{V_1(\wt\mT)}{V_1(\mT)}\right)^{-n\,2^{L-1}(C_{f_0}+2)^2}.\] Therefore, it follows by the assumption $\log B(\mT,\wt\mT)\geq a\log n$ that 
\[a\log n\leq \log \left( \frac{\Pi(\wt\mT)}{\Pi(\mT)}\right) + \frac{|\mT_{ext}|-|\wt\mT_{ext}|}{2}\log (1+n) - n\,2^{L-1}(C_{f_0}+2)^2\log (1+R_1(\mT,\wt\mT)).\]
Therefore, $\log (1+R_1(\mT,\wt\mT))\leq -b$, which means $-R_1(\mT,\wt\mT)\geq 1-e^{-b}$. If $b\in[0,1]$, we apply the second inequality in \eqref{two_ineq} to get $-R_1(\mT,\wt\mT)\geq b/2$.
\end{proof}

\begin{lemma} \label{lem:underfit_coupling}
Under the same assumptions of Theorem \ref{thm:coupling}, for any underfitted tree $\mT\in \bT_L$ i.e., for any $\mT\not\supset \mT^*$, 
\begin{align*}
\sum_{\wt\mT\in \mN_g(\mT)}R_1(\mT,\wt\mT)P(\mT,\wt\mT)&\leq -\frac{(A^2/8)\log^2 n}{2^{L+2}\, n(C_{f_0}+2)^2},\\
\sum_{\wt\mT\in \mN_p(\mT)}R_1(\mT,\wt\mT)P(\mT,\wt\mT)&\leq \frac{\e-1}{2n^{(A^2/8\log n-1)}}.
\end{align*}
The bounds are for both the Bayesian and Twiggy CART.
\end{lemma}
\begin{proof}
\textbf{The GROW movement.} 
By \eqref{eq:lem_b}, there exists some tree $\mG(\mT)=\wt\mT\in \mN_g(\mT)$ containing at least one extra signal node, such that $B(\mT,\mG(\mT))\geq n^{ (A^2\log n )/8 }$. By Lemma \ref{lem:b_connect} with $a= (A^2\log n )/8 $ and with large enough $n$ so that $b$ in \eqref{eq:b} is less than $1$\footnote{When $\wt\mT\supset \mT$, it is apparent $b\geq 0$ because $k=|\wt\mT_{ext}|-|\mT_{ext}|\geq 0 $ and $\log \left( \frac{\Pi(\wt\mT)}{\Pi(\mT)}\right) = k\log n^{-c} (1-n^{-c}) \leq 0 .$}, we find that 
\begin{align}
-R_1(\mT,\mG(\mT))&\geq \frac{1}{2n\,2^{L-1}(C_{f_0}+2)^2}\left(a\log n -\log \left( \frac{\Pi(\mG(\mT))}{\Pi(\mT)}\right) - \frac{|\mT_{ext}|-|\mG(\mT)_{ext}|}{2}\log (1+n) \right)\nonumber\\
&\geq \frac{1}{2^{L}\,n(C_{f_0}+2)^2}\left(a\log n -k\log \left( n^{-c}(1-n^{-c})\right) + \frac{k}{2}\log (1+n) \right)\nonumber\\
&\geq \frac{1}{2^{L}\,n(C_{f_0}+2)^2}\,a\log n ,\label{eq:best_r1}
\end{align}
where $k=|\mG(\mT)_{ext}|-|\mT_{ext}|\geq 1$. Now, for some $V\geq 1$, consider a set of good GROW moves as 
\[\mathcal{D}=\mathcal{D}(\mT)=\{\wt\mT\supset \mT: B(\mT,\wt\mT)\geq n^{ (A^2\log n )/2 -V}\}.\]Again using Lemma \ref{lem:b_connect}, we have for all $\wt\mT\in \mathcal{D}(\mT)$, 
\begin{align}
-R_1(\mT,\wt\mT)&\geq \frac{1}{2n\,2^{L-1}(C_{f_0}+2)^2}\left((a-V)\log n -\log \left( \frac{\Pi(\wt\mT)}{\Pi(\mT)}\right) - \frac{k}{2}\log (1+n) \right)\nonumber\\
&=\frac{1}{2^{L}\,n(C_{f_0}+2)^2}\left((a-V)\log n -k\log \left( n^{-c}(1-n^{-c})\right) + \frac{k}{2}\log (1+n) \right)\nonumber\\
&\geq \frac{1}{2^{L}\,n(C_{f_0}+2)^2}\,(a-V)\log n,\label{eq:good_r1}
\end{align}
where $k=|\wt\mT_{ext}|-|\mT_{ext}|\geq 1$.
Now we bound $P(\mT,\wt\mT)$ for $\wt\mT\in \mathcal{D}$. By the definition of $w_p(\mT|\wt\mT)$, for any $\wt\mT\in \mN_g(\mT),$
\begin{equation}\label{eq:vital_bound}
S(\wt\mT\rightarrow \mT)\geq S_{PRUNE} (\wt\mT\rightarrow \mT)/2 =\frac{w_p(\mT|\wt\mT)}{2Z_p(\wt\mT)}\geq \frac{1}{2Z_p(\wt\mT)}.
\end{equation}
{This lower bound of $S(\wt\mT\rightarrow \mT)$ in \eqref{eq:vital_bound} is why the two sided threshold in \eqref{eq:wp} is crucial in showing the mixing rate. Due to the two sided threshold, $S(\wt\mT\rightarrow \mT)$ is not too small so that the transition kernel $P(\mT,\wt\mT)$ is also not too small as the following:} For any $\wt\mT\in \mN_g(\mT),$
\begin{align}
P(\mT,\wt\mT)&=\min\{S(\mT\rightarrow \wt\mT),B(\mT,\wt\mT)S(\wt\mT\rightarrow \mT)\}\nonumber\\
&\geq\min\{\frac{w_g(\wt\mT|\mT)}{2Z_g(\mT)},\frac{B(\mT,\wt\mT)}{2Z_p(\wt\mT)} \}\nonumber\\
&\geq w_g(\wt\mT|\mT) \min\{\frac{1}{2Z_g(\mT)},\frac{1}{2Z_p(\wt\mT)}\}\nonumber\\
&\geq \frac{ w_g(\wt\mT|\mT) }{2Z_g(\mT)}\label{eq:p_lower}.
\end{align}
In the last inequality, we used Lemma \ref{lem:overfit_help2} (i) and (ii), for a large enough $A$, 
\[Z_g(\mT)\geq  n^{ (A^2\log n )/8}\geq n^{c-1/2}\geq Z_p(\wt\mT).\]
Define $\mathcal{D}'=\mathcal{D}\backslash\{\mG(\mT)\},$ which may be empty. Let $W=\sum_{\wt\mT\in\mathcal{D}'}w_g(\wt\mT|\mT)$. Then, since $|\mN_g(\mT) \backslash \mathcal{D}|\leq n$, and  for $\wt\mT\not\in \mathcal{D}$, $B(\mT,\wt\mT)\geq n^{ (A^2\log n )/8 -V}$, we have
\begin{align}
Z_g(\mT)&=\sum_{\wt\mT\in \mN_g(\mT)}w_g(\wt\mT|\mT)\nonumber\\
&= w_g(\mG(\mT)|\mT)+\sum_{\wt\mT\in\mathcal{D}'}w_g(\wt\mT|\mT)+ \sum_{\wt\mT\in \mN_g(\mT)\backslash\mathcal{D} }w_g(\wt\mT|\mT)\nonumber\\
&= n^{ (A^2\log n )/8 }+W+n^{ (A^2\log n )/8 -V+1}\nonumber\\
&\leq W+2n^{ (A^2\log n )/8 }.\label{eq:Z_bound1}
\end{align}
Now, putting all things together using Lemma \ref{lem:v_property} (ii), \eqref{eq:best_r1}, \eqref{eq:good_r1}, \eqref{eq:p_lower}, and \eqref{eq:Z_bound1}, and recalling $a= (A^2\log n )/8 $, we get
\begin{align}
-&\sum_{\wt\mT\in \mN_g(\mT)}R_1(\mT,\wt\mT)P(\mT,\wt\mT)\geq -\sum_{\wt\mT\in \mathcal{D}(\mT)}R_1(\mT,\wt\mT)P(\mT,\wt\mT)\nonumber\\
&\geq  \frac{\log n}{2^{L}\,n(C_{f_0}+2)^2}\left(\,a \frac{ n^{ (A^2\log n )/8 } }{2Z_g(\mT)}+(a-V)\sum_{\wt\mT\in \mathcal{D}'(\mT)}\frac{ w_g(\wt\mT|\mT) }{2Z_g(\mT)} \right)\nonumber\\
&\geq  \frac{a\log n}{2^{L+2}\,n(C_{f_0}+2)^2}\frac{n^{ (A^2\log n )/8 } +(1-V/a)W}{n^{ (A^2\log n )/8 }+W/2}.\label{eq:collective_p}
\end{align}
Therefore, {as long as $a\geq 2V$}, we have 
\[\sum_{\wt\mT\in \mN_g(\mT)}R_1(\mT,\wt\mT)P(\mT,\wt\mT)\leq -\frac{a\log n}{2^{L+2}\,n(C_{f_0}+2)^2}= -\frac{(A^2/8)\log^2 n}{2^{L+2}\,n(C_{f_0}+2)^2}.\]
\medskip

\noindent\textbf{The PRUNE movement.} 
By applying Lemma \ref{lem:overfit_help2} (i), we have  for any 1-node ($k$-node for Twiggy) subtree $\wt\mT\subset\mT$
\begin{align*}
B(\mT,\wt\mT)S(\wt\mT\rightarrow \mT)&\leq B(\mT,\wt\mT)\frac{w_g(\mT| \wt\mT)}{Z_g(\wt\mT)}\\&\leq \frac{1}{Z_g(\wt\mT)}\leq \frac{1}{n^{ (A^2\log n )/8}}.
\end{align*}
Therefore, by \eqref{eq:trans_bound}, we have 
\[R_1(\mT,\wt\mT)P(\mT,\wt\mT)\leq  R_1(\mT,\wt\mT)B(\mT,\wt\mT)S(\wt\mT\rightarrow \mT)\leq \frac{R_1(\mT,\wt\mT)}{n^{ (A^2\log n )/8}}.\]
By Lemma \ref{lem:v_property} (i), $R_1(\mT,\wt\mT)\leq \e-1$, and the pool size is $|\mN_p(\mT)|\leq n/2.$
Therefore, 
\[\sum_{\wt\mT\in \mN_p(\mT)} R_1(\mT,\wt\mT)P(\mT,\wt\mT)\leq \frac{\e-1}{2n^{(A^2/8\log n-1)}}.\]
\end{proof}

\subsection{Proof of Remark \ref{remark:coupling_app}}\label{pf:coupling_app}
Here, we present the non-informed counterpart of Proposition \ref{prop:v2_approx}. To achieve this, we modify $V_1$ as 
\begin{align}
V_1(\mT)&=\exp\left\{\frac{1}{2^{L}\,C_{f_0}^2n}\left( (\bm X\b ^*)'(I-P_{\mT}/n)\bm X\b ^*\right)\right\},\label{eq:rescuer2}
\end{align}
which is designed to ignore the error terms. This is to guarantee $R_1(\mT,\wt\mT)=0$ for $\wt\mT$ obtained by pruning non-signals from $\mT\in\bT_L$. All the properties of Lemma \ref{lem:v_property} can be shown to apply to the new $V_1$ on the event $\mA_n$ (with probability at least $1-4/n$). For example, for Lemma \ref{lem:v_property} (i), we obtain the bound by 
\[\|\bm X\b^*\|_2^2= n \|\b^*\|_2^2\leq n\, 2^{L}\,C_{f_0}^2.\]
\begin{proposition}\label{prop:v2_approx2} Under the same assumptions of Theorem \ref{thm:coupling}, for the Bayesian CART and Twiggy Bayesian CART algorithms described in Section \ref{section:sampler} and Section \ref{sec:twiggy}, with probability at least $1-4/n$ we have the following properties of the drift functions. 
\begin{itemize}
\item[(i)] For any underfitted tree $\mT\in\bT_L$,
\[\frac{(PV_1)(\mT)}{V_1(\mT)}\leq 1-\frac{\delta_1 A^2\log^2 n}{2^{2L+2}C_{f_0}^2 \,n}+\frac{\e-1}{2n^{(A^2/8\log n-1)}}.\]
\item[(ii)] For any overfitted tree $\mT\in\bT_L$ such that $\mT\neq \mT^*$, for $c>3/2$,
\begin{equation}\label{eq:v2_bcart}
\frac{(PV_2)(\mT)}{V_2(\mT)}\leq 1-\frac{1}{2^{2L+2}}+\frac{M}{n^{c-3/2}}+n^{1- (A^2\log n )/8 },
\end{equation}
where $M=\delta_1=1$ for the Bayesian CART and $M=2L,~\delta_1= \frac{2(D-1)}{D^{L}-1}$ for the Twiggy Bayesian CART. 
\end{itemize}
\end{proposition}
{ To ensure that the upper bound in \eqref{eq:v2_bcart} is less than 1, we impose a stronger condition on $c$, requiring $c\geq 4$. This is because if $c=7/2$, we may have $1/2^{2L+1}  \asymp \frac{n^{c-3/2}}{2},$ for example when $L=\Lmax.$ Now, with $\lambda_1 =1-\frac{\delta_1 A^2\log^2 n}{2^{2L+4}C_{f_0}^2 \,n} $ and $\lambda_2 = 1-\frac{1}{2^{2L+4}}$, }it is straightforward to extend Section \ref{sec:application_two_drift} (the application of the two-drift condition) to this case, obtaining the bound in Remark \ref{remark:coupling_app}. Proposition \ref{prop:v2_approx2} is derived from the non-informed counterpart of Lemma \ref{lem:key_R2} and Lemma \ref{lem:underfit_coupling} presented below. 
\begin{lemma} Under the same assumptions of Theorem \ref{thm:coupling}, for the Bayesian CART and Twiggy Bayesian CART algorithms described in Section \ref{section:sampler} and Section \ref{sec:twiggy}, for any overfitted tree $\mT\in \bT_L$, such that $\mT\neq\mT^*$,
\begin{align}
\sum_{\wt\mT\in \mN_p(\mT)}R_2(\mT,\wt\mT)P(\mT,\wt\mT)&\leq -\frac{1}{2^{2L+2}}+n^{1- (A^2\log n )/8 },\label{eq:prune_4}\\
\sum_{\wt\mT\in \mN_g(\mT)}R_2(\mT,\wt\mT)P(\mT,\wt\mT)&\leq \frac{M}{n^{c-3/2}},\label{eq:grow_4}
\end{align}where $M=1$ for the Bayesian CART, and $M=2L$ for the Twiggy CART. 
\end{lemma}
\proof 
\textbf{The PRUNE movement.} First, we consider the case of the Bayesian CART. The proof is the same as in Lemma \ref{lem:key_R2}, except for the bound on $-\sum_{\wt\mT\in b_{\mT}}R_2(\mT,\wt\mT)P(\mT,\wt\mT)$. Since $\mT$ is an overfitted tree and $\mT\neq\mT^*$, we have $|b_\mT|\geq 1$. We take any $\wt\mT\in b_\mT$. By \eqref{eq:lem_a}, we have $B(\mT,\wt\mT)\geq n^{(c-3/2)}$ and  $n^{(c-3/2)}S(\wt\mT\rightarrow \mT)/S(\mT\rightarrow \wt\mT)\geq 1$. This results in the acceptance rate of $1$, implying that for such $\wt\mT$, $P(\mT,\wt\mT)=S(\mT\rightarrow\wt\mT)$, and thus by applying Lemma \ref{lem:v_property} (iii), and then Lemma \ref{lem:overfit_help} (iii), we find that 
\begin{equation}\label{eq:contrast2}
    -R_2(\mT,\wt\mT)P(\mT,\wt\mT)\geq \frac{(|\mT|-|\wt\mT|)}{2^{L+1}}S(\mT\rightarrow \wt\mT)\geq  \frac{(|\mT|-|\wt\mT|)}{2^{L+1}\times 2^{L+1}}\geq \frac{1}{2^{2L+2}}.
\end{equation} The other parts of the proof in Lemma \ref{lem:key_R2} do not depend on the choice of the proposal probability $S(\cdot\rightarrow\cdot)$. Therefore, we have the result.

Now, for the Twiggy Bayesian CART, the only difference is in the lower bound of $P(\mT,\wt\mT)$. By \eqref{eq:reweight_tw}, \eqref{eq:tw_prune}, \eqref{eq:trans_bound}, and $D\leq \e,$
\begin{align*}
P(\mT,\wt\mT)&=\min\{S(\mT\rightarrow \wt\mT),B(\mT,\wt\mT)S( \wt\mT\rightarrow \mT)\}\\
&\geq \min\left\{\frac{1}{2^{L+1}} ,\frac{D-1}{2^L(D^L-1)}n^{c-3/2}\right\} =\frac{1}{2^{L+1}}.
\end{align*} 
Therefore, by proceeding as above, we obtain the bound.

\medskip

\textbf{The GROW movement.} The bound in \eqref{eq:grow_4} is the same as the informed case in Lemma \ref{lem:key_R2}. The proof of Lemma \ref{lem:key_R2} does not depend on a specific choice of $S(\cdot\rightarrow \cdot)$ but uses only \eqref{eq:application_of_trans_bound}, which is obtained by \eqref{eq:trans_bound}. Since the non-informed version shares all the same movement neighbor and posterior ratios, the proof also results in \eqref{eq:grow_4} in the current lemma.

\begin{lemma} Under the same assumptions of Theorem \ref{thm:coupling}, for the Bayesian CART and Twiggy Bayesian CART algorithms described in Section \ref{section:sampler} and Section \ref{sec:twiggy}, for any underfitted tree $\mT\in \bT_L$ i.e., for any $\mT\not\supset \mT^*$, 
\begin{align*}
\sum_{\wt\mT\in \mN_g(\mT)}R_1(\mT,\wt\mT)P(\mT,\wt\mT)&\leq -\frac{\delta_1 A^2\log^2 n}{2^{2L+2}C_{f_0}^2 \,n},\\
\sum_{\wt\mT\in \mN_p(\mT)}R_1(\mT,\wt\mT)P(\mT,\wt\mT)&\leq \frac{\e-1}{2n^{(A^2/8\log n-1)}},
\end{align*} 
where $\delta_1=1$ for Bayesian CART and $\delta_1= \frac{2(D-1)}{D^{L}-1}$.
\end{lemma}
\proof 
\textbf{The GROW movement.} Consider first the case of the Bayesian CART. As in Lemma \ref{lem:underfit_coupling}, by \eqref{eq:lem_b} there exists a tree $\mG(\mT)\supset\mT$ containing at least one extra signal node, such that $B(\mT,\mG(\mT))\geq n^{ (A^2\log n )/2 }$. This large posterior rate implies $P(\mT,\mG(\mT)) = S_{GROW}(\mT\rightarrow\mG(\mT))/2\geq 1/2^{L+1}$. By inequality \eqref{two_ineq}, and with a decomposition $P_{\mG(\mT)}=P_{\mT}+P_{\mG(\mT)\backslash \mT}$,
\begin{align*}
    -R_1(\mT,\mG(\mT))&\geq \frac{1}{2}\frac{1}{C_{f_0}^2n\,2^{L}}\left((\bm X\b ^*)'(P_{\mG(\mT)}/n-P_{\mT}/n)\bm X\b ^*\right)\\
    &\geq \frac{1}{2^{L+1}}\frac{1}{C_{f_0}^2n}\left((\bm X\b ^*)'(P_{\mG(\mT)\backslash \mT}/n)\bm X\b ^*\right)\\
    &\geq \frac{1}{2^{L+1}\,n\,C_{f_0}^2}n\frac{A^2\log^2n}{n}= \frac{A^2\log^2n}{2^{L+1}\,n\,C_{f_0}^2}.
\end{align*}
Besides $R_1(\mT,\wt\mT)\leq 0$ for all $\wt\mT\in N_G(\mT)$ by Lemma \ref{lem:v_property} (ii). Therefore, by considering this movement to $\mG(\mT)$, we obtain
\begin{align}
-&\sum_{\wt\mT\in \mN_g(\mT)}R_1(\mT,\wt\mT)P(\mT,\wt\mT)\geq -R_1(\mT,\mG(\mT))P(\mT,\mG(\mT))\label{eq:contrast3}\\
&\geq  \frac{A^2\log^2n}{2^{L+1}\,n\,C_{f_0}^2}P(\mT,\mG(\mT))\geq  \frac{A^2\log^2n}{2^{2L+2}\,n\,C_{f_0}^2}.\nonumber
\end{align}
Now, we consider the case of the Twiggy Bayesian CART. The only change in the above calculation is the lower bound for $P(\mT,\mG(\mT))$. By \eqref{eq:reweight_tw}, \eqref{eq:tw_prune}, and \eqref{eq:trans_bound}, 
\begin{align*}
P(\mT,\mG(\mT))&=\min\{S(\mT\rightarrow \mG(\mT)),B(\mT,\mG(\mT))S( \mG(\mT)\rightarrow \mT)\}\\
&\geq \min\left\{\frac{D-1}{2^L(D^L-1)},\frac{n^{(A^2\log n)/2)}}{2^{L+1}}\right\} = \frac{D-1}{2^L(D^L-1)}=\frac{\delta_1}{2^{L+1}}.
\end{align*}
Therefore, by proceeding as above, we obtain the bound.

\medskip

\textbf{The PRUNE movement.} 
Consider first the case of the Bayesian CART. There are two cases of a 1-node subtree $\wt\mT\subset \mT$. First, when $\wt\mT$ is made by pruning a \emph{non-signal} from $\mT$: Due to the modification of the new $V_1$ in \eqref{eq:rescuer2}, we have $R_1(\mT,\wt\mT)=0$.  Second, when $\wt\mT$ is made by pruning a \emph{signal} from $\mT$: We have from \eqref{eq:lem_b}, $B(\mT,\wt\mT)\leq n^{- (A^2\log n )/8}$, and by Lemma \ref{lem:v_property} (i), $R_1(\mT,\wt\mT)\leq \e-1$. From \eqref{eq:trans_bound}, we have $P(\mT,\wt\mT)\leq B(\mT,\wt\mT)$. Therefore, considering the maximum possible value of $R_1(\mT,\wt\mT)P(\mT,\wt\mT)$ and since the pool size is $|\mN_p(\mT)|\leq 2^L\leq n/2$, we have
\[\sum_{\wt\mT\in \mN_p(\mT)} R_1(\mT,\wt\mT)P(\mT,\wt\mT)\leq \frac{\e-1}{2n^{(A^2/8\log n-1)}}.\]
Now, when it comes to the case of the Twiggy Bayesian CART, there are two cases of a $k$-node subtree $\wt\mT\subset \mT$. First, when all nodes of $\mT\backslash\wt\mT$ are non-signals, and second when $\mT\backslash\wt\mT$ contains at least one signal. In these cases, the above reasoning applies in the same way.
\begin{remark}\label{rmk:why_save}{ The only algorithmic difference between the informed versions and their non-informed counterparts is the proposal distribution $S(\cdot\rightarrow \cdot)$, whether proposing uniformly or informatively. This difference brings two major benefits compared to the original non-informed algorithms. First, the proposal probability of $\mG(\mT)$ from $\mT$ in the canonical path, or namely, the best movement, is significantly improved. Note that for any MH-algorithm, $P(\mT,\wt\mT)\leq S(\mT\rightarrow\wt\mT)$. Therefore, no matter how much posterior increase can be brought by the best movement, its transition probability is still upper bounded by $S(\mT\rightarrow\wt\mT)\leq 1/2^L$ in the non-informed algorithms. This contrast is highlighted by comparing the large proposal probability bound in \eqref{eq:contrast1} with the small lower bound (of a uniform proposal) in \eqref{eq:contrast2}. Second, the change to the informed proposal increases the transition probability of a set of movements that reduce the drift function values, or namely, good movements. This plays an important role especially when handling underfitted tree cases (GROW) as in \eqref{eq:collective_p} and the following display, which exploit that the transition probability of good movements is more than $1/4$. Although there is no guarantee that there will be multiple good movements other than the best movement, even when there is only a single best movement, \eqref{eq:collective_p} implies then its transition probability is greater than $1/4$. In the proving technique of two-drift conditions, movements that have a small drift ratio ($R_1$ and $R_2$) are good movements. Here, such many good movements collectively reduce the expectation of the ratio in the next MCMC step. On the contrary, in the above proof of Remark \ref{remark:coupling_app}, we considered only a single best movement when handling underfitted tree cases (GROW) as in \eqref{eq:contrast3}. Note that this consideration was unavoidable. In the non-informed setting, like in the informed setting, guaranteeing multiple good movements (here, in the sense of posterior increase) is difficult other than a single best movement. However, unlike the case of the informed setting, the uniform proposal only guarantees that the transition probability of the best movement (signal obtaining) is $\leq 1/2^{L+1}$. Therefore, the upper bound in Remark \ref{remark:coupling_app} slower than that of the informed algorithms is not because only a single movement was considered in \eqref{eq:contrast3}. Rather, this is due to the difference in the proposal distributions.
}\end{remark}
\section{Comparison to \citep{yang2016computational}}\label{sec:comp_yang}
 
\citep{yang2016computational} showed rapid mixing of MH algorithm of a Bayesian variable selection problem for a standard linear model 
$$ 
Y=X\b^*+\bm\omega, 
$$ where $X\in\R^{n\times p}$ is the design matrix, $\b^*\in\R^p$ is the unknown regression vector and $\bm\omega\sim \mN (0,\sigma_0I_n)$. Here $p$ is the number of covariates and $n$ is the sample size. Denote by $\bm\gamma\in \{0,1\}^p$ the vector of indicators for influential regression weights in $\bm\b^*$. A coefficient $\beta_j^*\in \bm\b^*$ is considered influential if $|\beta_j^*|\geq C_\beta$ for a constant $C_\beta>0$ that depends on $(\sigma_0, n, p)$. The MH algorithm for Bayesian variable selection generates its proposal by randomly swapping two indicators or adding/removing one indicator in $\bm\gamma$. The Bayesian variable selection problem is highly connected to our tree sampling. However, there is no requirement on the selected variables in $\bm\gamma$ to maintain a systematic structure. This is an important contrast from our setting, where the selected nodes should compose a valid tree shape. Therefore, it is an interesting question whether this imposed tree structure would encourage even more rapid mixing in comparison with the standard Bayesian variable selection problem. 

In answering this question, we introduce some notations in \citep{yang2016computational} and compare them with our settings. Denote the design matrix of the selected columns by $X_{\bm\gamma}\in\R^{n\times |\bm\gamma|}$. The Bayesian hierarchical model considered in \citep{yang2016computational} is 
\begin{align}
\bm\omega&\sim \mN (0,\phi^{-1}I_n)\nonumber\\
\pi (\phi)&\propto \frac{1}{\phi},\nonumber\\
\bm\beta|\bm\gamma&\sim \mN(0,g\phi^{-1}(X'_{\bm\gamma}X_{\bm\gamma})^{-1}),\nonumber\\
\Pi(\bm\gamma)&\propto \Big(\frac{1}{p}\Big)^{\kappa|\bm\gamma|}\mathbb{I}[|\bm\gamma|\leq s_0],\label{eq:yang_prior}
\end{align}
where $s_0$ is the upper bound on the maximum number of important covariates, $g>0$ is the degree of dispersion in the regression prior, and $\kappa$ is the model size penalty. A hyperparameter $\alpha\geq 1/2$ is used to constraint the relationship between $g$ and $p$ by $g\asymp  p^{2\alpha}$. Our setting corresponds to when $p= n/2$, $C_\beta = A\log n/\sqrt{n}$, $g=n$, $s_0 = 2^L$, $\sigma_0^2=1$, $\kappa=c$, and $\alpha=1/2$. The consistency condition in \cite{yang2016computational}  ((9a), High SNR condition) is satisfied if $A^2\geq 30(4.5+\kappa)$ given $\Lmax\geq 5$. Due to the orthogonality of our design matrix $X'X=nI_p$ and Assumption \ref{ass:f_classic}, these hyperparameter settings meet their regularity conditions (Assumption A to D in \cite{yang2016computational}) by additionally assuming $C_{f_0}^22^L\leq \log (n/2)$, $c \geq 17+1/2$ and $L\leq \Lmax-\log_2\Lmax-4$ as follows. 

\noindent\textbf{Assumption A)} The condition (7a) is written as $C_{f_0}^22^L\leq \log (n/2)$, which leads to satisfy $\|\frac{1}{\sqrt n}X\b^*\|_2^2\leq \log n$. The other condition (7b) is trivial since non-influential nodes are regarded to have zero coefficients, with $\tilde{L}=0$.

\noindent\textbf{Assumption B)} The lower restricted eigenvalue condition is met for any $\nu\in(0,1]$ since $\frac{1}{n}X'X=I_p$. Due to the orthogonality of $X$ as discussed in \cite{yang2016computational}, the sparse projection condition is always satisfied when $L=4\nu^{-1}$. We set $\nu=1$ and $L=4$ since smaller $L$ is a less restrictive condition.

\noindent\textbf{Assumption C)} It is trivial with the hyperparameter settings of $g=n,p=n/2, \alpha=1/2$ and $c = \kappa\geq 17+1/2$.

\noindent\textbf{Assumption D)} Version $D(s_0)$ should be met for the Theorem 2 in \citep{yang2016computational}, which is 
\[\max\{1,(2\nu^{-2}\omega(X)+1)s^*\}\leq s_0\leq \frac{1}{32}\left\{\frac{n}{\log p} - 8\tilde{L}\right\},\] where $\omega:=\max_{\mT} \|(X_{\mT}'X_{\mT})^{-1} X_{\mT}'X_{\mT^*\backslash\mT}\|_{\rm op}^2$. In our translation, $s_0=2^L$ and $s^* = |\mT^*_{int}|$. The lower bound condition on $s_0$ is trivial because $X$ is orthogonal, i.e., $w(X)=0$. Since $\tilde{L}=0$, the right upper bound is satisfied when $L\leq \Lmax-\log_2\Lmax-4$.

Therefore, the rapid mixing guarantee (Theorem 2 in \citep{yang2016computational}) is translated as follows. 

\begin{theorem}[\citep{yang2016computational} Theorem 2] \label{theo:yang} Assume the model \eqref{eq:main_model} with Assumption \ref{ass:f_classic} and the spike-and-slab prior in \eqref{eq:yang_prior} with $\kappa=c \geq 17+1/2$. Consider the Spike-and-Slab MH algorithm in \cite{yang2016computational} without a tree structure restriction ($\bm\gamma$ is the vectorized $\mT\in \bT_L$). Assume $C_{f_0}^22^L\leq \log (n/2)$ and $1\leq L\leq \Lmax-\log_2\Lmax-4$. With a large enough constant $A>0$, with probability at least $1-c_3p^{-c_4}$, \begin{equation}\label{eq:yang}
    \tau_{\epsilon}\leq 3 \times 2^{2L}n \left[n\log (n/2)+(1+ 4c) 2^L\log (n/2))+\log (2/\epsilon)\right], 
\end{equation}
for some $c_3,$ and $c_4$. 
\end{theorem}

Now, for the comparison purpose, we match the settings by applying the sparsity prior in \eqref{eq:yang_prior} to our result instead of the classical Bayesian CART prior in \ref{section:bart_prior_only}. That is, for the comparison, we use the prior $\Pi(\mT)\propto \wt p_{lk}^{(-|\mT_{int}|)}\,\mathbb{I}[|\mT_{int}|\leq 2^L],$ where $\wt p_{lk} = (n/2)^{-c}$. Because $\wt p_{lk}=2^cp_{lk},$ it is easy to verify the consistency and the mixing rate results in Theorem \ref{lemma:consist} and Theorem \ref{eq:theo1} for $c>7/2$ as long as $\wt p_{lk} <1/2.$ Therefore, we can compare our upper bound in \eqref{eq:tau_bound1} against the bound in \eqref{eq:yang}.

\section{Additional Visualizations}\label{sec:add_vis}
\clearpage
\begin{figure}
    \centering
    \includegraphics[width=\linewidth]{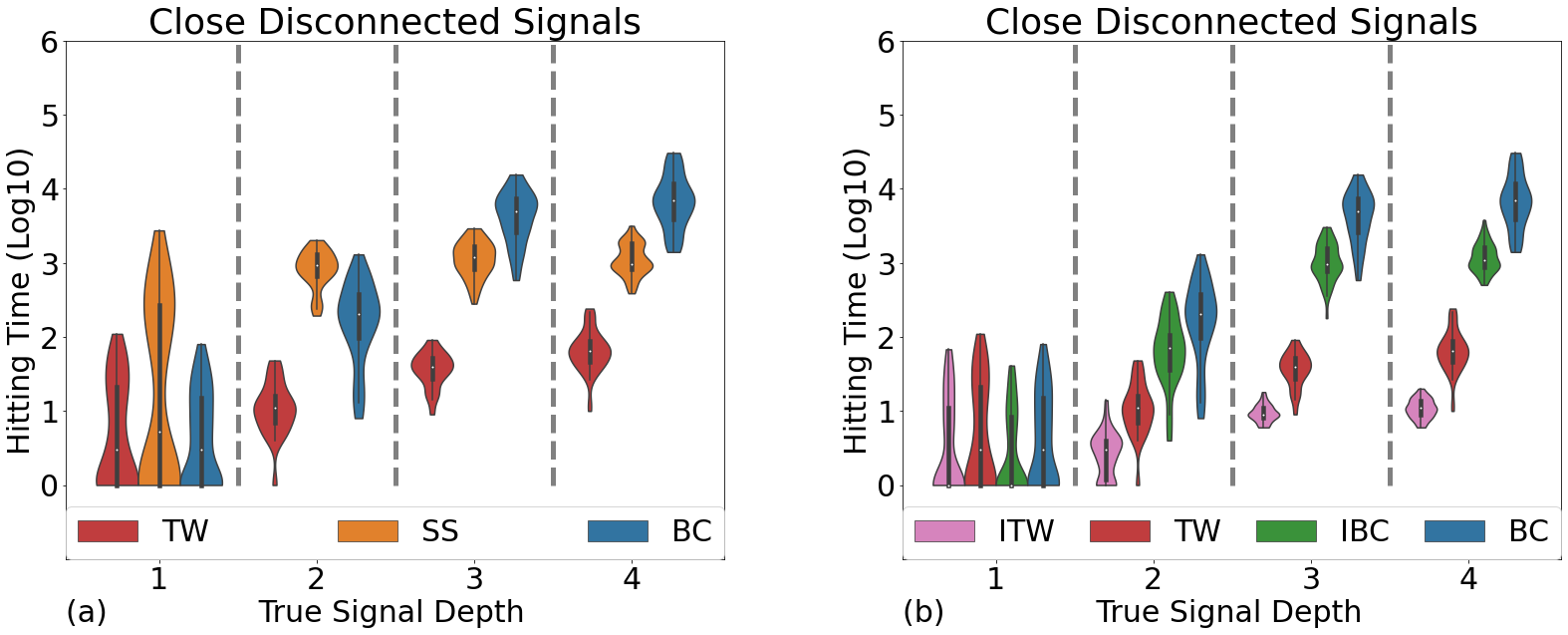}
    \caption{Hitting time $\tau=\min_{t\geq 0}\{ \mathcal B\subset \mT_{int}^t\}$ when true tree gets deeper of Case (3) (by gradually making the deeper part of the tree). {(Legend) BC and IBC: original and informed Bayesian CART, TW and ITW: original and informed Twiggy Bayesian CART, ss: Spike-and-Slab with prior $p_1^{ss}$.} (a) Twiggy Bayesian CART $<$ Bayesian CART. Spike-and-Slab performance is consistent across the true tree depth. (b) Informed (Twiggy) Bayesian CART hits the true signals faster than (Twiggy) Bayesian CART. However, informed Bayesian CART does not hit faster than Twiggy Bayesian CART.}
    \label{fig:informed4}
\end{figure}
\begin{figure}
    \centering
    \includegraphics[width=\linewidth]{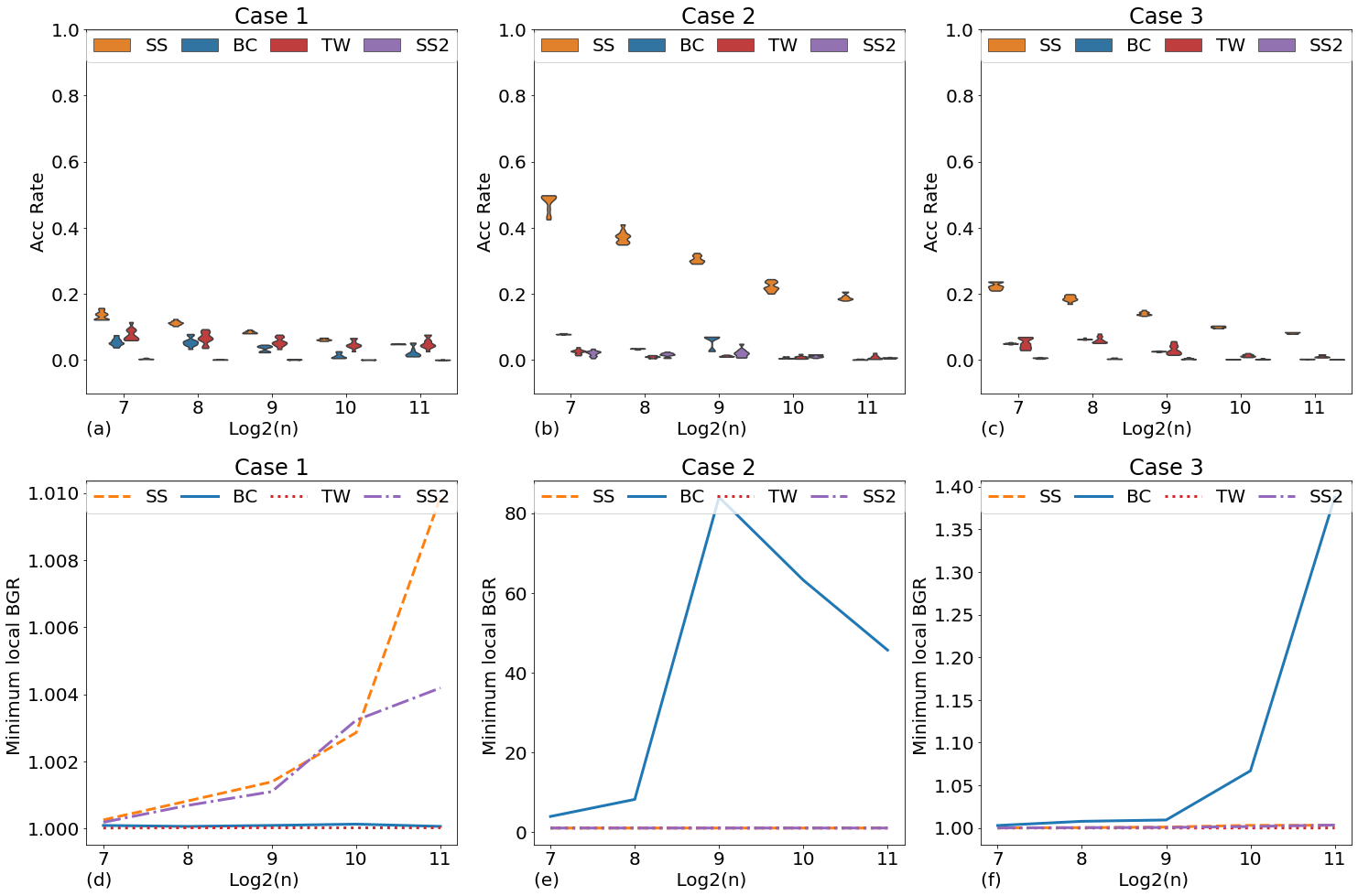}
    \caption{The acceptance rates and minimum local BGRs. {(Legend) BC: Bayesian CART, TW: Twiggy Bayesian CART, SS and SS2: Spike-and-Slab with prior $p_1^{ss}$ and $p_2^{ss}$ respectively.}}
    \label{fig:hit1}
\end{figure}

\begin{figure}
    \centering
    \includegraphics[width=\linewidth]{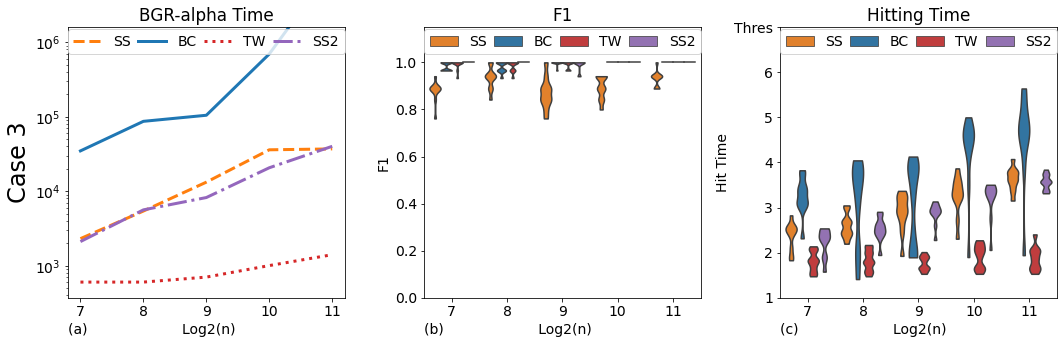}
    \caption{The MCMC performance measures for Case (3). {(Legend) BC: Bayesian CART, TW: Twiggy Bayesian CART, SS and SS2: Spike-and-Slab with prior $p_1^{ss}$ and $p_2^{ss}$ respectively.}}
    \label{fig:hit2}
\end{figure}

\begin{figure}
    \centering
    \includegraphics[width=\linewidth]{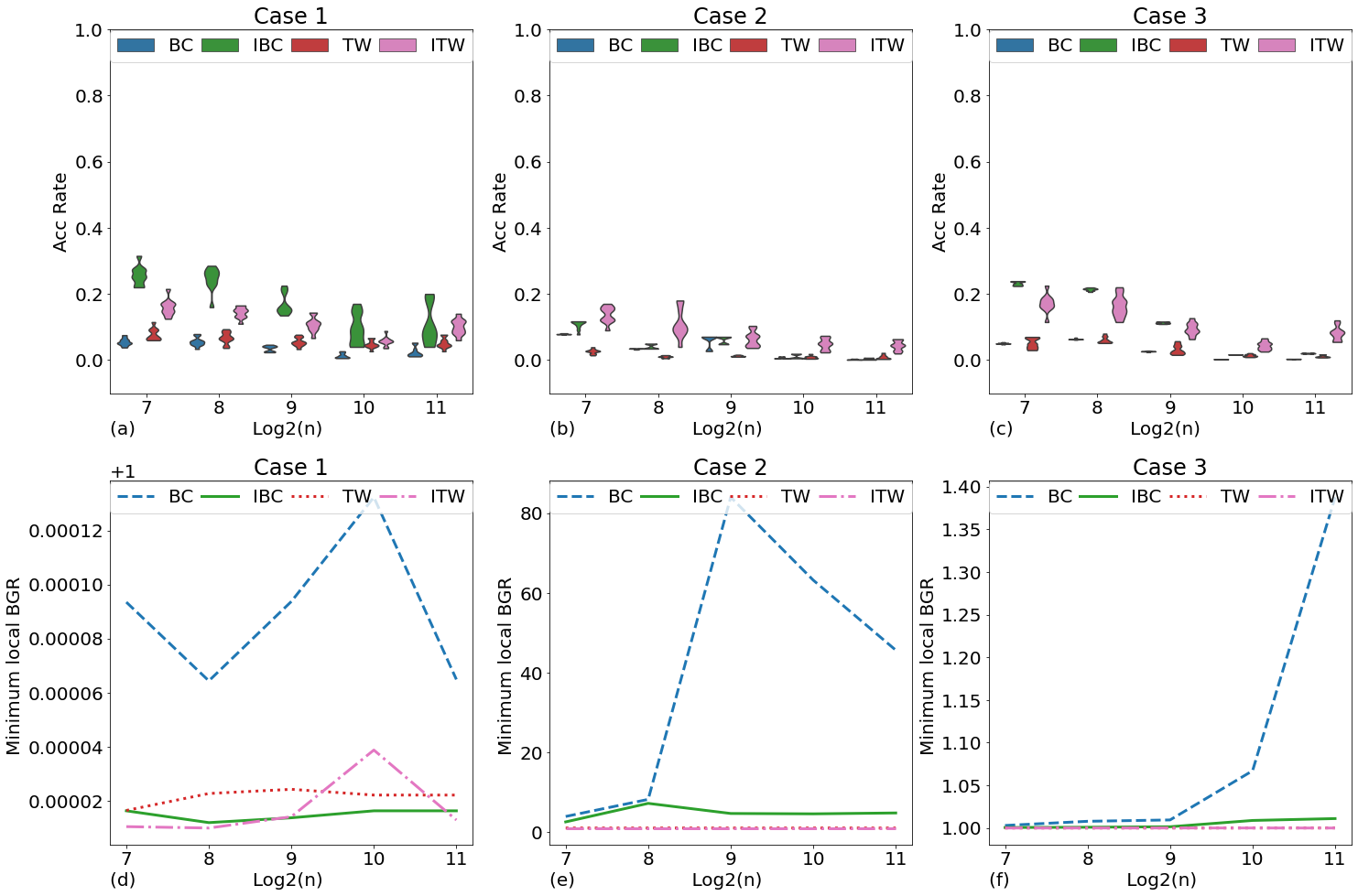}
    \caption{The acceptance rates, hit time and minimum local BGRs for Case (1) and (2). {(Legend) BC and IBC: original and informed Bayesian CART, TW and ITW: original and informed Twiggy Bayesian CART.}}
    \label{fig:tw2}
\end{figure}
\begin{figure}
    \centering
    \includegraphics[width=\linewidth]{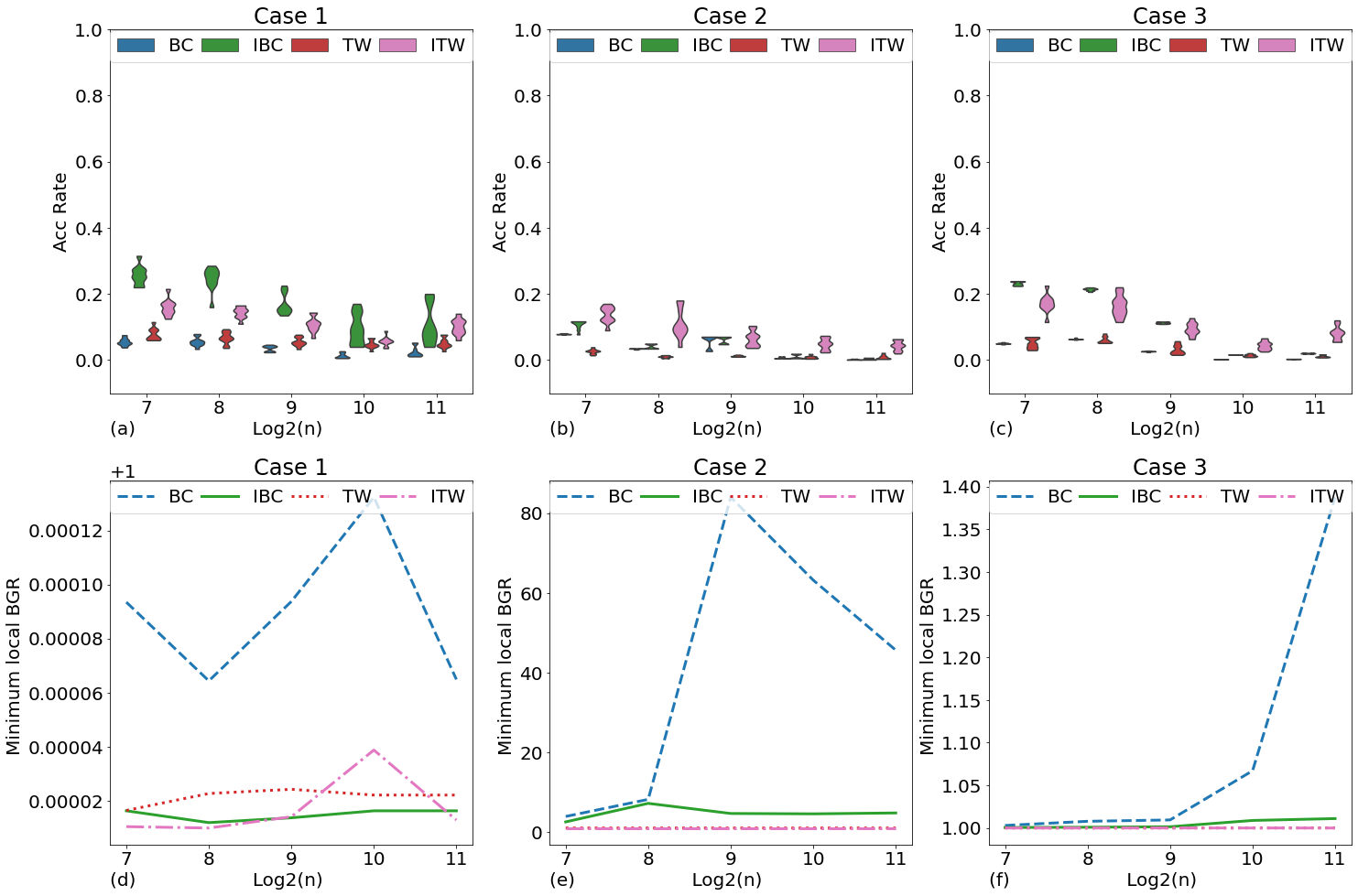}
    \caption{The MCMC performance measures for Case (3). {(Legend) BC and IBC: original and informed Bayesian CART, TW and ITW: original and informed Twiggy Bayesian CART.}}
    \label{fig:tw3}
\end{figure}

\begin{figure}
    \centering
    \includegraphics[width=\linewidth]{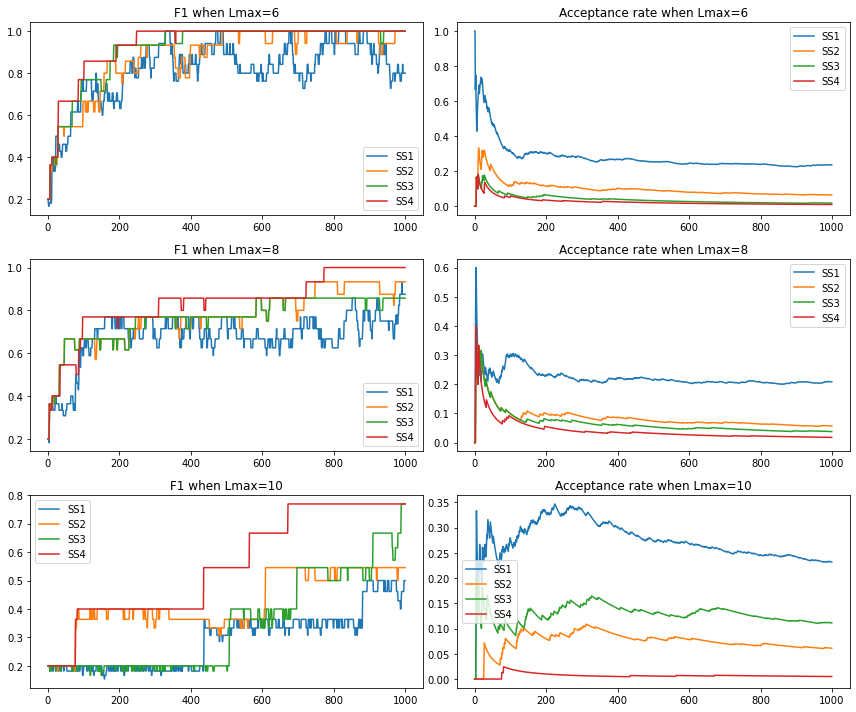}
    \caption{The behavior of Spike-and-Slab for Case (3) for different node inclusion priors for increasing data size ($\Lmax$). The x-axis in all plots are the number of iterations. SS1: $p_{lk} = 0.25/2^{\Lmax-6}$. SS2: $p_{lk} = 0.05/2^{\Lmax-6}$. SS3: $0.01\, n^{1/4} 2^{-l/2}$. SS4: $0.01\, n^{1/4} 6^{-l/2}$. SS4 has the smallest node inclusion prior, and so the smallest acceptance rate. However, in terms of grabbing the true signals without overfitting, SS4 shows the best performance.}
    \label{fig:ss_tune}
\end{figure}

\begin{figure}
    \centering
    \includegraphics[width=\linewidth]{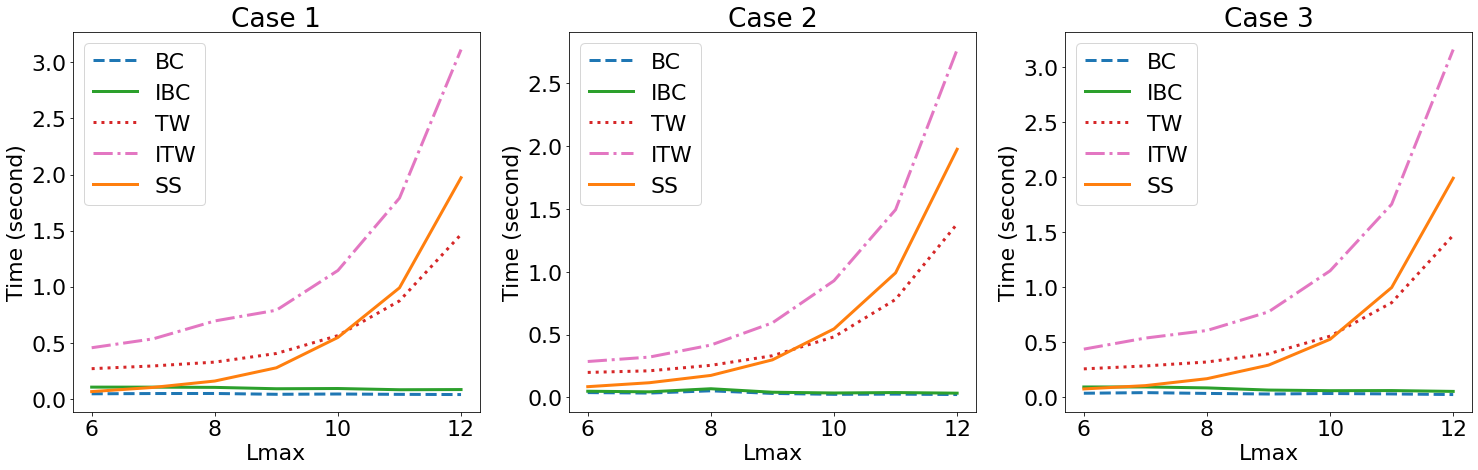}
    \caption{The computational times. Informed Bayesian CART is generally slower than Bayesian CART due to the time calculating the proposal probabilities (e.g., \eqref{eq:wp}).}
    \label{fig:comp}
\end{figure}

\begin{figure}
    \centering
    \includegraphics[width=\linewidth]{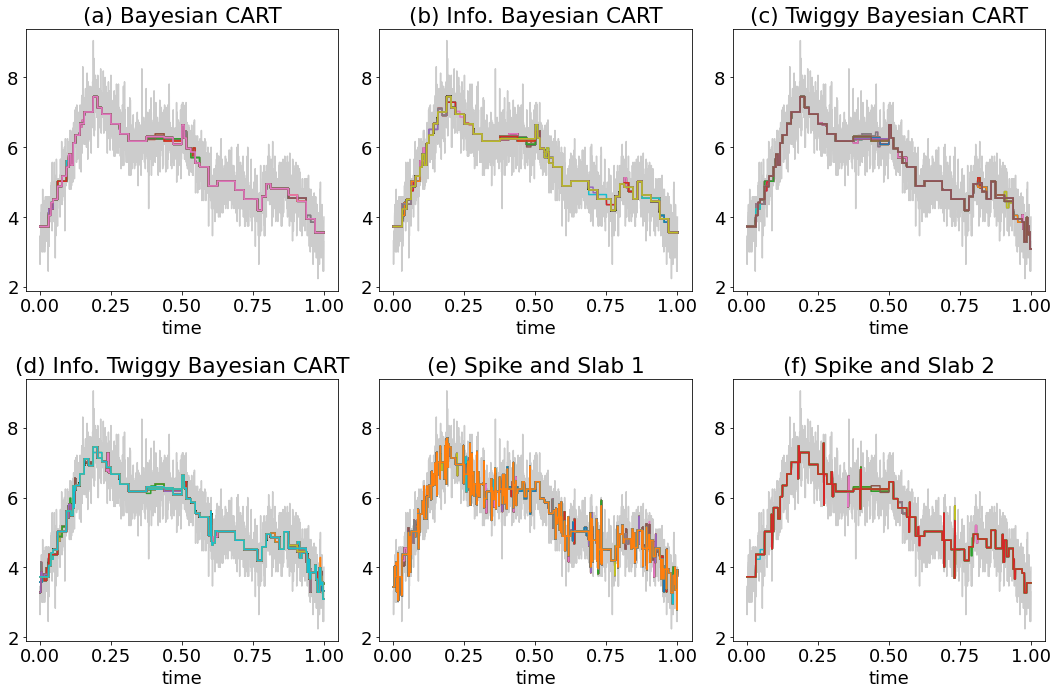}
    \caption{The visualization of 1000 samples after 10,000 burn-in of the MCMC chains on Call Center Data. The gray lines are the data. (a) Bayesian CART (b) informed Bayesian CART (c) Twiggy Bayesian CART (d) informed Twiggy Bayesian CART (e) Spike-and-Slab (prior: $p_{lk}^{ss,1} = 0.01$) (f) Spike-and-Slab (prior: $p_{lk}^{ss,2} = 0.01 \times 6^{-l/2}$)}
    \label{fig:call6}
\end{figure}

\clearpage

\bibliographystyle{abbrvnat}

\bibliography{ref}

\end{document}